\documentclass[11pt, twoside, leqno]{article}

\usepackage{amssymb}
\usepackage{amsmath}
\usepackage{amsthm}
\usepackage{color}
\usepackage{mathrsfs}

\usepackage{indentfirst}

\usepackage{txfonts}

\allowdisplaybreaks

\def\ls{\lesssim}
\def\gs{\gtrsim}

\def\XXint#1#2#3{{\setbox0=\hbox{$#1{#2#3}{\int}$ }
\vcenter{\hbox{$#2#3$ }}\kern-.6\wd0}}

\def\({\left(}
\def \){ \right)}

\newtheorem{theorem}{Theorem}[section]
\newtheorem{lemma}[theorem]{Lemma}

\newtheorem{proposition}[theorem]{Proposition}

\theoremstyle{definition}
\newtheorem{remark}[theorem]{Remark}
\newtheorem{definition}[theorem]{Definition}
\renewcommand{\appendix}{\par
   \setcounter{section}{0}%
   \setcounter{subsection}{0}%
   \setcounter{subsubsection}{0}%
   \gdef\thesection{\@Alph\c@section}%
   \gdef\thesubsection{\@Alph\c@section.\@arabic\c@subsection}%
   \gdef\theHsection{\@Alph\c@section.}%
   \gdef\theHsubsection{\@Alph\c@section.\@arabic\c@subsection}%
   \csname appendixmore\endcsname
 }

\numberwithin{equation}{section}

\begin{document}

\arraycolsep=1pt

\title{\bf\Large  Bilinear Hilbert Transforms and (Sub)Bilinear Maximal Functions along Convex Curves
\footnotetext{\hspace{-0.35cm} 2010 {\it
Mathematics Subject Classification}. Primary 42B20;
Secondary 47B38.
\endgraf {\it Key words and phrases.} Bilinear Hilbert transform, (sub)bilinear maximal function, convex curve, time frequency analysis.
\endgraf Junfeng Li is supported by NSFC-DFG  (\#\, 11761131002), Haixia Yu is supported by "the Fundamental Research Funds for the Central University" (\#\, 20lgpy144).}}
\author{Junfeng Li and Haixia Yu\footnote{Corresponding author.}}
\date{}
\maketitle

\vspace{-0.7cm}

\begin{center}
\begin{minipage}{13cm}
{\small {\bf Abstract}\quad In this paper, we determine the $L^p(\mathbb{R})\times L^q(\mathbb{R})\rightarrow L^r(\mathbb{R})$ boundedness of the bilinear Hilbert transform $H_{\gamma}(f,g)$ along a convex curve $\gamma$
$$H_{\gamma}(f,g)(x):=\mathrm{p.\,v.}\int_{-\infty}^{\infty}f(x-t)g(x-\gamma(t))
\,\frac{\textrm{d}t}{t},$$
where $p$, $q$, and $r$ satisfy $\frac{1}{p}+\frac{1}{q}=\frac{1}{r}$, and $r>\frac{1}{2}$, $p>1$, and $q>1$. Moreover, the same $L^p(\mathbb{R})\times L^q(\mathbb{R})\rightarrow L^r(\mathbb{R})$ boundedness property holds for the corresponding (sub)bilinear maximal function $M_{\gamma}(f,g)$ along a convex curve $\gamma$
$$M_{\gamma}(f,g)(x):=\sup_{\varepsilon>0}\frac{1}{2\varepsilon}\int_{-\varepsilon}^{\varepsilon}|f(x-t)g(x-\gamma(t))|
\,\textrm{d}t.$$
}
\end{minipage}
\end{center}

\tableofcontents

\section{Introduction}

\subsection{Main problem and main result}

The \emph{bilinear Hilbert transform $H_{\gamma}(f,g)$} along a curve $\gamma$ is defined as
$$H_{\gamma}(f,g)(x):=\mathrm{p.\,v.}\int_{-\infty}^{\infty}f(x-t)g(x-\gamma(t))\,\frac{\textrm{d}t}{t}$$
for $f$ and $g$ in the Schwartz class $\mathcal{S}(\mathbb{R})$. The corresponding \emph{(sub)bilinear maximal function $M_{\gamma}(f,g)$} is defined as
$$M_{\gamma}(f,g)(x):=\sup_{\varepsilon>0}\frac{1}{2\varepsilon}\int_{-\varepsilon}^{\varepsilon}|f(x-t)g(x-\gamma(t))|\,\textrm{d}t.$$
The $L^p(\mathbb{R})\times L^q(\mathbb{R})\rightarrow L^r(\mathbb{R})$ boundedness property for these two operators with some general curves $\gamma$ are of great interest to us. We start with a special case $\gamma:=\textrm{P}$, a polynomial of degree $d$ with no linear term and constant term, where $d\in \mathbb{N}$ and $d>1$. In \cite{LX}, Li and Xiao set up the $L^p(\mathbb{R})\times L^q(\mathbb{R})\rightarrow L^r(\mathbb{R})$ boundedness for $H_{\gamma}(f,g)$ and $M_{\gamma}(f,g)$ where $p$, $q$, and $r$ satisfy $\frac{1}{p}+\frac{1}{q}=\frac{1}{r}$, and $r>\frac{d-1}{d}$, $p>1$, and $q>1$ .  Moreover, they showed that $r>\frac{d-1}{d}$ is sharp up to the end point. By replacing $\gamma$ with a homogeneous curve $t^d$ with $d\in \mathbb{N}$, $d>1$, the range of $r$ was extended by Li and Xiao to $r>\frac{1}{2}$. Furthermore, they believe that with some special conditions on $\gamma$, the full range boundedness of $H_{\gamma}(f,g)$ and $M_{\gamma}(f,g)$ can be obtained. Here and hereafter, we omit the relationship that $p$, $q$, and $r$ satisfy $\frac{1}{p}+\frac{1}{q}=\frac{1}{r}$, $p>1$, and $q>1$ and the fact that $d>1$ and $d\in\mathbb{N}$. We call the full range bounded if the range of $r$ is $(\frac{1}{2},\infty)$. In this paper, we provide some sufficient conditions of $\gamma$ regarding this concern.

\begin{theorem}\label{theorem 1.1}
Let $\gamma\in C^{3}(\mathbb{R})$ be either odd or even, with $\gamma(0)=\gamma'(0)=0$, $\lim_{t\rightarrow 0^+}\gamma'(t)/\gamma''(t)=0$, and convex on $(0,\infty)$. Furthermore, if
\begin{align}\label{eq:1.0}
there~ exist~ positive ~constants~ C_1~ and~ C_2 ~such~ that~ C_1\leq (\frac{\gamma'}{\gamma''})'(t)\leq C_2~ on~ (0,\infty).
\end{align}
Then, there exists a positive constant $C$ such that
$$\left\|H_{\gamma}(f,g)\right\|_{L^{r}(\mathbb{R})}\leq C \|f\|_{L^{p}(\mathbb{R})}\|g\|_{L^{q}(\mathbb{R})}$$
for any $f\in L^{p}(\mathbb{R})$ and $g\in L^{q}(\mathbb{R})$, where $p,q,r$ satisfy $\frac{1}{p}+\frac{1}{q}=\frac{1}{r}$, and $r>\frac{1}{2}$, $p>1$, $q>1$.
\end{theorem}

\begin{theorem}\label{theorem 1.2}
Under the same conditions of $\gamma$, we have
$$\left\|M_{\gamma}(f,g)\right\|_{L^{r}(\mathbb{R})}\leq C \|f\|_{L^{p}(\mathbb{R})}\|g\|_{L^{q}(\mathbb{R})}.$$
\end{theorem}

\begin{remark}\label{remark 1.0}
It is easy to see that $H_{\gamma}(f,g)$ does not map $L^\infty(\mathbb{R})\times L^\infty(\mathbb{R})$ into $ L^\infty(\mathbb{R})$. Moreover, it is trivial that $M_{\gamma}(f,g)$ is bounded from $L^\infty(\mathbb{R})\times L^\infty(\mathbb{R})$ into $ L^\infty(\mathbb{R})$. Therefore, we can restrict the range of $r$ as $(\frac{1}{2},\infty)$ in the rest of the paper.
\end{remark}

\begin{remark}\label{remark 1.00}
The following are some curves satisfying the conditions of Theorem \ref{theorem 1.1}; we here write only the part for $t\in [0,\infty)$ based on its odd or even property:
\begin{enumerate}
\item[\rm(i)] for any $t\in [0,\infty)$, $\gamma_1(t):=t^\alpha$ under $\alpha\in(1,\infty)$;
\item[\rm(ii)] for any $t\in [0,\infty)$, $\gamma_2(t):=t^\alpha\log(1+t)$ under $\alpha\in(1,\infty)$;
\item[\rm(iii)] for any $t\in [0,\infty)$ and $K\in \mathbb{N}$, $\gamma_3(t):=\sum\limits_{i=1}^{K} t^{\alpha_i}$ under $\alpha_i\in(1,\infty)$ for all $i=1,2, \cdots, K$.
\end{enumerate}
\end{remark}

\begin{remark}\label{remark 1.000}
For a more general curve $\gamma$, Lie \cite{Lie1} introduced a set $\mathcal{NF}^C$ and obtained the $L^2(\mathbb{R})\times L^2(\mathbb{R})\rightarrow L^1(\mathbb{R})$ boundedness of $H_{\gamma}(f,g)$ for $\gamma\in \mathcal{NF}^C$. Later, it was extended to the $L^p(\mathbb{R})\times L^q(\mathbb{R})\rightarrow L^r(\mathbb{R})$ boundedness with $r\geq 1$ in \cite{Lie2}. Furthermore, Gaitan and Lie \cite{GaL} obtained the same boundedness for $M_{\gamma}(f,g)$. It is worth noting that these results are sharp in the sense that we cannot take $\frac{1}{2}<r<1$, since the polynomial $\textrm{P}$ stated in \cite{LX} belongs to the set $\mathcal{NF}^C$. More recently, Guo and Xiao \cite{GX} obtained the $L^2(\mathbb{R})\times L^2(\mathbb{R})\rightarrow L^1(\mathbb{R})$ boundedness of $H_{\gamma}(f,g)$ and $M_{\gamma}(f,g)$, where $\gamma\in \textbf{F}(-1,1)$; the definition of the set $\textbf{F}(-1,1)$ can be found on P. $970$ in \cite{GX}.

The argument of this paper is based on the works of Guo and Xiao \cite{GX}, Li \cite{L2}, Li and Xiao \cite{LX} and Lie \cite{Lie1, Lie2}, but we also make several contributions:

\begin{itemize}

\item[$\star$] Our conditions may be easier to check than $\mathcal{NF}^C$ in Lie \cite{Lie1, Lie2} and $\textbf{F}(-1,1)$ in Guo and Xiao \cite{GX}. Moreover, we require less regularity on $\gamma$ than that in \cite{Lie1, Lie2, GX}. On the other hand, we only require that the curve $\gamma$ belongs to $ C^{3}$, but $\mathcal{NF}^C\subset C^{4}$ and $\textbf{F}(-1,1)\subset C^{5}$.

\item[$\star$] We obtain the full range boundedness for $H_{\gamma}(f,g)$ and $M_{\gamma}(f,g)$. Therefore, our results extend the results of Li \cite{L2} and Li and Xiao \cite[Theorem 3]{LX}, which concern the homogeneous curve $\gamma(t):=t^d$, to more general classes of curves.

\item[$\star$] The main difference between this paper and the abovementioned works is a partition of unity. We split our multiplier by the following partition of unity; i.e.,
    \begin{align*}
\sum_{m,n,k\in \mathbb{Z}}\phi\left(\frac{\xi}{2^{m+j}}\right)\phi\left(\frac{\eta}{2^{k}}\right)\phi\left(\frac{\gamma'(2^{-j})}{2^{n+j-k}}\right)=1,
\end{align*}
    see \eqref{eq:dwfj}, instead of
    \begin{align*}
\sum_{m,n\in \mathbb{Z}}\phi\left(\frac{\xi}{2^{m+j}}\right)\phi\left(\frac{\eta\gamma'(2^{-j})}{2^{n+j}}\right)=1,
\end{align*}
where $\phi$ is a standard bump function supported on $\{t\in \mathbb{R}:\ \frac{1}{2}\leq |t|\leq 2\}$ such that $0\leq\phi(t)\leq1$ and $\Sigma_{l\in \mathbb{Z}} \phi (2^{-l}t)=1$ for all $t\neq 0$. The aim of this partition of unity \eqref{eq:dwfj} is to avoid using uniform paraproduct estimates at the low-frequency part. In \cite{L2}, Li used a uniform paraproduct estimate, i.e., \cite[Theorem 4.1]{L2}, to bound the low-frequency part. In this paper, we present an easy way to dispose of the low-frequency part by using this partition of unity and the Littlewood-Paley theory together with the uniform estimates \eqref{eq:3.7} and \eqref{eq:3.6}.

\end{itemize}
\end{remark}

\subsection{Background and motivation}

There are rich backgrounds from which to study the boundedness property of $H_{\gamma}(f,g)$ and $M_{\gamma}(f,g)$.

\begin{itemize}

\item[$\diamondsuit$] If we take $\gamma(t):=t$, the boundedness of these two operators is trivial. This follows from the boundedness of the classical Hilbert transform, the Hardy-Littlewood maximal function and the H\"older inequality.

\item[$\diamondsuit$] If $\gamma(t):=-t$, these operators turn out to be the standard bilinear Hilbert transform and the corresponding (sub)bilinear maximal function whose boundedness is not easy to obtain. Lacey and Thiele \cite{LT1,LT2} obtained the boundedness with $r>\frac{2}{3}$ for the standard bilinear Hilbert transform. For the same boundedness of the corresponding maximal function, we refer to Lacey \cite{L}. In the same paper, a counterexample showed that if $\frac{1}{2}<r<\frac{2}{3}$ the boundedness fails for these operators.

\item[$\diamondsuit$] If we take $\gamma(t):=t^d$ or $\gamma:=\textrm{P}$ or a more general curve $\gamma$, the boundedness of these two operators has been stated at the beginning of this paper and in Remark \ref{remark 1.000}, where $P$ is a polynomial of degree $d$ with no linear term and constant term.

\item[$\diamondsuit$] There are some other types of bilinear Hilbert transforms. Let
$$H_{\alpha,\beta}(f,g)(x):=\mathrm{p.\,v.}\int_{-\infty}^{\infty}f(x-\alpha t)g(x-\beta t)\,\frac{\textrm{d}t}{t}.$$
Grafakos and Li \cite{GL} set up the uniform boundedness for $r>1$. Later, Li \cite{L1} extended the index to $r>\frac{2}{3}$. Recently, Dong \cite{D1} considered the \emph{bilinear Hilbert transform $H_{\textrm{P},\textrm{Q}}(f,g)$} along two polynomials
$$H_{\textrm{P},\textrm{Q}}(f,g)(x):=\mathrm{p.\,v.}\int_{-\infty}^{\infty}f\left(x-\textrm{P}(t)\right)g\left(x-\textrm{Q}(t)\right)\,\frac{\textrm{d}t}{t}$$
and the corresponding \emph{maximal operator $M_{\textrm{P},\textrm{Q}}(f,g)$}
$$M_{\textrm{P},\textrm{Q}}(f,g)(x):=\sup_{\varepsilon>0}\frac{1}{2\varepsilon}\int_{-\varepsilon}^{\varepsilon}|f\left(x-\textrm{P}(t)\right)g\left(x-\textrm{Q}(t)\right)|
\,\textrm{d}t.$$
Here, $\textrm{P}$ and $\textrm{Q}$ are polynomials with no constant term. Dong proved that these operators are bounded for $r>\frac{d}{d+1}$, where $d$ is the correlation degree of these two polynomials $\textrm{P}$ and $\textrm{Q}$. For the definition of the correlation degree, we refer the reader to P. $2$ in \cite{D1}. There are many other related works; see, for example, \cite{DT,D2,DL,DM,FL,LL}.

\end{itemize}

The study of the boundedness of $H_{\gamma}(f,g)$ and $M_{\gamma}(f,g)$ originated from $\textrm{Calder}\acute{\textrm{o}}\textrm{n}$ \cite{C} in order to study the Cauchy transform along Lipschitz curves, but there have also been many other motivations:

\begin{itemize}

\item[$\rhd$] One of the motivations arises from ergodic theory. For instance, for $n\in \mathbb{N}$, the $L^r(\mathbb{R})$-norm convergence property of the non-conventional bilinear averages
$$\frac{1}{N}\sum_{n=1}^Nf\left(T^n\right)g\left(T^{n^2}\right)$$
as $N$ tends to $\infty$. Here, $T$ is an invertible and measure-preserving transformation of a finite measure space. For more details, we refer to \cite{De,F,HK}.

\item[$\rhd$] Another motivation is offered by number theory. There are many various nonlinear extensions of Roth's theorem for some sets with positive density; see, for example, \cite{B2,B1,DGR}.

\item[$\rhd$] There have also been many developments during the last few years regarding the \emph{Hilbert transform $H_{\gamma}f$} along the curve $\gamma$ defined as
$$H_{\gamma}f(x_1,x_2):=\mathrm{p.\,v.}\int_{-\infty}^{\infty}f(x_1-t,x_2-\gamma(t))
\,\frac{\textrm{d}t}{t},$$
and the corresponding \emph{maximal function $M_{\gamma}f$} along the curve $\gamma$ defined as
$$M_{\gamma}f(x):=\sup_{\varepsilon>0}\frac{1}{2\varepsilon}\int_{-\varepsilon}^{\varepsilon}|f(x_1-t,x_2-\gamma(t))|
\,\textrm{d}t.$$
These operators were initiated by Fabes and $\textrm{Rivi}\grave{\textrm{e}}\textrm{re}$ \cite{FR} and Jones \cite{J} in order to understand the behavior of the constant-coefficient parabolic differential operators. Later, $H_{\gamma}f$ and $M_{\gamma}f$ were extended to cover more general classes of curves \cite{CCVWW,CVWW,CCCD,CNSW,NVWW,VWW}. $H_{\gamma}(f,g)$ is closely associated with $H_{\gamma}f$ since they have the same multiplier. Indeed, we can rewrite $H_{\gamma}(f,g)(x)$ as
$$\int_{-\infty}^{\infty}\int_{-\infty}^{\infty}\hat{f}(\xi)\hat{g}(\eta)e^{i\xi x}e^{i\eta x} \left(\mathrm{p.\,v.}\int_{-\infty}^{\infty}e^{-i\xi t}e^{-i\eta \gamma(t)}
\,\frac{\textrm{d}t}{t}\right)\,\textrm{d}\xi \,\textrm{d}\eta$$
and $H_{\gamma}f(x_1,x_2)$ can be rewritten as
$$\int_{-\infty}^{\infty}\int_{-\infty}^{\infty}\hat{f}(\xi,\eta)e^{i\xi x_1}e^{i\eta x_2} \left(\mathrm{p.\,v.}\int_{-\infty}^{\infty}e^{-i\xi t}e^{-i\eta \gamma(t)}
\,\frac{\textrm{d}t}{t}\right)\,\textrm{d}\xi \,\textrm{d}\eta.$$
Therefore, we can find many similarities between the approaches of $H_{\gamma}(f,g)$ and $H_{\gamma}f$.

\end{itemize}

\subsection{Organization and notations}

We now present the structure of the rest of this paper.

\begin{itemize}

\item[$\bullet$] In Section \ref{section 2}, we give some preliminaries for our proof. Subsection \ref{subsection 2.1} provides two inequalities about the curve $\gamma$, which will be used repeatedly in our proof. Subsection \ref{subsection 2.2} is devoted to splitting $H_{\gamma}(f,g)$ into the following three parts: the low-frequency part $H_{\gamma}^1(f,g)$; the high-frequency part away from the diagonal part $H_{\gamma}^2(f,g)$ and the high-frequency part near the diagonal part $H_{\gamma}^3(f,g)$; see \eqref{eq:2.14} and \eqref{eq:2.15}.
 At the same time, we split $\frac{1}{t}=\sum_{j\in \mathbb{Z}}2^j\rho(2^jt)$ and set up a uniform boundedness for each $H_{\gamma,j}(f,g)$, where the corresponding kernel is $2^j\rho(2^jt)$. A similar decomposition of $M_{\gamma}(f,g)$ can be found in Subsection \ref{subsection 2.3}: the low-frequency part $M_{\gamma}^1(f,g)$, the high-frequency part far from the diagonal part $M_{\gamma}^2(f,g)$ and the high-frequency part close to the diagonal part $M_{\gamma}^3(f,g)$; see \eqref{eq:2.26}.

\item[$\bullet$] In Section \ref{section 3}, we obtain the full range boundedness for $H_{\gamma}^1(f,g)$ by using Taylor series expansion.

\item[$\bullet$] In Section \ref{section 4}, we establish the full range boundedness for $H_{\gamma}^2(f,g)$. To this aim, we consider two cases according to the function that has the higher frequency.

\item[$\bullet$] In Section \ref{section 5}, we prove the $L^2(\mathbb{R})\times L^2(\mathbb{R})\rightarrow L^1(\mathbb{R})$ boundedness of $H_{\gamma}^3(f,g)$. To this aim, we obtain a $2^{-\varepsilon_0 m}$ decay for $H_{m}(f,g)$ (see \eqref{eq:6.2}) defined on the frequency piece along the diagonal for some positive constants $\varepsilon_0$. Here, we used the $TT^*$ argument, H\"ormander's theorem \cite[Theorem 1.1]{H}, the stationary phase method and $\sigma$-uniformity.

\item[$\bullet$] In Section \ref{section 6}, we obtain the full range weak-$L^p(\mathbb{R})\times L^q(\mathbb{R})\rightarrow L^r(\mathbb{R})$ boundedness for $H_{m}(f,g)$ with a bound $m$. By interpolation with the $L^2(\mathbb{R})\times L^2(\mathbb{R})\rightarrow L^1(\mathbb{R})$ estimate that has a decay bound of $2^{-\varepsilon_0 m}$, we can finish our proof. To obtain the weak boundedness, we need to split $H_{m}(f,g)$ into the following three error terms: $|H^1_m|(f,g)$, $|H^2_m|(f,g)$ and $|H^3_m|(f,g)$, see \eqref{eq:6.9}, and the major term $|H^4_m|(f,g)$ in \eqref{eq:6.12}. Subsection \ref{subsection 6.1} describes the estimation for these three error terms, and Subsection \ref{subsection 6.2} is devoted to establishing the major term $|H^4_m|(f,g)$ by using the method of time frequency analysis, which is the most difficult part.

\item[$\bullet$] In Section \ref{section 7}, we set up the full range boundedness of $M_{\gamma}(f,g)$.

\end{itemize}

Throughout this paper, we denote by $C$ a \emph{positive
constant} that is independent of the main parameters involved, whose exact
value is allowed to change from line to line. The \emph{positive constants with subscripts},
$C_1$ and $ C_2$, are fixed constants. The symbol $a\ls b$ or $b\gs a$ means that there exists a positive constant $C$ such that $a\le Cb$. $a\approx b$ means $a\ls b$ and $b\ls a$. We use $\mathcal{S}(\mathbb{R})$ to denote the \emph{Schwartz class} on $\mathbb{R}$. Let $\mathbb{Z}_-:=\mathbb{Z}\backslash \mathbb{N}$ with $\mathbb{N}:=\{0,1,2,\cdots\}$. $\hat{f}$ denotes the Fourier transform of $f$, and $\check{f}$ is the inverse Fourier transform of $f$. For any $0<p<\infty$, we denote $p'$ as its
\emph{conjugate index} if $\frac{1}{p}+\frac{1}{p'}=1$. It is obvious that $p'<0$ if $0<p<1$. For any set $E$, we use $\chi_E$ to denote its \emph{characteristic function}. $\sharp E$ denotes the \emph{cardinality} of it. $E^{\complement}$ indicates its \emph{complementary set}.

\bigskip

\noindent{\it \textbf{Acknowledgments}}. The second author would like to thank his postdoctoral advisor Prof. Lixin Yan for the many valuable comments and helpful discussions.

\section{Preliminaries}\label{section 2}

\subsection{The curve $\gamma$}\label{subsection 2.1}

We first explain some simple properties of the curve $\gamma$ in Theorem \ref{theorem 1.1} which will be used in this paper. Since $\gamma\in C^{3}(\mathbb{R})$, $\gamma(0)=\gamma'(0)=0$ and $\gamma$ is convex on $(0,\infty)$, we have $\gamma''(t)\geq 0$ on $(0,\infty)$. By \eqref{eq:1.0}, we obtain $\gamma''(t)\neq 0$ on $(0,\infty)$ and thus $\gamma''(t)> 0$ on $(0,\infty)$. Therefore, $\gamma'$ is strictly increasing and $\gamma'(t)>0$ on $(0,\infty)$. On the other hand, let us set $G_1(t):=2C_2t-\gamma'(t)/\gamma''(t)$ and $G_2(t):=\gamma'(t)/\gamma''(t)-C_1t/2$, by \eqref{eq:1.0}, we then have $G'_1(t)\geq C_2$ and $G'_2(t)\geq C_1/2$ on $(0,\infty)$. This, combined with $\lim_{t\rightarrow 0^+}\gamma'(t)/\gamma''(t)=0$, leads to $G_1(t)\geq 0$ and $G_2(t)\geq 0$ on $(0,\infty)$ and therefore
\begin{align}\label{eq:1.00}
\frac{1}{2C_2}\leq\frac{t\gamma''(t)}{\gamma'(t)}\leq \frac{2}{C_1}, \quad  {\rm for \ any \ }t\in (0,\infty).
\end{align}
Since $\gamma(0)=\gamma'(0)=0$, by the Cauchy mean value theorem, for any $t\in(0,\infty)$, there exists $\tau_1\in (0,t)$ such that
$\frac{t\gamma'(t)}{\gamma(t)}=\frac{t\gamma'(t)-0\gamma'(0)}{\gamma(t)-\gamma(0)}=\frac{\gamma'(\tau_1)+\tau_2\gamma''(\tau_1)}{\gamma'(\tau_2)}$.
Thus, by \eqref{eq:1.00},
\begin{align}\label{eq:1.01}
1+\frac{1}{2C_2}\leq \frac{t\gamma'(t)}{\gamma(t)}\leq 1+\frac{2}{C_1}, \quad  {\rm for \ any \ }t\in (0,\infty),
\end{align}
which further implies that $\gamma(t)/t$ is strictly increasing on $(0,\infty)$.

Let $G_3(t):=\log_2 \gamma'(t)$ for any $t\in (0,\infty)$, by \eqref{eq:1.00}, we then have $1/2C_2t\leq G'_3(t)\leq 2/C_1t$ for any $t\in (0,\infty)$. By the Lagrange mean value theorem, there exists a constant $\theta\in[1,2]$ such that $G_3(2t)-G_3(t)=G'_3(\theta t)t\in[1/4C_2, 2/C_1]$, which further leads to
\begin{align}\label{eq:1.02}
2^{\frac{1}{4C_2}}\leq \frac{\gamma'(2t)}{\gamma'(t)}\leq 2^{\frac{2}{C_1}}, \quad  {\rm for \ any \ }t\in (0,\infty).
\end{align}
By the Cauchy mean value theorem and  $\gamma(0)=0$, for any $t\in(0,\infty)$, there exists $\tau_2\in (0,t)$ such that $\frac{\gamma(2t)}{\gamma(t)}=\frac{\gamma(2t)-\gamma(0)}{\gamma(t)-\gamma(0)}=\frac{2\gamma'(2\tau_2)}{\gamma'(\tau_2)}$. This, combined with \eqref{eq:1.02}, implies
\begin{align}\label{eq:1.03}
2^{1+\frac{1}{4C_2}}\leq \frac{\gamma(2t)}{\gamma(t)}\leq 2^{1+\frac{2}{C_1}}, \quad  {\rm for \ any \ }t\in (0,\infty).
\end{align}

\subsection{Decomposition of $H_\gamma(f,g)$}\label{subsection 2.2}

For $H_\gamma(f,g)$, we rewrite it as
\begin{align}\label{eq:2.4}
H_{\gamma}(f,g)(x):=\int_{-\infty}^{\infty}\int_{-\infty}^{\infty}\hat{f}(\xi)\hat{g}(\eta)e^{i\xi x}e^{i\eta x} m(\xi,\eta)\,\textrm{d}\xi \,\textrm{d}\eta,
\end{align}
where
\begin{align}\label{eq:2.5}
m(\xi,\eta):=\mathrm{p.\,v.}\int_{-\infty}^{\infty}e^{-i\xi t}e^{-i\eta \gamma(t)}
\,\frac{\textrm{d}t}{t}.
\end{align}
Let $\rho$ be an odd smooth function supported on $\{t\in \mathbb{R}:\ \frac{1}{2}\leq |t|\leq 2\}$ such that
$\frac{1}{t}=\sum_{j\in \mathbb{Z}}2^j\rho(2^jt)$. Then,
\begin{align}\label{eq:2.6}
m(\xi,\eta)=\sum_{j\in \mathbb{Z}}m_j(\xi,\eta),
\end{align}
where
\begin{align}\label{eq:2.7}
m_j(\xi,\eta):=\int_{-\infty}^{\infty}e^{-i2^{-j}\xi t}e^{-i\eta \gamma(2^{-j}t)}\rho(t)
\,\textrm{d}t.
\end{align}
Therefore, we split $H_{\gamma}(f,g)$ as
\begin{align*}
H_{\gamma}(f,g)(x)=\sum_{j\in \mathbb{Z}}\int_{-\infty}^{\infty}\int_{-\infty}^{\infty}\hat{f}(\xi)\hat{g}(\eta)e^{i\xi x}e^{i\eta x} m_j(\xi,\eta)\,\textrm{d}\xi \,\textrm{d}\eta
=:\sum_{j\in \mathbb{Z}}H_{\gamma,j}(f,g)(x).
\end{align*}

From Proposition \ref{proposition 2.1} below, we need only to consider that $|j|$ large enough. Before giving Proposition \ref{proposition 2.1}, we first state the following lemma which can be found in \cite[Lemma 4.7]{GPRY}.

\begin{lemma}\label{lemma 2.0}
Let $I\subset \mathbb{R}$ be an interval, $k\in \mathbb{N}$, $f\in C^k(I)$, and suppose that for some $\sigma>0$, $|f^{(k)}(x)|\geq \sigma$ for all $x\in I$. Then there exists a positive constant $C$ depending only on $k$ such that
$$\left|\{x\in I:\ |f(x)|\leq \rho\} \right|\leq C \left(\frac{\rho}{\sigma}\right)^{\frac{1}{k}}$$
for all $\rho>0$.
\end{lemma}

\begin{proposition}\label{proposition 2.1}
Let $\gamma$ and $p,q,r$ be the same as in Theorem \ref{theorem 1.1}. Then there exists a positive constant $C$ independent of $j$ such that
$$\left\|H_{\gamma,j}(f,g)\right\|_{L^{r}(\mathbb{R})}\leq C \|f\|_{L^{p}(\mathbb{R})}\|g\|_{L^{q}(\mathbb{R})}$$
for all $f\in L^{p}(\mathbb{R})$ and $g\in L^{q}(\mathbb{R})$.
\end{proposition}

\begin{proof}
For $r\geq 1$, by the Minkowski inequality and $\mathrm{H}\ddot{\mathrm{o}}\mathrm{lder}$ inequality, we have
\begin{align*}
\left\|H_{\gamma,j}(f,g)\right\|_{L^{r}(\mathbb{R})}
\leq\int_{-\infty}^{\infty} \left\||f(\cdot-t)|^r\right\|_{L^{\frac{p}{r}}(\mathbb{R})} ^{\frac{1}{r}} \left\||g(\cdot-\gamma(t))|^r\right\|_{L^{\frac{q}{r}}(\mathbb{R})} ^{\frac{1}{r}}     2^j|\rho(2^jt)|\,\textrm{d}t\ls\|f\|_{L^{p}(\mathbb{R})}\|g\|_{L^{q}(\mathbb{R})}.
\end{align*}
For $\frac{1}{2}<r<1$, we rewrite $H_{\gamma,j}(f,g)$ as
\begin{align*}
H_{\gamma,j}(f,g)(x)=\int_{-\infty}^{\infty}f(x-2^{-j}t)g(x-\gamma(2^{-j}t))\rho(t)\,\textrm{d}t.
\end{align*}
Let
$$
\begin{cases}A_1:=\{2^{-j}t:\ t\in\textrm{supp}~\rho\};\\
A_\gamma:=\{\gamma(2^{-j}t):\ t\in\textrm{supp}~\rho\};\\
E_\lambda:=\{t\in\textrm{supp}~\rho:\ 2^{\lambda}<2^{-j}|\gamma'(2^{-j}t)|\leq 2\cdot2^{\lambda}\},
\end{cases}
$$
where $\lambda\in \mathbb{Z}$. Furthermore, let $A_1(\lambda):=\{2^{-j}t:\ t\in E_\lambda\}$ and $A_\gamma(\lambda):=\{\gamma(2^{-j}t):\ t\in E_\lambda\}$; we have
\begin{align}\label{eq:2.1.1}
H_{\gamma,j}(f,g)(x)=\sum_{\lambda\in \mathbb{Z}}\int_{E_\lambda}f(x-2^{-j}t)g(x-\gamma(2^{-j}t))\rho(t)\,\textrm{d}t=:\sum_{\lambda\in \mathbb{Z}}H^{\lambda}_{\gamma,j}(f,g)(x).
\end{align}

To estimate $H_{\gamma,j}(f,g)$, we consider the following three cases:
$$2^{\lambda}\geq 2\cdot2^{-j},\quad 2^{\lambda}\leq \frac{1}{4}2^{-j}\quad \text{and}
\quad \frac{1}{2}2^{-j}\leq2^{\lambda}\leq 2^{-j}. $$

{\bf Case~I:} $2^{\lambda}\geq 2\cdot2^{-j}$

We observe that $$|A_1(\lambda)|=2^{-j}|E_\lambda|\leq 2^{\lambda}|E_\lambda|\quad\text{and}\quad|A_\gamma(\lambda)|\leq 2\cdot2^{\lambda}|E_\lambda|.$$ Without loss of generality, we may restrict $x \in I_\lambda$ of length $2\cdot2^{\lambda}|E_\lambda|$. By the Cauchy-Schwarz inequality, we have
\begin{align}\label{eq:2.1.2}
\left\|H^{\lambda}_{\gamma,j}(f,g)\right\|^{\frac{1}{2}}_{L^{\frac{1}{2}}(\mathbb{R})}
\leq |I_\lambda|^{\frac{1}{2}} \left(\int_{I_\lambda}\int_{E_\lambda}|f(x-2^{-j}t)g(x-\gamma(2^{-j}t))|\,\textrm{d}t\,\textrm{d}x\right)^{\frac{1}{2}}.
\end{align}
 We change the variables $u:=x-2^{-j}t$ and $v:=x-\gamma(2^{-j}t)$, and  $$\left|\frac{\partial(u,v)}{\partial(x,t)}\right|=|2^{-j}-2^{-j}\gamma'(2^{-j}t)|\geq 2^{\lambda}-\frac{2^{\lambda}}{2}=\frac{2^{\lambda}}{2}$$ for all $t\in E_\lambda$. Note that $|I_\lambda|=2\cdot2^{\lambda}|E_\lambda|$; we can control the last term in \eqref{eq:2.1.2} by
\begin{align*}
\left(\frac{|I_\lambda|}{2^{\lambda}}\right)^{\frac{1}{2}} \left(\|f\|_{L^{1}(\mathbb{R})}\|g\|_{L^{1}(\mathbb{R})}\right)^{\frac{1}{2}}\approx |E_\lambda|^{\frac{1}{2}} \left(\|f\|_{L^{1}(\mathbb{R})}\|g\|_{L^{1}(\mathbb{R})}\right)^{\frac{1}{2}}.
\end{align*}
Henceforth, in this case,
\begin{align}\label{eq:2.1.3}
\left\|H^{\lambda}_{\gamma,j}(f,g)\right\|_{L^{\frac{1}{2}}(\mathbb{R})} \ls|E_\lambda|\cdot \|f\|_{L^{1}(\mathbb{R})}\|g\|_{L^{1}(\mathbb{R})}.
\end{align}
For $E_\lambda$, by \eqref{eq:1.00}, we have $$2^{-2j}|\gamma''(2^{-j}t)|=2^{-j}\left|\frac{2^{-j}t\gamma''(2^{-j}t)}{\gamma'(2^{-j}t)}\right|\cdot\left|\frac{\gamma'(2^{-j}t)}{t}\right|\gs 2^{-j} \gamma'\left(\frac{1}{2}2^{-j}\right) $$ for $t\in\textrm{supp}~\rho$. By Lemma \ref{lemma 2.0}, it implies
\begin{align}\label{eq:2.1.4}
|E_\lambda|\leq \frac{2^{\lambda}}{2^{-j} \gamma'\left(\frac{1}{2}2^{-j}\right)}.
\end{align}
On the other hand, since $\gamma'$ is strictly increasing on $(0,\infty)$, we have
\begin{align}\label{eq:2.1.5}
2^{\lambda}\leq 2^{-j} \gamma'(2\cdot2^{-j}).
\end{align}
From \eqref{eq:2.1.4}, \eqref{eq:2.1.5} and \eqref{eq:1.02}, we obtain
\begin{align}\label{eq:2.1.6}
\sum_{\lambda\in \mathbb{Z}:\ 2^{\lambda}\geq 2\cdot2^{-j}} |E_\lambda|^{\frac{1}{2}}\leq \left[\frac{1}{2^{-j} \gamma'\left(\frac{1}{2}2^{-j}\right)}\right]^{\frac{1}{2}}  \sum_{\lambda\in \mathbb{Z}:\ 2^{\lambda}\leq 2^{-j} \gamma'(2\cdot2^{-j})} 2^{\frac{\lambda}{2}}\ls \left[\frac{2^{-j} \gamma'(2\cdot2^{-j})}{2^{-j} \gamma'\left(\frac{1}{2}2^{-j}\right)}\right]^{\frac{1}{2}} \ls 1.
\end{align}
Therefore, from \eqref{eq:2.1.3}, \eqref{eq:2.1.6}, we have
\begin{align}\label{eq:2.1.7}
\left\|\sum_{\lambda\in \mathbb{Z}:\ 2^{\lambda}\geq 2\cdot2^{-j}} H^{\lambda}_{\gamma,j}(f,g)\right\|_{L^{\frac{1}{2}}(\mathbb{R})} \ls \|f\|_{L^{1}(\mathbb{R})}\|g\|_{L^{1}(\mathbb{R})}.
\end{align}
As in \eqref{eq:2.1.6}, we have $\sum_{\lambda\in \mathbb{Z}:\ 2^{\lambda}\geq 2\cdot2^{-j}} |E_\lambda|\ls 1$, by the $\mathrm{H}\ddot{\mathrm{o}}\mathrm{lder}$ inequality, for all $p>1$,
\begin{align}\label{eq:2.1.8}
\left\|\sum_{\lambda\in \mathbb{Z}:\ 2^{\lambda}\geq 2\cdot2^{-j}} H^{\lambda}_{\gamma,j}(f,g)\right\|_{L^{1}(\mathbb{R})} \ls \sum_{\lambda\in \mathbb{Z}:\ 2^{\lambda}\geq 2\cdot2^{-j}} |E_\lambda|\cdot \|f\|_{L^{p}(\mathbb{R})}\|g\|_{L^{p'}(\mathbb{R})}\ls \|f\|_{L^{p}(\mathbb{R})}\|g\|_{L^{p'}(\mathbb{R})}.
\end{align}
By interpolation between \eqref{eq:2.1.7} and \eqref{eq:2.1.8}, for any $\frac{1}{2}<r<1$, we obtain
\begin{align}\label{eq:2.1.9}
\left\|\sum_{\lambda\in \mathbb{Z}:\ 2^{\lambda}\geq 2\cdot2^{-j}} H^{\lambda}_{\gamma,j}(f,g)\right\|_{L^{r}(\mathbb{R})} \ls\|f\|_{L^{p}(\mathbb{R})}\|g\|_{L^{q}(\mathbb{R})}.
\end{align}

{\bf Case~II:} $2^{\lambda}\leq \frac{1}{4}2^{-j}$

Noting that $$|A_\gamma(\lambda)|\leq 2\cdot2^{\lambda}|E_\lambda|\leq \frac{1}{2}2^{-j}|E_\lambda|=\frac{1}{2}|A_1(\lambda)|\leq 2^{-j}|E_\lambda|,$$
we restrict $x$ in an interval $I_\lambda$ of length $2\cdot2^{-j}|E_\lambda|$. On the other hand, we have $$\left|\frac{\partial(u,v)}{\partial(x,t)}\right|=|2^{-j}-2^{-j}\gamma'(2^{-j}t)|\geq 2^{-j}-2\cdot2^{\lambda}\geq \frac{1}{2}2^{-j}$$ for all $t\in E_\lambda$. As in Case I, we also have $$\left\|H^{\lambda}_{\gamma,j}(f,g)\right\|_{L^{\frac{1}{2}}(\mathbb{R})} \ls|E_\lambda|\cdot \|f\|_{L^{1}(\mathbb{R})}\|g\|_{L^{1}(\mathbb{R})}.$$ Furthermore, \begin{align}\label{eq:2.1.10}
\left\|\sum_{\lambda\in \mathbb{Z}:\ 2^{\lambda}\leq \frac{1}{4}2^{-j}} H^{\lambda}_{\gamma,j}(f,g)\right\|_{L^{r}(\mathbb{R})} \ls\|f\|_{L^{p}(\mathbb{R})}\|g\|_{L^{q}(\mathbb{R})}
\end{align}
for all $\frac{1}{2}<r<1$.

{\bf Case~III:} $\frac{1}{2}2^{-j}\leq2^{\lambda}\leq 2^{-j}$

We are free to assume that $2^{-j}\gamma'(2^{-j}t)>0$. Otherwise, it can be handled exactly in the same way as Case II, since $$|A_\gamma(\lambda)|\leq 2\cdot 2^{-j}|E_\lambda|,\quad |A_1(\lambda)|= 2^{-j}|E_\lambda|$$ and $$\left|\frac{\partial(u,v)}{\partial(x,t)}\right|=|2^{-j}-2^{-j}\gamma'(2^{-j}t)|\geq 2^{-j}$$ for $t\in E_\lambda$. We now consider
\begin{align*}
\mathbb{H}_{\gamma,j}(f,g)(x):=\int_{E_\gamma}f(x-2^{-j}t)g(x-\gamma(2^{-j}t))\rho(t)\,\textrm{d}t,
\end{align*}
where $E_\gamma:=\{t\in\textrm{supp}~\rho:\ \frac{1}{2}<\gamma'(2^{-j}t)\leq 2\}$. Let
$E_\gamma(h):=\{t\in E_\gamma:\ 2^h<|\gamma'(2^{-j}t)-1|\leq 2\cdot2^h\}$. By simple calculation, we know that $E_\gamma\subset \bigcup_{h\in \mathbb{Z}_-}E_\gamma(h)$. Let
\begin{align}\label{eq:2.1.11}
\mathbb{H}^{h}_{\gamma,j}(f,g)(x):=\int_{E_\gamma(h)}f(x-2^{-j}t)g(x-\gamma(2^{-j}t))\rho(t)\,\textrm{d}t.
\end{align}
Noting that $|\{2^{-j}t:\ t\in E_\gamma(h)\}|=2^{-j}|E_\gamma(h)|$ and $|\{\gamma(2^{-j}t):\ t\in E_\gamma(h)\}|\ls 2^{-j}|E_\gamma(h)|$, without loss of generality, we restrict $x\in I_h$ of length $2^{-j}|E_\gamma(h)|$. By the $\mathrm{H}\ddot{\mathrm{o}}\mathrm{lder}$ inequality, for all $p>1$, we can bound $\|\mathbb{H}^{h}_{\gamma,j}(f,g)\|_{L^{1}(\mathbb{R})}$ by
\begin{align}\label{eq:2.1.12}
\int_{I_h}\left|\int_{E_\gamma(h)}f(x-2^{-j}t)g(x-\gamma(2^{-j}t))\rho(t)\,\textrm{d}t\right|\,\textrm{d}x
\leq |E_\gamma(h)|\cdot \|f\|_{L^{p}(\mathbb{R})}\|g\|_{L^{p'}(\mathbb{R})}.
\end{align}
The Cauchy-Schwarz inequality allows us to obtain
\begin{align}\label{eq:2.1.13}
\left\|\mathbb{H}^{h}_{\gamma,j}(f,g)\right\|^{\frac{1}{2}}_{L^{\frac{1}{2}}(\mathbb{R})}  \ls |I_h|^{\frac{1}{2}} \left[\int_{I_h}\int_{E_\gamma(h)}|f(x-2^{-j}t)g(x-\gamma(2^{-j}t))|\,\textrm{d}t\,\textrm{d}x\right]^{\frac{1}{2}}.
\end{align}
Let $u:=x-2^{-j}t$ and $v:=x-\gamma(2^{-j}t)$, noting that $|\frac{\partial(u,v)}{\partial(x,t)}|=|2^{-j}-2^{-j}\gamma'(2^{-j}t)|\geq 2^h2^{-j}$ for all $t\in E_\gamma(h)$ and $|I_h|=2^{-j}|E_\gamma(h)|$, we control the last term in \eqref{eq:2.1.13} by $(\frac{|E_\gamma(h)|}{2^h})^{\frac{1}{2}} (\|f\|_{L^{1}(\mathbb{R})}\|g\|_{L^{1}(\mathbb{R})})^{\frac{1}{2}}$. Henceforth, in this case,
\begin{align}\label{eq:2.1.14}
\left\|\mathbb{H}^{h}_{\gamma,j}(f,g)\right\|_{L^{\frac{1}{2}}(\mathbb{R})} \ls\frac{|E_\gamma(h)|}{2^h} \|f\|_{L^{1}(\mathbb{R})}\|g\|_{L^{1}(\mathbb{R})}.
\end{align}
Noting that $|E_\gamma(h)|\leq 4$, from \eqref{eq:2.1.12} and \eqref{eq:2.1.14}, by interpolation, we obtain the boundedness of $\mathbb{H}^{h}_{\gamma,j}(f,g)$ for all $\frac{1}{2}<r<1$ for the case that $h=-1$. Noting that $E_\gamma\subset \bigcup_{h\in \mathbb{Z}_-}E_\gamma(h)$, in what follows, we will  focus on the second case; i.e., $h\in \mathbb{Z}_-$ and $h<-1$.

From the fact that $\gamma'$ is strictly increasing on $(0,\infty)$ and \eqref{eq:1.00}, we have $2^{-j}|\gamma''(2^{-j}t)|\gs  \gamma'(\frac{1}{2}2^{-j}) $ for $t\in E_\gamma(h)$. By Lemma \ref{lemma 2.0}, we have $|E_\gamma(h)|\ls \frac{2^{h}}{ \gamma'(\frac{1}{2}2^{-j})}$. Furthermore, we have $$\gamma'(2\cdot2^{-j})\geq |\gamma'(2^{-j}t)|=|\gamma'(2^{-j}t)-1+1|\geq 1-2\cdot2^h\geq \frac{1}{2}$$ for $t\in E_\gamma(h)$. Therefore, \eqref{eq:1.02} implies $|E_\gamma(h)|\ls \frac{2^{h}}{\gamma'(2\cdot2^{-j}) }\frac{\gamma'(2\cdot2^{-j})}{\gamma'(\frac{1}{2}2^{-j})}\ls 2^{h}$. By interpolation \eqref{eq:2.1.12} and \eqref{eq:2.1.14}, there exists a positive constant $\epsilon$ independent of $j$ such that
\begin{align}\label{eq:2.1.16}
\left\|\mathbb{H}^{h}_{\gamma,j}(f,g)\right\|_{L^{r}(\mathbb{R})} \ls 2^{\epsilon h}\|f\|_{L^{p}(\mathbb{R})}\|g\|_{L^{q}(\mathbb{R})}
\end{align}
for $\frac{1}{2}<r<1$, which leads to
\begin{align}\label{eq:2.1.17}
\left\|\sum_{h\in \mathbb{Z}_-:\ h<-1} \mathbb{H}^{h}_{\gamma,j}(f,g)\right\|_{L^{r}(\mathbb{R})} \ls\|f\|_{L^{p}(\mathbb{R})}\|g\|_{L^{q}(\mathbb{R})}.
\end{align}

Putting all the estimates together, we finish  the proof of Proposition \ref{proposition 2.1}
\end{proof}

Let $\phi$ be a standard bump function supported on $\{t\in \mathbb{R}:\ \frac{1}{2}\leq |t|\leq 2\}$ such that $0\leq\phi(t)\leq1$ and $\Sigma_{l\in \mathbb{Z}} \phi (2^{-l}t)=1$ for all $t\neq 0$. We decompose the unity as
\begin{align}\label{eq:dwfj}
\sum_{m,n,k\in \mathbb{Z}}\phi\left(\frac{\xi}{2^{m+j}}\right)\phi\left(\frac{\eta}{2^{k}}\right)\phi\left(\frac{\gamma'(2^{-j})}{2^{n+j-k}}\right)=1
\end{align}
for all $\xi,\eta\neq 0$. Let
\begin{align}\label{eq:2.8}
m_{j,m,n,k}(\xi,\eta):=m_j(\xi,\eta) \phi\left(\frac{\xi}{2^{m+j}}\right)\phi\left(\frac{\eta}{2^{k}}\right)\phi\left(\frac{\gamma'(2^{-j})}{2^{n+j-k}}\right).
\end{align}
Then
\begin{align}\label{eq:2.9}
m_j(\xi,\eta)=\sum_{m,n,k\in \mathbb{Z}}m_{j,m,n,k}(\xi,\eta).
\end{align}
We denote the diagonal as $$\triangle:=\left\{(m,n)\in \mathbb{Z}^2:\ m,n\geq 0, |m-n|\leq 2/ C_1+1 \right\},$$ and split $m_j$ as the following three parts:
\begin{align}\label{eq:2.10}
m_j(\xi,\eta)=m_j^1(\xi,\eta)+m_j^2(\xi,\eta)+m_j^3(\xi,\eta).
\end{align}
Accordingly, the \emph{low-frequency part $m_j^1$} is
\begin{align}\label{eq:2.11}
m_j^1(\xi,\eta):=\sum_{(m,n)\in (\mathbb{Z_-})^2,k\in \mathbb{Z}}m_{j,m,n,k}(\xi,\eta).
\end{align}
The \emph{high-frequency part away from the diagonal part $m_j^2$} is
\begin{align}\label{eq:2.12}
m_j^2(\xi,\eta):=\sum_{(m,n)\in \mathbb{Z}^2\backslash \((\mathbb{Z_-})^2 \bigcup \triangle  \),k\in \mathbb{Z}}m_{j,m,n,k}(\xi,\eta).
\end{align}
The \emph{high-frequency part near the diagonal part $m_j^3$} is
\begin{align}\label{eq:2.13}
m_j^3(\xi,\eta):=\sum_{(m,n)\in \triangle,k\in \mathbb{Z}}m_{j,m,n,k}(\xi,\eta).
\end{align}
Accordingly, we can split $H_\gamma(f,g)$ into three parts:
\begin{align}\label{eq:2.14}
H_{\gamma}(f,g)(x)=H_{\gamma}^1(f,g)(x)+H_{\gamma}^2(f,g)(x)+H_{\gamma}^3(f,g)(x),
\end{align}
where
\begin{align}\label{eq:2.15}
\begin{cases}H_{\gamma}^1(f,g)(x):=\sum_{j\in \mathbb{Z}}\int_{-\infty}^{\infty}\int_{-\infty}^{\infty}\hat{f}(\xi)\hat{g}(\eta)e^{i\xi x}e^{i\eta x} m_j^1(\xi,\eta)\,\textrm{d}\xi \,\textrm{d}\eta;\\
H_{\gamma}^2(f,g)(x):=\sum_{j\in \mathbb{Z}}\int_{-\infty}^{\infty}\int_{-\infty}^{\infty}\hat{f}(\xi)\hat{g}(\eta)e^{i\xi x}e^{i\eta x} m_j^2(\xi,\eta)\,\textrm{d}\xi \,\textrm{d}\eta;\\
H_{\gamma}^3(f,g)(x):=\sum_{j\in \mathbb{Z}}\int_{-\infty}^{\infty}\int_{-\infty}^{\infty}\hat{f}(\xi)\hat{g}(\eta)e^{i\xi x}e^{i\eta x} m_j^3(\xi,\eta)\,\textrm{d}\xi \,\textrm{d}\eta.
\end{cases}
\end{align}

\subsection{Decomposition of $M_\gamma(f,g)$}\label{subsection 2.3}

$M_{\gamma}(f,g)$ is a positive operator, and we may assume that $f$ and $g$ are non-negative. By simple calculation, we deduce that
\begin{align}\label{eq:2.18}
M_{\gamma}(f,g)(x)\leq \sup_{j\in \mathbb{Z}}  \int_{-\infty}^{\infty}f(x-t)g(x-\gamma(t)) 2^j|\rho(2^jt)|\,\textrm{d}t=:\sup_{j\in \mathbb{Z}}M_{j,\gamma}(f,g)(x).
\end{align}
As in Proposition \ref{proposition 2.1} above, we will focus on the case where $|j|$ is large enough.
 We rewrite $M_{j,\gamma}(f,g)$ as
\begin{align}\label{eq:2.19}
M_{j,\gamma}(f,g)(x):=\int_{-\infty}^{\infty}\int_{-\infty}^{\infty}\hat{f}(\xi)\hat{g}(\eta)e^{i\xi x}e^{i\eta x} \tilde{m}_j(\xi,\eta)\,\textrm{d}\xi \,\textrm{d}\eta,
\end{align}
where
\begin{align}\label{eq:2.20}
\tilde{m}_j(\xi,\eta):=\mathrm{p.\,v.}\int_{-\infty}^{\infty}e^{-i2^{-j}\xi t}e^{-i\eta \gamma(2^{-j}t)}|\rho(t)|
\,\textrm{d}t.
\end{align}
As in \eqref{eq:2.8}, we define $\tilde{m}_{j,m,n,k}$ as
\begin{align}\label{eq:2.21}
\tilde{m}_{j,m,n,k}(\xi,\eta):=\tilde{m}_j(\xi,\eta) \phi\left(\frac{\xi}{2^{m+j}}\right)\phi\left(\frac{\eta}{2^{k}}\right)\phi\left(\frac{\gamma'(2^{-j})}{2^{n+j-k}}\right),
\end{align}
and split $\tilde{m}_j$ as
\begin{align}\label{eq:2.22}
\tilde{m}_j(\xi,\eta)=\tilde{m}_j^1(\xi,\eta)+\tilde{m}_j^2(\xi,\eta)+\tilde{m}_j^3(\xi,\eta).
\end{align}
The \emph{low-frequency part $\tilde{m}_j^1$} is
\begin{align}\label{eq:2.23}
\tilde{m}_j^1(\xi,\eta):=\sum_{(m,n)\in (\mathbb{Z_-})^2,k\in \mathbb{Z}}\tilde{m}_{j,m,n,k}(\xi,\eta).
\end{align}
The \emph{high-frequency part away from the diagonal part $\tilde{m}_j^2$} is
\begin{align}\label{eq:2.24}
\tilde{m}_j^2(\xi,\eta):=\sum_{(m,n)\in \mathbb{Z}^2\backslash \((\mathbb{Z_-})^2 \bigcup \triangle  \),k\in \mathbb{Z}}\tilde{m}_{j,m,n,k}(\xi,\eta).
\end{align}
The \emph{high-frequency part near the diagonal part $\tilde{m}_j^3$} is
\begin{align}\label{eq:2.25}
\tilde{m}_j^3(\xi,\eta):=\sum_{(m,n)\in \triangle,k\in \mathbb{Z}}\tilde{m}_{j,m,n,k}(\xi,\eta).
\end{align}
Similarly,
\begin{align}\label{eq:2.26}
M_{\gamma}(f,g)(x)\leq M^1_{\gamma}(f,g)(x)+M^2_{\gamma}(f,g)(x)+M^3_{\gamma}(f,g)(x).
\end{align}

\section{The boundedness of $H_{\gamma}^1(f,g)$}\label{section 3}

For $m_j^1$ in \eqref{eq:2.11}, we employ the Taylor series expansion, i.e., $e^{-i2^{-j}\xi t}=\sum_{u\in \mathbb{N}}\frac{(-i2^{-j}\xi t)^u}{u\mathrm{!}}$ and $e^{-i\eta \gamma(2^{-j}t)}=\sum_{v\in \mathbb{N}}\frac{(-i\eta \gamma(2^{-j}t))^v}{v\mathrm{!}}$. Since $\phi$ is supported on $\{t\in \mathbb{R}:\ \frac{1}{2}\leq |t|\leq 2\}$, we have $|\frac{\xi}{2^{m+j}}|\leq 2$ and $|\frac{\eta\gamma'(2^{-j})}{2^{n+j}}|=|\frac{\eta}{2^{k}}|\cdot|\frac{\eta\gamma'(2^{-j})}{2^{n+j-k}}|\leq 4$. Noting that $\rho$ is supported on $\{t\in \mathbb{R}:\ \frac{1}{2}\leq |t|\leq 2\}$ and $m\in \mathbb{Z_-}$, we have $|2^{-j}\xi t|\leq 4\cdot2^m\leq 4$. On the other hand, since $\gamma$ is either odd or even and increasing on $(0,\infty)$ and satisfies \eqref{eq:1.02} and \eqref{eq:1.01}, and $\rho$ is supported on $\{t\in \mathbb{R}:\ \frac{1}{2}\leq |t|\leq 2\}$, $n\in \mathbb{Z_-}$, we  have $$|\eta\gamma(2^{-j}t)|\ls|\eta\gamma(2^{-j})|\ls|\eta2^{-j}\gamma'(2^{-j})|\ls 1.$$ Furthermore, let
\begin{align*}
\bar{\phi}_u(\xi):=\xi^u\phi(\xi) \quad\textrm{and}\quad \psi_\lambda(\xi):=2^\lambda\psi(2^\lambda\xi),
\end{align*}
where $u,\lambda\in \mathbb{Z}$. Then, we can write $m_j^1(\xi,\eta)$ as
\begin{align*}
\sum_{m,n\in \mathbb{Z_-}}\sum_{k\in \mathbb{Z}}\sum_{u,v\in \mathbb{N}}  \frac{(-i)^{u+v}2^{mu}2^{nv}}{u\mathrm{!}v\mathrm{!}} \bar{\phi}_u\left(\frac{\xi}{2^{m+j}}\right) \bar{\phi}_v\left(\frac{\eta}{2^{k}}\right) \bar{\phi}_v\left(\frac{\gamma'(2^{-j})}{2^{n+j-k}}\right)\int_{-\infty}^{\infty}t^u\left(\frac{2^j\gamma(2^{-j}t)}{\gamma'(2^{-j})}\right)^v\rho(t)\,\textrm{d}t .
\end{align*}
Therefore, $H_{\gamma}^1(f,g)(x)$ can be rewritten as
\begin{align*}
\sum_{j,k\in \mathbb{Z}}\sum_{m,n\in \mathbb{Z_-}}\sum_{u,v\in \mathbb{N}}\frac{(-i)^{u+v}2^{mu}2^{nv}}{u\mathrm{!}v\mathrm{!}} \bar{\phi}_v\left(\frac{\gamma'(2^{-j})}{2^{n+j-k}}\right) \int_{-\infty}^{\infty}t^u\left(\frac{2^j\gamma(2^{-j}t)}{\gamma'(2^{-j})}\right)^v\rho(t)\,\textrm{d}t \cdot  \check{\bar{\phi}}_{u,m+j} \ast f(x)\cdot \check{\bar{\phi}}_{v,k} \ast g(x),
\end{align*}
where $\check{\bar{\phi}}_{u,\lambda}(\xi):=2^\lambda\check{\bar{\phi}}_u(2^\lambda\xi)$ and $\check{\bar{\phi}}_u$ means the inverse Fourier transform of $\bar{\phi}_u$; by the Cauchy-Schwarz inequality, we dominate  $H_{\gamma}^1(f,g)(x)$  by
\begin{align}\label{eq:3.3}
&\sum_{m,n\in \mathbb{Z_-}}\sum_{u,v\in \mathbb{N}}\frac{2^{mu}2^{nv}}{u\mathrm{!}v\mathrm{!}} \left|\int_{-\infty}^{\infty}t^u\left(\frac{2^j\gamma(2^{-j}t)}{\gamma'(2^{-j})}\right)^v\rho(t)\,\textrm{d}t\right| \\
\times&\left[ \sum_{j,k\in \mathbb{Z}}\left|\bar{\phi}_v\left(\frac{\gamma'(2^{-j})}{2^{n+j-k}}\right)\right|\cdot\left|\check{\bar{\phi}}_{u,m+j} \ast f(x)\right|^2\right]^{\frac{1}{2}}\cdot\left[ \sum_{j,k\in \mathbb{Z}}\left|\bar{\phi}_v\left(\frac{\gamma'(2^{-j})}{2^{n+j-k}}\right)\right|\cdot\left|\check{\bar{\phi}}_{v,k} \ast g(x)\right|^2\right]^{\frac{1}{2}}.\nonumber
\end{align}
For $r\geq1$, by the triangle inequality and $\mathrm{H}\ddot{\mathrm{o}}\mathrm{lder}$ inequality, the $L^{r}(\mathbb{R})$ norm of $H_{\gamma}^1(f,g)$ is at most
\begin{align}\label{eq:3.4}
&\sum_{m,n\in \mathbb{Z_-}}\sum_{u,v\in \mathbb{N}}\frac{2^{mu}2^{nv}}{u\mathrm{!}v\mathrm{!}} \left|\int_{-\infty}^{\infty}t^u\left(\frac{2^j\gamma(2^{-j}t)}{\gamma'(2^{-j})}\right)^v\rho(t)\,\textrm{d}t\right| \\
\times&\left\| \left[ \sum_{j,k\in \mathbb{Z}}\left|\bar{\phi}_v\left(\frac{\gamma'(2^{-j})}{2^{n+j-k}}\right)\right|\cdot\left|\check{\bar{\phi}}_{u,m+j} \ast f\right|^2\right]^{\frac{1}{2}} \right\|_{L^{p}(\mathbb{R})}
\left\| \left[ \sum_{j,k\in \mathbb{Z}}\left|\bar{\phi}_v\left(\frac{\gamma'(2^{-j})}{2^{n+j-k}}\right)\right|\cdot\left|\check{\bar{\phi}}_{v,k} \ast g\right|^2\right]^{\frac{1}{2}}  \right\|_{L^{q}(\mathbb{R})}.\nonumber
\end{align}
From \eqref{eq:1.02} and \eqref{eq:1.01}, it is easy to see that
\begin{align}\label{eq:3.5}
\left|\int_{-\infty}^{\infty}t^u\left(\frac{2^j\gamma(2^{-j}t)}{\gamma'(2^{-j})}\right)^v\rho(t)\,\textrm{d}t\right| \ls 1.
\end{align}
We claim that
\begin{align}\label{eq:3.7}
\sum_{j\in \mathbb{Z}} \left|\phi\left(\frac{\gamma'(2^{-j})}{2^{n+j-k}}\right)\right|\ls 1
\end{align}
holds uniformly for $n\in \mathbb{Z_-}$ and $k\in \mathbb{Z}$. Since $0\leq\phi\leq1$, it is enough to show that the sum in \eqref{eq:3.7} has at most a finite number of terms and the number independent of $n\in \mathbb{Z_-}$ and $k\in \mathbb{Z}$. Indeed, since $\phi$ is supported on $\{t\in \mathbb{R}:\ \frac{1}{2}\leq |t|\leq 2\}$, then, for the $j$-th term in \eqref{eq:3.7}, it implies $\frac{1}{2}\frac{2^j}{\gamma'(2^{-j})}\leq \frac{1}{2^{n-k}}\leq 2\frac{2^j}{\gamma'(2^{-j})} $. For the $j+K$-th term in \eqref{eq:3.7}, it implies $\frac{1}{2}\frac{2^{j+K}}{\gamma'(2^{-j-K})}\leq \frac{1}{2^{n-k}}\leq 2\frac{2^{j+K}}{\gamma'(2^{-j-K})} $, where $K>0$. Therefore, it suffices to show that there exists a positive constant $K$ that depends only on $\gamma$ such that $2\frac{2^j}{\gamma'(2^{-j})}<\frac{1}{2}\frac{2^{j+K}}{\gamma'(2^{-j-K})}$. From \eqref{eq:1.01}, this is a direct consequence of $\frac{\gamma(2^{-j})}{\gamma(2^{-j-K})}>4\frac{1+\frac{2}{C_1}}{1+\frac{1}{2C_2}}$. From \eqref{eq:1.03}, we need only to take $K$ that satisfies $2^{(1+\frac{1}{4C_2})K}>4\frac{1+\frac{2}{C_1}}{1+\frac{1}{2C_2}}$.

As in \eqref{eq:3.7}, we also have that
\begin{align}\label{eq:3.6}
\sum_{k\in \mathbb{Z}} \left|\phi\left(\frac{\gamma'(2^{-j})}{2^{n+j-k}}\right)\right|\ls 1
\end{align}
holds uniformly for $n\in \mathbb{Z_-}$ and $j\in \mathbb{Z}$. \eqref{eq:3.6} and the Littlewood-Paley theory implies
\begin{align}\label{eq:3.8}
\left\| \left[ \sum_{j,k\in \mathbb{Z}}\left|\bar{\phi}_v\left(\frac{\gamma'(2^{-j})}{2^{n+j-k}}\right)\right|\cdot\left|\check{\bar{\phi}}_{u,m+j} \ast f\right|^2\right]^{\frac{1}{2}} \right\|_{L^{p}(\mathbb{R})}\ls \left\| \left[ \sum_{j\in \mathbb{Z}}\left|\check{\bar{\phi}}_{u,m+j} \ast f\right|^2\right]^{\frac{1}{2}} \right\|_{L^{p}(\mathbb{R})}\ls\|f\|_{L^{p}(\mathbb{R})}.
\end{align}
From \eqref{eq:3.7}, as in \eqref{eq:3.8},
\begin{align}\label{eq:3.08}
\left\| \left[ \sum_{j,k\in \mathbb{Z}}\left|\bar{\phi}_v\left(\frac{\gamma'(2^{-j})}{2^{n+j-k}}\right)\right|\cdot\left|\check{\bar{\phi}}_{v,k} \ast g\right|^2\right]^{\frac{1}{2}}  \right\|_{L^{q}(\mathbb{R})}\ls\|g\|_{L^{q}(\mathbb{R})}.
\end{align}
From the fact that $\sum_{m,n\in \mathbb{Z_-}}\sum_{u,v\in \mathbb{N}}\frac{2^{mu}2^{nv}}{u\mathrm{!}v\mathrm{!}}\ls 1$ and the estimates \eqref{eq:3.4}, \eqref{eq:3.5}, \eqref{eq:3.8} and \eqref{eq:3.08}, we have
\begin{align}\label{eq:3.9}
\left\|H_{\gamma}^1(f,g)\right\|_{L^{r}(\mathbb{R})}\ls\|f\|_{L^{p}(\mathbb{R})}\|g\|_{L^{q}(\mathbb{R})}.
\end{align}

For $\frac{1}{2}<r<1$, from \eqref{eq:3.3}, we bound $\|H_{\gamma}^1(f,g)\|^r_{L^{r}(\mathbb{R})}$ by
\begin{align}\label{eq:3.09}
&\sum_{m,n\in \mathbb{Z_-}}\sum_{u,v\in \mathbb{N}}\left(\frac{2^{mu}2^{nv}}{u\mathrm{!}v\mathrm{!}} \right)^r \left|\int_{-\infty}^{\infty}t^u\left(\frac{2^j\gamma(2^{-j}t)}{\gamma'(2^{-j})}\right)^v\rho(t)\,\textrm{d}t\right|^r \\
\times&\left\|\left[ \sum_{j,k\in \mathbb{Z}}\left|\bar{\phi}_v\left(\frac{\gamma'(2^{-j})}{2^{n+j-k}}\right)\right|\cdot\left|\check{\bar{\phi}}_{u,m+j} \ast f\right|^2\right]^{\frac{1}{2}}\cdot\left[ \sum_{j,k\in \mathbb{Z}}\left|\bar{\phi}_v\left(\frac{\gamma'(2^{-j})}{2^{n+j-k}}\right)\right|\cdot\left|\check{\bar{\phi}}_{v,k} \ast g\right|^2\right]^{\frac{1}{2}}\right\|^r_{L^{r}(\mathbb{R})}.\nonumber
\end{align}
The fact $\sum_{m,n\in \mathbb{Z_-}}\sum_{u,v\in \mathbb{N}}\left(\frac{2^{mu}2^{nv}}{u\mathrm{!}v\mathrm{!}} \right)^r\ls 1$ for all $\frac{1}{2}<r<1$ combining \eqref{eq:3.5}, \eqref{eq:3.8}, \eqref{eq:3.08} and \eqref{eq:3.09}, leads to
\begin{align}\label{eq:3.010}
\left\|H_{\gamma}^1(f,g)\right\|_{L^{r}(\mathbb{R})}\ls\|f\|_{L^{p}(\mathbb{R})}\|g\|_{L^{q}(\mathbb{R})}.
\end{align}
This is the desired estimate for the first item $H_{\gamma}^1(f,g)$.

\section{The boundedness of $H_{\gamma}^2(f,g)$}\label{section 4}

Noting that $$\frac{i}{2^{-j}\xi +\eta 2^{-j}\gamma'(2^{-j}t)}\frac{\textrm{d}}{\textrm{d}t}e^{-i2^{-j}\xi t-i\eta \gamma(2^{-j}t)}= e^{-i2^{-j}\xi t-i\eta \gamma(2^{-j}t)},$$ we split $m_{j,m,n,k}$ in \eqref{eq:2.8} as the following two parts:
\begin{align}\label{eq:3.10}
m_{j,m,n,k}(\xi,\eta)=A_{j,m,n,k}(\xi,\eta)+B_{j,m,n,k}(\xi,\eta),
\end{align}
where $A_{j,m,n,k}$ is defined as
\begin{align}\label{eq:3.11}
\left(\int_{-\infty}^{\infty}e^{-i2^{-j}\xi t-i\eta \gamma(2^{-j}t)}\frac{-i\rho'(t)}{2^{-j}\xi +\eta 2^{-j}\gamma'(2^{-j}t)}
\,\textrm{d}t\right)\phi\left(\frac{\xi}{2^{m+j}}\right)\phi\left(\frac{\eta}{2^{k}}\right)\phi\left(\frac{\gamma'(2^{-j})}{2^{n+j-k}}\right)
\end{align}
and $B_{j,m,n,k}$ is defined as
\begin{align}\label{eq:3.12}
\left(\int_{-\infty}^{\infty}e^{-i2^{-j}\xi t-i\eta \gamma(2^{-j}t)}\frac{i\eta2^{-2j}\gamma''(2^{-j}t) }{(2^{-j}\xi +\eta 2^{-j}\gamma'(2^{-j}t))^2}\rho(t)
\,\textrm{d}t\right)\phi\left(\frac{\xi}{2^{m+j}}\right)\phi\left(\frac{\eta}{2^{k}}\right)\phi\left(\frac{\gamma'(2^{-j})}{2^{n+j-k}}\right).
\end{align}
Based on this decomposition, we split $H_{\gamma}^2(f,g)$ as follows:
\begin{align}\label{eq:3.13}
&\sum_{j\in \mathbb{Z}}\sum_{(m,n)\in \mathbb{Z}^2\backslash \((\mathbb{Z_-})^2 \bigcup \triangle  \)}\sum_{k\in \mathbb{Z}}A_{\gamma,j,m,n,k}(f,g)(x)+\sum_{j\in \mathbb{Z}}\sum_{(m,n)\in \mathbb{Z}^2\backslash \((\mathbb{Z_-})^2 \bigcup \triangle  \)}\sum_{k\in \mathbb{Z}}B_{\gamma,j,m,n,k}(f,g)(x)\\
=:&A_{\gamma}^2(f,g)(x)+B_{\gamma}^2(f,g)(x),\nonumber
\end{align}
where
\begin{align}\label{eq:3.14}
\begin{cases}A_{\gamma,j,m,n,k}(f,g)(x):=\int_{-\infty}^{\infty}\int_{-\infty}^{\infty}\hat{f}(\xi)\hat{g}(\eta)e^{i\xi x}e^{i\eta x} A_{j,m,n,k}(\xi,\eta)\,\textrm{d}\xi \,\textrm{d}\eta;\\
B_{\gamma,j,m,n,k}(f,g)(x):=\int_{-\infty}^{\infty}\int_{-\infty}^{\infty}\hat{f}(\xi)\hat{g}(\eta)e^{i\xi x}e^{i\eta x} B_{j,m,n,k}(\xi,\eta)\,\textrm{d}\xi \,\textrm{d}\eta.
\end{cases}
\end{align}
Since $(m,n)\in \mathbb{Z}^2\backslash ((\mathbb{Z_-})^2 \bigcup \triangle  )$, there are two cases: $n>|m|+2/ C_1+1$ and $m>|n|+2/ C_1+1$.

\subsection{Case 1: $n>|m|+2/ C_1+1$}\label{subsection 4.1}

Applying the Taylor series expansion, we have
\begin{align}\label{eq:3.16}
\frac{1}{2^{-j}\xi +\eta 2^{-j}\gamma'(2^{-j}t)}=\frac{1}{2^n}\frac{1}{\frac{\eta\gamma'(2^{-j})}{2^{n+j}}}\sum_{l\in \mathbb{N}}\frac{(-1)^l}{2^{l(n-m)}}\left[\frac{\frac{\xi}{2^{m+j}}}{\frac{\eta\gamma'(2^{-j})}{2^{n+j}}}\right]^l\left( \frac{\gamma'(2^{-j})}{\gamma'(2^{-j}t)}\right)^{l+1}.
\end{align}
Then, $A_{\gamma,j,m,n,k}(f,g)(x)$ can be written as
\begin{align*}
\frac{-i}{2^n}\sum_{l\in \mathbb{N}}\frac{(-1)^l}{2^{l(n-m)}} \bar{\phi}_{-l-1}\left(\frac{\gamma'(2^{-j})}{2^{n+j-k}}\right)\check{\bar{\phi}}_{l,m+j} \ast f(x)\cdot \check{\bar{\phi}}_{-l-1,k} \ast g(x)\cdot \int_{-\infty}^{\infty}e^{-i2^{-j}\xi t-i\eta \gamma(2^{-j}t)}\rho'(t)\left( \frac{\gamma'(2^{-j})}{\gamma'(2^{-j}t)}\right)^{l+1}
\,\textrm{d}t.
\end{align*}
Noting that $\gamma$ is either odd or even, $\gamma'$ is increasing on $(0,\infty)$, $\rho$ is supported on $\{t\in \mathbb{R}:\ \frac{1}{2}\leq |t|\leq 2\}$, and by \eqref{eq:1.02}, we have
\begin{align}\label{eq:3.19}
\left|\int_{-\infty}^{\infty}e^{-i2^{-j}\xi t-i\eta \gamma(2^{-j}t)}\rho'(t)\left( \frac{\gamma'(2^{-j})}{\gamma'(2^{-j}t)}\right)^{l+1}
\,\textrm{d}t\right|\ls 2^{\frac{2l}{C_1}}.
\end{align}
By the Cauchy-Schwarz inequality,
\begin{align}\label{eq:3.019}
A_{\gamma}^2(f,g)(x)\leq &\sum_{(m,n)\in \mathbb{Z}^2\backslash \((\mathbb{Z_-})^2 \bigcup \triangle  \)} \frac{1}{2^n} \sum_{l\in \mathbb{N}}\frac{2^{\frac{2l}{C_1}}}{2^{l(n-m)}}\left[ \sum_{j,k\in \mathbb{Z}}\left|\bar{\phi}_{-l-1}\left(\frac{\gamma'(2^{-j})}{2^{n+j-k}}\right)\right|\cdot\left|\check{\bar{\phi}}_{l,m+j} \ast f(x)\right|^2\right]^{\frac{1}{2}}\\
&\times\left[ \sum_{j,k\in \mathbb{Z}}\left|\bar{\phi}_{-l-1}\left(\frac{\gamma'(2^{-j})}{2^{n+j-k}}\right)\right|\cdot\left|\check{\bar{\phi}}_{-l-1,k} \ast g(x)\right|^2\right]^{\frac{1}{2}}.\nonumber
\end{align}

For $r\geq1$, by the triangle inequality and $\mathrm{H}\ddot{\mathrm{o}}\mathrm{lder}$ inequality, we bound $\|A_{\gamma}^2(f,g)\|_{L^{r}(\mathbb{R})}$ by
\begin{align}\label{eq:3.20}
\sum_{(m,n)\in \mathbb{Z}^2\backslash \((\mathbb{Z_-})^2 \bigcup \triangle  \)} \frac{1}{2^n} \sum_{l\in \mathbb{N}}\frac{2^{\frac{2l}{C_1}}}{2^{l(n-m)}}&\left\| \left[ \sum_{j,k\in \mathbb{Z}}\left|\bar{\phi}_{-l-1}\left(\frac{\gamma'(2^{-j})}{2^{n+j-k}}\right)\right|\cdot\left|\check{\bar{\phi}}_{l,m+j} \ast f\right|^2\right]^{\frac{1}{2}} \right\|_{L^{p}(\mathbb{R})}\\
\times&\left\| \left[ \sum_{j,k\in \mathbb{Z}}\left|\bar{\phi}_{-l-1}\left(\frac{\gamma'(2^{-j})}{2^{n+j-k}}\right)\right|\cdot\left|\check{\bar{\phi}}_{-l-1,k} \ast g\right|^2\right]^{\frac{1}{2}}  \right\|_{L^{q}(\mathbb{R})}.\nonumber
\end{align}
On the other hand, from $n>|m|+2/ C_1+1$, it implies $\sum_{l\in \mathbb{N}}\frac{2^{\frac{2l}{C_1}}}{2^{l(n-m)}}\ls 1$. From $n>|m|+2/ C_1+1$, we have
\begin{align}\label{eq:3.21}\sum_{(m,n)\in \mathbb{Z}^2\backslash \((\mathbb{Z_-})^2 \bigcup \triangle  \)}\frac{1}{2^n}\ls \sum_{(m,n)\in \mathbb{Z}^2\backslash \((\mathbb{Z_-})^2 \bigcup \triangle  \)}\frac{1}{2^{\frac{n}{2}}}\frac{1}{2^{\frac{|m|}{2}}}\ls 1.
\end{align}
As in \eqref{eq:3.6} and \eqref{eq:3.8}, we assert that
\begin{align}\label{eq:3.22}
\left\|A_{\gamma}^2(f,g)\right\|_{L^{r}(\mathbb{R})}
\ls\|f\|_{L^{p}(\mathbb{R})}\|g\|_{L^{q}(\mathbb{R})}.
\end{align}
As in \eqref{eq:3.09}, it is easy to see that \eqref{eq:3.22} also holds for $\frac{1}{2}<r<1$.

Applying the Taylor series expansion again, we have
\begin{align}\label{eq:3.23}
\frac{\eta2^{-2j}\gamma''(2^{-j}t) }{(2^{-j}\xi +\eta 2^{-j}\gamma'(2^{-j}t))^2}=\frac{1}{2^n}\frac{1}{\frac{\xi}{2^{m+j}}}\frac{2^{-j} \gamma''(2^{-j}t) }{\gamma'(2^{-j})}\sum_{l\in \mathbb{N}}\frac{(-1)^l}{2^{l(n-m)}}\left[\frac{\frac{\xi}{2^{m+j}}}{\frac{\eta\gamma'(2^{-j})}{2^{n+j}}}\right]^{l+1}\left( \frac{\gamma'(2^{-j})}{\gamma'(2^{-j}t)}\right)^{l+1}.
\end{align}
We can then write $B_{\gamma,j,m,n}(f,g)(x)$ as
\begin{align*}
&\frac{i}{2^n}\sum_{l\in \mathbb{N}}\frac{(-1)^l}{2^{l(n-m)}} \bar{\phi}_{-l-1}\left(\frac{\gamma'(2^{-j})}{2^{n+j-k}}\right)\check{\bar{\phi}}_{l,m+j} \ast f(x)\cdot \check{\bar{\phi}}_{-l-1,k} \ast g(x) \\
 \times&\int_{-\infty}^{\infty}e^{-i2^{-j}\xi t-i\eta \gamma(2^{-j}t)}\frac{2^{-j} \gamma''(2^{-j}t) }{\gamma'(2^{-j})}\left( \frac{\gamma'(2^{-j})}{\gamma'(2^{-j}t)}\right)^{l+1}\rho(t)\,\textrm{d}t.\nonumber
\end{align*}
Since $\gamma$ is either odd or even, $\gamma'$ is increasing on $(0,\infty)$, $\rho$ is supported on $\{t\in \mathbb{R}:\ \frac{1}{2}\leq |t|\leq 2\}$, and by \eqref{eq:1.00} and \eqref{eq:1.02}, it is easy to see that $|\frac{2^{-j} \gamma''(2^{-j}t) }{\gamma'(2^{-j})}|=|\frac{2^{-j}t \gamma''(2^{-j}t) }{\gamma'(2^{-j}t)} \frac{\gamma'(2^{-j}t)}{\gamma'(2^{-j})}\frac{1}{t}|\ls 1$. Thus
\begin{align}\label{eq:3.26}
\left|\int_{-\infty}^{\infty}e^{-i2^{-j}\xi t-i\eta \gamma(2^{-j}t)}\frac{2^{-j} \gamma''(2^{-j}t) }{\gamma'(2^{-j})}\left( \frac{\gamma'(2^{-j})}{\gamma'(2^{-j}t)}\right)^{l+1}\rho(t)\,\textrm{d}t\right|\ls 2^{\frac{2l}{C_1}}.
\end{align}
As in \eqref{eq:3.20} and \eqref{eq:3.21}, for $r>\frac{1}{2}$, we assert that
\begin{align}\label{eq:3.27}
\left\|B_{\gamma}^2(f,g)\right\|_{L^{r}(\mathbb{R})}
\ls\|f\|_{L^{p}(\mathbb{R})}\|g\|_{L^{q}(\mathbb{R})}.
\end{align}

\subsection{Case 2: $m>|n|+2/ C_1+1$}\label{subsection 4.2}

By the Taylor series expansion, we have
\begin{align}\label{eq:3.28}
\frac{1}{2^{-j}\xi +\eta 2^{-j}\gamma'(2^{-j}t)}=\frac{1}{2^m}\frac{1}{\frac{\xi}{2^{m+j}}}\sum_{l\in \mathbb{N}}\frac{(-1)^l}{2^{l(m-n)}}\left[\frac{\frac{\eta\gamma'(2^{-j})}{2^{n+j}}}{\frac{\xi}{2^{m+j}}}\right]^l\left( \frac{\gamma'(2^{-j}t)}{\gamma'(2^{-j})}\right)^{l}.
\end{align}
As in case 1, we have
\begin{align}\label{eq:3.29}
\left\|A_{\gamma}^2(f,g)\right\|_{L^{r}(\mathbb{R})}
\ls\|f\|_{L^{p}(\mathbb{R})}\|g\|_{L^{q}(\mathbb{R})}
\end{align}
for $r>\frac{1}{2}$. Furthermore,
\begin{align}\label{eq:3.30}
\frac{\eta2^{-2j}\gamma''(2^{-j}t) }{(2^{-j}\xi +\eta 2^{-j}\gamma'(2^{-j}t))^2}=\frac{1}{2^m}\frac{1}{\frac{\xi}{2^{m+j}}}\frac{2^{-j} \gamma''(2^{-j}t) }{\gamma'(2^{-j}t)}\sum_{l\in \mathbb{N}}\frac{(-1)^l}{2^{(l+1)(m-n)}}\left[\frac{\frac{\eta\gamma'(2^{-j})}{2^{n+j}}}{\frac{\xi}{2^{m+j}}}\right]^{l+1}\left( \frac{\gamma'(2^{-j}t)}{\gamma'(2^{-j})}\right)^{l+1}.
\end{align}
As in case 1 , for $r>\frac{1}{2}$, again we have
\begin{align}\label{eq:3.31}
\left\|B_{\gamma}^2(f,g)\right\|_{L^{r}(\mathbb{R})}
\ls\|f\|_{L^{p}(\mathbb{R})}\|g\|_{L^{q}(\mathbb{R})}.
\end{align}

From \eqref{eq:3.13}, \eqref{eq:3.22}, \eqref{eq:3.27}, \eqref{eq:3.29} and \eqref{eq:3.31}, we may obtain
\begin{align}\label{eq:3.32}
\left\|H_{\gamma}^2(f,g)\right\|_{L^{r}(\mathbb{R})}\ls\|f\|_{L^{p}(\mathbb{R})}\|g\|_{L^{q}(\mathbb{R})}
\end{align}
for $r>\frac{1}{2}$. This is the desired estimate for the second item $H_{\gamma}^2(f,g)$.

\section{The $L^2(\mathbb{R})\times L^2(\mathbb{R})\rightarrow L^1(\mathbb{R})$ boundedness of $H_{\gamma}^3(f,g)$}\label{section 5}

Note that $\triangle=\{(m,n)\in \mathbb{Z}^2:\ m,n\geq 0, |m-n|\leq 2/ C_1+1 \}$ and $m_j^3=\sum_{(m,n)\in \triangle,k\in \mathbb{Z}}m_{j,m,n,k}$; without loss of generality, we may write that
\begin{align}\label{eq:3.33}
m_j^3=\sum_{m\in \mathbb{N}}\sum_{k\in \mathbb{Z} }m_{j,m,k}
\end{align}
where $m_{j,m,k}:=m_{j,m,m,k}$. Therefore, we rewrite $H_{\gamma}^3(f,g)$ as
\begin{align}\label{eq:3.34}
H_{\gamma}^3(f,g)(x)=\sum_{j\in \mathbb{Z}}\sum_{m\in \mathbb{N}}\sum_{k\in \mathbb{Z}}H_{j,m,k}(f,g)(x).
\end{align}
where
\begin{align}\label{eq:3.35}
H_{j,m,k}(f,g)(x):= \phi\left(\frac{\gamma'(2^{-j})}{2^{m+j-k}}\right)\int_{-\infty}^{\infty}\int_{-\infty}^{\infty}\hat{f}(\xi)\hat{g}(\eta)e^{i\xi x}e^{i\eta x} m_j(\xi,\eta) \phi\left(\frac{\xi}{2^{m+j}}\right)\phi\left(\frac{\eta}{2^{k}}\right)\,\textrm{d}\xi \,\textrm{d}\eta
\end{align}
with $m_j$ as in \eqref{eq:2.7}.

We observe that in order to obtain
\begin{align}\label{eq:3.36}
\left\|H_{\gamma}^3(f,g)\right\|_{L^{1}(\mathbb{R})}\ls \|f\|_{L^{2}(\mathbb{R})}\|g\|_{L^{2}(\mathbb{R})},
\end{align}
it suffices to prove the following Proposition \ref{proposition 3.2}. Indeed, notice that $\phi$ is a standard bump function supported on $\{t\in \mathbb{R}:\ \frac{1}{2}\leq |t|\leq 2\}$; let $\Phi$ be a bump function supported on $\{t\in \mathbb{R}:\ \frac{1}{8}\leq |t|\leq 8\}$ such that $\Phi(t)=1$ on $\{t\in \mathbb{R}:\ \frac{1}{4}\leq |t|\leq 4\}$; thus, it is safe to insert $\Phi$ into $H_{j,m,k}(f,g)$. In other words, recall that $\psi_\lambda(\xi)=2^\lambda\psi(2^\lambda\xi)$, we have
\begin{align}\label{eq:3.37}
H_{j,m,k}(f,g)(x)=\Phi\left(\frac{\gamma'(2^{-j})}{2^{m+j-k}}\right)H_{j,m,k} \left(\check{\Phi}_{m+j} \ast f, \check{\Phi}_{k} \ast g\right)(x).
\end{align}
By the triangle inequality, the Cauchy-Schwarz inequality and \eqref{eq:3.38}, it implies that $\|H_{\gamma}^3(f,g)\|_{L^{1}(\mathbb{R})}$ can be bounded by
\begin{align*}
\sum_{m\in \mathbb{N}} 2^{-\varepsilon_0 m}  \left[\sum_{j,k\in \mathbb{Z}}\left|\Phi\left(\frac{\gamma'(2^{-j})}{2^{m+j-k}}\right)\right|\cdot\left\|\check{\Phi}_{m+j} \ast f\right\|_{L^{2}(\mathbb{R})}^2\right]^{\frac{1}{2}} \cdot \left[\sum_{j,k\in \mathbb{Z}}\left|\Phi\left(\frac{\gamma'(2^{-j})}{2^{m+j-k}}\right)\right|\cdot\left\|\check{\Phi}_{k} \ast g\right\|_{L^{2}(\mathbb{R})}^2\right]^{\frac{1}{2}}.
\end{align*}
By this estimate with \eqref{eq:3.7}, \eqref{eq:3.6} and the Littlewood-Paley theory, we may obtain \eqref{eq:3.36}.

\begin{proposition}\label{proposition 3.2} There exist positive constants $C$ and $\varepsilon_0$ such that
\begin{align}\label{eq:3.38}
\left\|H_{j,m,k}(f,g)\right\|_{L^{1}(\mathbb{R})}\leq C 2^{-\varepsilon_0 m}\|f\|_{L^{2}(\mathbb{R})}\|g\|_{L^{2}(\mathbb{R})}
\end{align}
holds uniformly for $j,k\in \mathbb{Z}$.
\end{proposition}

As in \cite{L2}, we define the  \emph{bilinear operator $B_{j,m,k}(f,g)(x)$} as
\begin{align}\label{eq:3.39}
2^{\frac{m+j-k}{2}} \phi\left(\frac{\gamma'(2^{-j})}{2^{m+j-k}}\right) \int_{-\infty}^{\infty} \check{\phi}\ast f\left(2^{m+j-k}x-2^m t\right) \check{\phi}\ast g\left(x-2^k\gamma(2^{-j}t)\right) \rho(t)\,\textrm{d}t \quad\textrm{if}\quad j\geq 0;
\end{align}
and
\begin{align}\label{eq:3.40}
2^{\frac{k-m-j}{2}} \phi\left(\frac{\gamma'(2^{-j})}{2^{m+j-k}}\right) \int_{-\infty}^{\infty} \check{\phi}\ast f\left(x-2^m t\right) \check{\phi}\ast g\left(2^{k-m-j}x-2^k\gamma(2^{-j}t)\right) \rho(t)\,\textrm{d}t\quad\textrm{if}\quad j< 0.
\end{align}
We observe that \eqref{eq:3.38} is equivalent to
\begin{align}\label{eq:3.41}
\left\|B_{j,m,k}(f,g)\right\|_{L^{1}(\mathbb{R})}\ls 2^{-\varepsilon_0 m}\|f\|_{L^{2}(\mathbb{R})}\|g\|_{L^{2}(\mathbb{R})}.
\end{align}
Indeed, let $\xi:=2^{m+j}\xi$ and $\eta:=2^k\eta$; we then write $H_{j,m,k}(f,g)(x)$ as
\begin{align*}
2^{m+j+k}\phi\left(\frac{\gamma'(2^{-j})}{2^{m+j-k}}\right)\int_{-\infty}^{\infty}\int_{-\infty}^{\infty}\hat{f}(2^{m+j}\xi)\hat{g}(2^k\eta)e^{i2^{m+j}\xi x}e^{i2^k\eta x}
 \left(\int_{-\infty}^{\infty}e^{-i2^{m}\xi t}e^{-i2^k\eta \gamma(2^{-j}t)}\rho(t)
\,\textrm{d}t\right) \phi(\xi)\phi(\eta)\,\textrm{d}\xi \,\textrm{d}\eta.
\end{align*}
For $j\geq 0$, let $x:=2^{-k}x $; we may therefore write
\begin{align}\label{eq:3.43}
2^{-k}H_{j,m,k}(f,g)\left(2^{-k}x\right)=& 2^{\frac{m+j-k}{2}}\phi\left(\frac{\gamma'(2^{-j})}{2^{m+j-k}}\right)\int_{-\infty}^{\infty}\int_{-\infty}^{\infty}2^{\frac{m+j}{2}}\hat{f}(2^{m+j}\xi)2^{\frac{k}{2}}\hat{g}(2^k\eta)e^{i2^{m+j-k}\xi x}e^{i\eta x}\\
 &\times\left(\int_{-\infty}^{\infty}e^{-i2^{m}\xi t}e^{-i2^k\eta \gamma(2^{-j}t)}\rho(t)
\,\textrm{d}t\right) \phi(\xi)\phi(\eta)\,\textrm{d}\xi \,\textrm{d}\eta.\nonumber
\end{align}
Note that $\|2^{\frac{m+j}{2}}\hat{f}(2^{m+j}\cdot)\|_{L^{2}(\mathbb{R})}=\|\hat{f}\|_{L^{2}(\mathbb{R})}$ and $\|2^{\frac{k}{2}}\hat{g}(2^k\cdot)\|_{L^{2}(\mathbb{R})}=\|\hat{g}\|_{L^{2}(\mathbb{R})}$; we conclude that \eqref{eq:3.38} is equivalent to \eqref{eq:3.41}, where $B_{j,m,k}(f,g)$ is defined in \eqref{eq:3.39}. For $j< 0$, the only difference is to make the variable change  $x:=\frac{x}{2^{m+j}}$, and we omit the details.

We now turn to the proof of \eqref{eq:3.41}. In what follows, we will only focus on the first case, i.e., $j\geq 0$, since the second case, i.e., $j< 0$, is similar. From Proposition \ref{proposition 2.1} and \eqref{eq:3.41}, we can assume that $j$ and $m$ are sufficiently large.\\

\textbf{Claim:} \eqref{eq:3.41} is equivalent to Proposition \ref{proposition 3.3} below.

\begin{proof}[Proof of the Claim]
This claim is essentially \cite[Lemma 5.1]{L2}. As in \cite{L2}, let $\psi$ be a nonnegative Schwartz function such that $\hat{\psi}$ is supported on $\{t\in \mathbb{R}:\ |t|\leq \frac{1}{100}\}$ and satisfies $\hat{\psi}(0)=1$. Then, $B_{j,m,k}(f,g)(x)$ can be written as
\begin{align*}
& 2^{\frac{m+j-k}{2}} \phi\left(\frac{\gamma'(2^{-j})}{2^{m+j-k}}\right) \sum_{N\in \mathbb{Z}} \sum_{k_1,k_2\in \mathbb{Z}}\int_{-\infty}^{\infty} \left(\chi_{[2^m(N+k_1),2^m(N+k_1+1)]}\ast \psi_{-m}\cdot \check{\phi}\ast f\right)\left(2^{m+j-k}x-2^m t\right)\\
\times& \left(\chi_{\left[2^{k-j}(N+k_2),2^{k-j}(N+k_2+1)\right]}\ast \psi_{j-k} \cdot\check{\phi}\ast g\right)\left(x-2^k\gamma(2^{-j}t)\right) \cdot\rho(t)\,\textrm{d}t\cdot\chi_{\left[2^{k-j}N, 2^{k-j}(N+1)\right]}(x).\nonumber
\end{align*}

We split $$B_{j,m,k}(f,g)(x):=B^I_{j,m,k}(f,g)(x)+B^{II}_{j,m,k}(f,g)(x),$$ where $B^I_{j,m,k}(f,g)$ sums over $A:=\left\{k_1,k_2\in \mathbb{Z}:\ \max\{|k_1|,|k_2|\}\geq \Theta \right\}$ and $B^{II}_{j,m,k}(f,g)$ sums over $B:=\left\{k_1,k_2\in \mathbb{Z}:\ \max\{|k_1|,|k_2|\}< \Theta \right\}$, where $\Theta:=2^{\frac{\varepsilon_0'}{3} m}$ is sufficiently large.

For $B^I_{j,m,k}(f,g)$,  we have $|(\chi_{[2^m(N+k_1),2^m(N+k_1+1)]}\ast \psi_{-m})(2^{m+j-k}x-2^m t)|\ls \frac{1}{|k_1|^3}$. Note that $\frac{\gamma(t)}{t}$ is increasing on $(0,\infty)$; it implies $2^j\gamma(2^{-j}t)$ is sufficiently small if $j$ is sufficiently large, and we also have $|(\chi_{[2^{k-j}(N+k_2),2^{k-j}(N+k_2+1)]}\ast \psi_{j-k}) (x-2^k\gamma(2^{-j}t))|\ls \frac{1}{|k_2|^3}$. Thus, $B^I_{j,m,k}(f,g)(x)$ is bounded by
\begin{align}\label{eq:3.44}
& 2^{\frac{m+j-k}{2}} \phi\left(\frac{\gamma'(2^{-j})}{2^{m+j-k}}\right) \sum_{N\in \mathbb{Z}} \sum_{k_1,k_2\in A} \frac{1}{|k_1|^3}\frac{1}{|k_2|^3}\int_{-\infty}^{\infty} \left| \check{\phi}\ast f\left(2^{m+j-k}x-2^m t\right)\right|\\
&\times\left|\check{\phi}\ast g\left(x-2^k\gamma(2^{-j}t)\right)\right|\cdot| \rho(t)|\,\textrm{d}t\cdot\chi_{\left[2^{k-j}N, 2^{k-j}(N+1)\right]}(x)\nonumber\\
\ls&\frac{1}{\Theta} 2^{\frac{m+j-k}{2}} \phi\left(\frac{\gamma'(2^{-j})}{2^{m+j-k}}\right)\int_{-\infty}^{\infty} \left| \check{\phi}\ast f\left(2^{m+j-k}x-2^m t\right)\check{\phi}\ast g\left(x-2^k\gamma(2^{-j}t)\right)\right|\cdot| \rho(t)|\,\textrm{d}t.\nonumber
\end{align}
The last inequality is a result of the fact that $\sum_{|k|\geq \Theta}\frac{1}{|k|^3} \ls \frac{1}{\Theta}$ and $\sum_{k\in \mathbb{Z}}\frac{1}{|k|^3} \ls 1$.
 By the $\mathrm{H}\ddot{\mathrm{o}}\mathrm{lder}$ and Young inequalities, it now follows that
\begin{align}\label{eq:3.45}
\left\|B^I_{j,m,k}(f,g)\right\|_{L^{1}(\mathbb{R})}\ls \frac{1}{\Theta}\|f\|_{L^{2}(\mathbb{R})}\|g\|_{L^{2}(\mathbb{R})}.
\end{align}

For $B^{II}_{j,m,k}(f,g)$, let $j$ and $m$ be sufficiently large; we have that
\begin{align*}
\mathcal{F}\left(\chi_{[2^m(N+k_1),2^m(N+k_1+1)]}\ast \psi_{-m}\cdot \check{\phi}\ast f\right)
\quad\textrm{and}\quad
\mathcal{F}\left(\chi_{\left[2^{k-j}(N+k_2),2^{k-j}(N+k_2+1)\right]}\ast \psi_{j-k} \cdot\check{\phi}\ast g\right)
\end{align*}
are supported on $\{t\in \mathbb{R}:\ \frac{1}{4}\leq |t|\leq 4\}$, where $\mathcal{F}(f)$ means the Fourier transform of $f$. Then, $B^{II}_{j,m,k}(f,g)(x)$ can be written as
\begin{align*}
&\sum_{N\in \mathbb{Z}} \sum_{k_1,k_2\in B}B^*_{j,m,k}\left(\chi_{[2^m(N+k_1),2^m(N+k_1+1)]}\ast \psi_{-m}\cdot \check{\phi}\ast f,\chi_{\left[2^{k-j}(N+k_2),2^{k-j}(N+k_2+1)\right]}\ast \psi_{j-k} \cdot\check{\phi}\ast g\right)(x)\\
\times &\chi_{\left[2^{k-j}N, 2^{k-j}(N+1)\right]}(x).\nonumber
\end{align*}
The definition of $B^*_{j,m,k}(f,g)$ will be given in Proposition \ref{proposition 3.3}. From $\sum_{N\in \mathbb{Z}} |\chi_{[2^m(N+k_1),2^m(N+k_1+1)]}\ast \psi_{-m} |^2\ls \|\psi\|_{L^{1}(\mathbb{R})}^2$ and $\sum_{N\in \mathbb{Z}} |\chi_{[2^{k-j}(N+k_2),2^{k-j}(N+k_2+1)]}\ast \psi_{j-k} |^2\ls \|\psi\|_{L^{1}(\mathbb{R})}^2$, by the Cauchy-Schwarz inequality and the Young inequality, \eqref{eq:3.48}, $\|B^{II}_{j,m,k}(f,g)\|_{L^{1}(\mathbb{R})}$ can be bounded by
\begin{align}\label{eq:3.46}
& \Theta^2 2^{-\varepsilon_0' m} \left[\sum_{N\in \mathbb{Z}} \left\|\chi_{[2^m(N+k_1),2^m(N+k_1+1)]}\ast \psi_{-m}\cdot \check{\phi}\ast f\right\|_{L^{2}(\mathbb{R})}^2\right]^{\frac{1}{2}}\\
\times&\left[\sum_{N\in \mathbb{Z}} \left\|\chi_{\left[2^{k-j}(N+k_2),2^{k-j}(N+k_2+1)\right]}\ast \psi_{j-k} \cdot\check{\phi}\ast g\right\|_{L^{2}(\mathbb{R})}^2\right]^{\frac{1}{2}}
\ls\Theta^2 2^{-\varepsilon_0' m}\|f\|_{L^{2}(\mathbb{R})}\|g\|_{L^{2}(\mathbb{R})}.\nonumber
\end{align}
From \eqref{eq:3.45} and \eqref{eq:3.46}, note that $\Theta=2^{\frac{\varepsilon_0'}{3} m}$; we obtain
\begin{align}\label{eq:3.47}
\left\|B_{j,m,k}(f,g)\right\|_{L^{1}(\mathbb{R})}\ls 2^{-\frac{\varepsilon_0'}{3} m}\|f\|_{L^{2}(\mathbb{R})}\|g\|_{L^{2}(\mathbb{R})}.
\end{align}
This is \eqref{eq:3.41} if we let $\varepsilon_0:=\frac{\varepsilon_0'}{3}$.
\end{proof}

\begin{proposition}\label{proposition 3.3} There exist positive constants $C$ and $\varepsilon_0'$ such that
\begin{align}\label{eq:3.48}
\left\|B^*_{j,m,k}(f,g)\cdot\chi_{\left[2^{k-j}N, 2^{k-j}(N+1)\right]}\right\|_{L^{1}(\mathbb{R})}\leq C 2^{-\varepsilon_0' m}\|f\|_{L^{2}(\mathbb{R})}\|g\|_{L^{2}(\mathbb{R})}
\end{align}
holds uniformly for $j\in \mathbb{Z}$ and $N\in \mathbb{Z}$, where $B^*_{j,m,k}(f,g)$ is defined as
\begin{align*}
B^*_{j,m,k}(f,g)(x):=2^{\frac{m+j-k}{2}} \phi\left(\frac{\gamma'(2^{-j})}{2^{m+j-k}}\right) \int_{-\infty}^{\infty} \check{\Phi}\ast f\left(2^{m+j-k}x-2^m t\right) \check{\Phi}\ast g\left(x-2^k\gamma(2^{-j}t)\right) \rho(t)\,\textrm{d}t
\end{align*}
and $\Phi$ is the same bump function as in \eqref{eq:3.37}.
\end{proposition}

We will prove Proposition \ref{proposition 3.3} in three steps.

\subsection{An estimate by using $TT^*$ argument}\label{subsection 5.1}

In this subsection, we show that
\begin{align}\label{eq:3.50}
\left\|B^*_{j,m,k}(f,g)\cdot\chi_{\left[2^{k-j}N, 2^{k-j}(N+1)\right]}\right\|_{L^{1}(\mathbb{R})}\ls  2^{-\frac{2m+j-k}{6}}\|f\|_{L^{2}(\mathbb{R})}\|g\|_{L^{2}(\mathbb{R})}.
\end{align}
By the $\mathrm{H}\ddot{\mathrm{o}}\mathrm{lder}$ inequality, it suffices to show that
\begin{align}\label{eq:3.51}
\left\|B^*_{j,m,k}(f,g)\right\|_{L^{2}(\mathbb{R})}\ls  2^{-\frac{2m+j-k}{6}}2^{\frac{j-k}{2}}\|f\|_{L^{2}(\mathbb{R})}\|g\|_{L^{2}(\mathbb{R})}.
\end{align}
We rewrite $B^*_{j,m,k}(f,g)(x)$ as
\begin{align}\label{eq:3.52}
2^{\frac{m+j-k}{2}} \phi\left(\frac{\gamma'(2^{-j})}{2^{m+j-k}}\right)\int_{-\infty}^{\infty}\int_{-\infty}^{\infty}\hat{f}(\xi)\hat{g}(\eta)
e^{i2^{m+j-k}\xi x}e^{i\eta x}I_{j,m,k}(\xi,\eta) \Phi(\xi)\Phi(\eta)\,\textrm{d}\xi \,\textrm{d}\eta,
\end{align}
where
\begin{align}\label{eq:3.53}
I_{j,m,k}(\xi,\eta):=\int_{-\infty}^{\infty}e^{-i2^m\xi t}e^{-i 2^k\eta \gamma(2^{-j}t)}\rho(t)\,\textrm{d}t.
\end{align}
Let
\begin{align}\label{eq:3.54}
\varphi(t,\xi,\eta):=-\xi t- 2^{k-m}\eta \gamma(2^{-j}t).
\end{align}
We denote $\varphi_1'$ as the derivative of $\varphi$ with respect to the first variable $t$.
 Then
\begin{align}\label{eq:3.55}
\varphi_1'(t,\xi,\eta)=-\xi -2^{k-m-j} \gamma'(2^{-j}t)\eta \quad \textrm{and}\quad \varphi_1''(t,\xi,\eta)= -2^{k-m-2j}\gamma''(2^{-j}t)\eta.
\end{align}
We define $t_0(\xi,\eta)$ as
\begin{align}\label{eq:3.56}
\varphi_1'(t_0(\xi,\eta),\xi,\eta)=0.
\end{align}
For simplicity, we may further assume that $\frac{1}{2}\leq|t_0(\xi,\eta)|\leq 2$. By the method of the stationary phase, we assert that
\begin{align}\label{eq:3.57}
I_{j,m,k}(\xi,\eta)=e^{i2^m \varphi(t_0(\xi,\eta),\xi,\eta)}\left(\frac{2\pi}{-i2^m\varphi_1''(t_0(\xi,\eta),\xi,\eta)} \right)^{\frac{1}{2}} \rho(t_0(\xi,\eta))+ O\left(2^{-\frac{3}{2}m}\right).
\end{align}
For $h\in L^{2}(\mathbb{R})$, we have $\int_{-\infty}^{\infty}B^*_{j,m,k}(f,g)(x)h(x)\,\textrm{d}x$ is equal to
\begin{align}\label{eq:3.58}
2^{\frac{m+j-k}{2}} \phi\left(\frac{\gamma'(2^{-j})}{2^{m+j-k}}\right)\int_{-\infty}^{\infty}\int_{-\infty}^{\infty}\hat{f}(\xi)\Phi(\xi) \hat{g}(\eta)\Phi(\eta)I_{j,m,k}(\xi,\eta) \check{h}\left(2^{m+j-k}\xi+\eta\right)\,\textrm{d}\xi \,\textrm{d}\eta.
\end{align}
Based on \eqref{eq:3.57}, we estimate $\int_{-\infty}^{\infty}B^*_{j,m,k}(f,g)(x)h(x)\,\textrm{d}x$ by considering the following two parts:

{\bf Part A:} $O\left(2^{-\frac{3}{2}m}\right)$

With some abuse of notation, we write $\int_{-\infty}^{\infty}B^*_{j,m,k}(f,g)(x)h(x)\,\textrm{d}x$ as
\begin{align*}
2^{\frac{m+j-k}{2}} \phi\left(\frac{\gamma'(2^{-j})}{2^{m+j-k}}\right)\int_{-\infty}^{\infty}\int_{-\infty}^{\infty}\hat{f}(\xi)\Phi(\xi) \hat{g}(\eta)\Phi(\eta) \check{h}\left(2^{m+j-k}\xi+\eta\right)\cdot O\left(2^{-\frac{3}{2}m}\right)\,\textrm{d}\xi \,\textrm{d}\eta.
\end{align*}
Thus, by the $\mathrm{H}\ddot{\mathrm{o}}\mathrm{lder}$ inequality and Plancherel's formula, we bound $|\int_{-\infty}^{\infty}B^*_{j,m,k}(f,g)(x)h(x)\,\textrm{d}x|$ by
\begin{align*}
2^{-\frac{3}{2}m} 2^{\frac{m+j-k}{2}}\int_{-\infty}^{\infty} |\hat{f}(\xi)\Phi(\xi)| \int_{-\infty}^{\infty} \left|\hat{g}(\eta)\Phi(\eta) \check{h}\left(2^{m+j-k}\xi+\eta\right)\right| \,\textrm{d}\eta \,\textrm{d}\xi
\ls  2^{-\frac{3}{2}m} 2^{\frac{m+j-k}{2}}\|f\|_{L^{2}(\mathbb{R})} \|g\|_{L^{2}(\mathbb{R})}\|h\|_{L^{2}(\mathbb{R})}.
\end{align*}
Furthermore, let $h:=\textrm{sgn}(B^*_{j,m,k}(f,g))\cdot\chi_{[2^{k-j}N, 2^{k-j}(N+1)]}$; we have
\begin{align*}
\left\|B^*_{j,m,k}(f,g)\cdot\chi_{\left[2^{k-j}N, 2^{k-j}(N+1)\right]}\right\|_{L^{1}(\mathbb{R})}\ls 2^{- m}\|f\|_{L^{2}(\mathbb{R})}\|g\|_{L^{2}(\mathbb{R})}.
\end{align*}
This is \eqref{eq:3.48} as desired.

{\bf Part B:} $e^{i2^m \varphi(t_0(\xi,\eta),\xi,\eta)}\left(\frac{2\pi}{-i2^m\varphi_1''(t_0(\xi,\eta),\xi,\eta)} \right)^{\frac{1}{2}} \rho(t_0(\xi,\eta))$

From \eqref{eq:1.00} and \eqref{eq:1.02}, we have $|\frac{2^{-j} \gamma''(2^{-j}t) }{\gamma'(2^{-j})}|=|\frac{2^{-j}t \gamma''(2^{-j}t) }{\gamma'(2^{-j}t)} \frac{\gamma'(2^{-j}t)}{\gamma'(2^{-j})}\frac{1}{t}|\approx 1$. For simplicity, we write $\int_{-\infty}^{\infty}B^*_{j,m,k}(f,g)(x)h(x)\,\textrm{d}x$ as
\begin{align}\label{eq:3.59}
2^{\frac{j-k}{2}} \phi\left(\frac{\gamma'(2^{-j})}{2^{m+j-k}}\right)\int_{-\infty}^{\infty}\int_{-\infty}^{\infty}\hat{f}(\xi)\Phi(\xi) \hat{g}(\eta)\Phi(\eta) e^{i2^m \varphi(t_0(\xi,\eta),\xi,\eta)} \check{h}\left(2^{m+j-k}\xi+\eta\right)\,\textrm{d}\xi \,\textrm{d}\eta.
\end{align}
Changing variables  $\xi:=\frac{\xi-\eta}{2^{m+j-k}}$ and $\eta:=2^{m+j-k}\eta$, we have
\begin{align*}
&2^{\frac{j-k}{2}}\phi\left(\frac{\gamma'(2^{-j})}{2^{m+j-k}}\right)\int_{-\infty}^{\infty}\int_{-\infty}^{\infty}
(\hat{f}\Phi)\left(\frac{\xi}{2^{m+j-k}}-\eta\right) (\hat{g}\Phi)\left(2^{m+j-k}\eta\right) \\
\times& e^{i2^m \varphi\left(t_0\left(\frac{\xi}{2^{m+j-k}}-\eta,2^{m+j-k}\eta\right),\frac{\xi}{2^{m+j-k}}-\eta,2^{m+j-k}\eta\right)} \check{h}(\xi)\,\textrm{d}\xi \,\textrm{d}\eta.\nonumber
\end{align*}
From the H\"older inequality and Plancherel formula, it can be bounded by $2^{\frac{j-k}{2}} \|T_{j,m,k}(f,g)\|_{L^{2}(\mathbb{R})} \|h\|_{L^{2}(\mathbb{R})}$, where $T_{j,m,k}(f,g)(\xi)$ is defined as
\begin{align*}
\phi\left(\frac{\gamma'(2^{-j})}{2^{m+j-k}}\right)\int_{-\infty}^{\infty}(\hat{f}\Phi)\left(\frac{\xi}{2^{m+j-k}}-\eta\right) (\hat{g}\Phi)\left(2^{m+j-k}\eta\right)
e^{i2^m \varphi\left(t_0\left(\frac{\xi}{2^{m+j-k}}-\eta,2^{m+j-k}\eta\right),\frac{\xi}{2^{m+j-k}}-\eta,2^{m+j-k}\eta\right)}  \,\textrm{d}\eta.
\end{align*}
Therefore, \eqref{eq:3.51} can be reduced to
\begin{align}\label{eq:3.62}
\left\|T_{j,m,k}(f,g)\right\|_{L^{2}(\mathbb{R})}\ls  2^{-\frac{2m+k+j}{6}} \|f\|_{L^{2}(\mathbb{R})}\|g\|_{L^{2}(\mathbb{R})}.
\end{align}
By the $TT^*$ argument, we obtain that $\|T_{j,m,k}(f,g)\|^2_{L^{2}(\mathbb{R})}$ equals to
\begin{align*}
\int_{-\infty}^{\infty}\Bigg[\int_{-\infty}^{\infty}\int_{-\infty}^{\infty}\phi^2\left(\frac{\gamma'(2^{-j})}{2^{m+j-k}}\right)F(\xi,\eta_1,\eta_2) & G(\xi,\eta_1,\eta_2) e^{i2^m \varphi\left(t_0\left(\frac{\xi}{2^{m+j-k}}-\eta_1,2^{m+j-k}\eta_1\right),\frac{\xi}{2^{m+j-k}}-\eta_1,2^{m+j-k}\eta_1\right)}\\
\times & e^{-i2^m \varphi\left(t_0\left(\frac{\xi}{2^{m+j-k}}-\eta_2,2^{m+j-k}\eta_2\right),\frac{\xi}{2^{m+j-k}}-\eta_2,2^{m+j-k}\eta_2\right)} \,\textrm{d}\eta_1 \,\textrm{d}\eta_2\Bigg]\,\textrm{d}\xi,\nonumber
\end{align*}
where
$$
\begin{cases}F(\xi,\eta_1,\eta_2):=(\hat{f}\Phi)\left(\frac{\xi}{2^{m+j-k}}-\eta_1\right)\overline{(\hat{f}\Phi)\left(\frac{\xi}{2^{m+j-k}}-\eta_2\right)};\\
G(\xi,\eta_1,\eta_2):=(\hat{g}\Phi)\left(2^{m+j-k}\eta_1\right)\overline{(\hat{g}\Phi)\left(2^{m+j-k}\eta_2\right)}.
\end{cases}
$$
Let $\eta_1:=\eta$ and $\eta_2:=\eta+\tau$; then, $\|T_{j,m,k}(f,g)\|^2_{L^{2}(\mathbb{R})}$ is equal to
\begin{align*}
\int_{-\infty}^{\infty}\Bigg[\int_{-\infty}^{\infty}\int_{-\infty}^{\infty}\phi^2\left(\frac{\gamma'(2^{-j})}{2^{m+j-k}}\right)&F_\tau\left(\frac{\xi}{2^{m+j-k}}-\eta\right) G_\tau\left(2^{m+j-k}\eta\right) e^{i2^m \varphi\left(t_0\left(\frac{\xi}{2^{m+j-k}}-\eta,2^{m+j-k}\eta\right),\frac{\xi}{2^{m+j-k}}-\eta,2^{m+j-k}\eta\right)} \\
\times & e^{-i2^m \varphi\left(t_0\left(\frac{\xi}{2^{m+j-k}}-\eta-\tau,2^{m+j-k}(\eta+\tau)\right),\frac{\xi}{2^{m+j-k}}-\eta-\tau,2^{m+j-k}(\eta+\tau)\right)}\,\textrm{d}\xi \,\textrm{d}\eta\Bigg]\,\textrm{d}\tau,\nonumber
\end{align*}
where $F_\tau(\cdot):=(\hat{f}\Phi)(\cdot)\overline{(\hat{f}\Phi)(\cdot-\tau)}$ and $G_\tau(\cdot):=(\hat{g}\Phi)(\cdot)\overline{(\hat{g}\Phi)\left(\cdot+2^{m+j-k}\tau\right)}$. Furthermore, let $u:=\frac{\xi}{2^{m+j-k}}-\eta$ and $v:=2^{m+j-k}\eta$; we have
\begin{align}\label{eq:3.63}
\left\|T_{j,m,k}(f,g)\right\|^2_{L^{2}(\mathbb{R})}=\phi^2\left(\frac{\gamma'(2^{-j})}{2^{m+j-k}}\right)\int_{-\infty}^{\infty}\left(\int_{-\infty}^{\infty}\int_{-\infty}^{\infty}F_\tau(u) G_\tau(v) e^{i2^m Q_\tau(u,v)}
 \,\textrm{d}u \,\textrm{d}v\right)\,\textrm{d}\tau,
\end{align}
where
\begin{align}\label{eq:3.64}
Q_\tau(u,v):=\varphi(t_0(u,v),u,v)-\varphi\left(t_0\left(u-\tau,v+2^{m+j-k}\tau\right),u-\tau,v+2^{m+j-k}\tau\right).
\end{align}
To estimate the bilinear operator $T_{j,m,k}(f,g)$ and obtain \eqref{eq:3.62}, we use the $\textrm{H}\ddot{\textrm{o}}\textrm{rmander}$ theorem on the nondegenerate phase \cite{H}. To this aim, we need to establish Proposition \ref{proposition 3.4} below. Let us postpone the proof of Proposition \ref{proposition 3.4} for the moment. We now turn to $\|T_{j,m,k}(f,g)\|^2_{L^{2}(\mathbb{R})}$ in \eqref{eq:3.63}. Let $\tau_0:=2^{-\frac{2m+j-k}{3}}$; noting that $\Phi$ is supported on $\{t\in \mathbb{R}:\ \frac{1}{8}\leq |t|\leq 8\}$, we split it as
\begin{align}\label{eq:3.65}
\left\|T_{j,m,k}(f,g)\right\|^2_{L^{2}(\mathbb{R})}=&\phi^2\left(\frac{\gamma'(2^{-j})}{2^{m+j-k}}\right)\int_{|\tau|\leq \tau_0}\left(\int_{-\infty}^{\infty}\int_{-\infty}^{\infty}F_\tau(u) G_\tau(v) e^{i2^m Q_\tau(u,v)}
 \,\textrm{d}u \,\textrm{d}v\right)\,\textrm{d}\tau\\
 &+\phi^2\left(\frac{\gamma'(2^{-j})}{2^{m+j-k}}\right)\int_{\tau_0<|\tau|\leq 16}\left(\int_{-\infty}^{\infty}\int_{-\infty}^{\infty}F_\tau(u) G_\tau(v) e^{i2^m Q_\tau(u,v)}
 \,\textrm{d}u \,\textrm{d}v\right)\,\textrm{d}\tau.\nonumber
\end{align}
By the $\mathrm{H}\ddot{\mathrm{o}}\mathrm{lder}$ inequality and Plancherel formula, it is easy to see that $\|F_\tau\|_{L^{1}(\mathbb{R})} \leq\|f\|^2_{L^{2}(\mathbb{R})}$ and $\|G_\tau\|_{L^{1}(\mathbb{R})}\leq\|g\|^2_{L^{2}(\mathbb{R})}$. On the other hand, we have
\begin{align*}
\left(\int_{\tau_0<|\tau|\leq 16}\|F_\tau\|^2_{L^{2}(\mathbb{R})} \,\textrm{d}\tau\right)^{\frac{1}{2}} =\left(\int_{\tau_0<|\tau|\leq 16}\int_{-\infty}^{\infty} \left|(\hat{f}\Phi)(x)\overline{(\hat{f}\Phi)(x-\tau)}\right|^2 \,\textrm{d}x  \,\textrm{d}\tau\right)^{\frac{1}{2}}\leq \|f\|^2_{L^{2}(\mathbb{R})},
\end{align*}
and $(\int_{\tau_0<|\tau|\leq 16}\|G_\tau\|^2_{L^{2}(\mathbb{R})} \,\textrm{d}\tau)^{\frac{1}{2}} \leq  2^{-\frac{m+j-k}{2}}\|g\|^2_{L^{2}(\mathbb{R})}$. With all the estimates, from \eqref{eq:3.67}, by the $\mathrm{H}\ddot{\mathrm{o}}\mathrm{lder}$ inequality and $\textrm{H}\ddot{\textrm{o}}\textrm{rmander}$ \cite[Theorem 1.1]{H}, we bound $\|T_{j,m,k}(f,g)\|^2_{L^{2}(\mathbb{R})}$ by
\begin{align}\label{eq:3.66}
&\int_{|\tau|\leq \tau_0}\|F_\tau\|_{L^{1}(\mathbb{R})}\|G_\tau\|_{L^{1}(\mathbb{R})} \,\textrm{d}\tau+\int_{\tau_0<|\tau|\leq 16} \left(2^m|\tau_0|\right)^{-\frac{1}{2}}\|F_\tau\|_{L^{2}(\mathbb{R})} \|G_\tau\|_{L^{2}(\mathbb{R})}  \,\textrm{d}\tau.\\
\leq&\tau_0\|f\|^2_{L^{2}(\mathbb{R})}\|g\|^2_{L^{2}(\mathbb{R})}+ \left(2^m|\tau_0|\right)^{-\frac{1}{2}} \left(\int_{\tau_0<|\tau|\leq 16}\|F_\tau\|^2_{L^{2}(\mathbb{R})} \,\textrm{d}\tau\right)^{\frac{1}{2}}\left(\int_{\tau_0<|\tau|\leq 16}\|G_\tau\|^2_{L^{2}(\mathbb{R})} \,\textrm{d}\tau\right)^{\frac{1}{2}}\nonumber\\
\leq&\left(\tau_0+\left(2^m|\tau_0| \right)^{-\frac{1}{2}}2^{-\frac{m+j-k}{2}} \right)\|f\|^2_{L^{2}(\mathbb{R})}\|g\|^2_{L^{2}(\mathbb{R})}
\ls 2^{-\frac{2m+j-k}{3}} \|f\|^2_{L^{2}(\mathbb{R})}\|g\|^2_{L^{2}(\mathbb{R})}.\nonumber
\end{align}
Thus, we obtain \eqref{eq:3.62}, which leads to \eqref{eq:3.50}.

\begin{proposition}\label{proposition 3.4}
Let $u,v,u-\tau,v+2^{m+j-k}\tau \in\textrm{supp}~\Phi$. Then, there exists a positive constant $C$ such that
\begin{align}\label{eq:3.67}
\left|\frac{\textrm{d}^2Q_\tau}{\textrm{d}u\textrm{d}v}(u,v)\right|\geq C |\tau|,
\end{align}
where $j$ is large enough.
\end{proposition}

\begin{proof}
Recall that
\begin{align}\label{eq:3.68}
Q_\tau(u,v)=\varphi(t_0(u,v),u,v)-\varphi\left(t_0\left(u-\tau,v+2^{m+j-k}\tau\right),u-\tau,v+2^{m+j-k}\tau\right)
\end{align}
and
\begin{align}\label{eq:3.69}
\varphi(t,u,v)=-u t-v 2^{k-m} \gamma(2^{-j}t).
\end{align}
Let $\Psi(u,v):=\varphi(t_0(u,v),u,v)$. Then
\begin{align}\label{eq:3.70}
\Psi(u,v)=-u t_0(u,v)- v 2^{k-m}\gamma(2^{-j}t_0(u,v)),
\end{align}
where $t_0(u,v)$ satisfies $\varphi_1'(t_0(u,v),u,v)=0$. The definition of $t_0(u,v)$ leads to
\begin{align}\label{eq:3.71}
-u - v2^{-m-j+k}\gamma'(2^{-j}t_0(u,v))=0.
\end{align}
Furthermore, we have
\begin{align}\label{eq:3.72}
\frac{\textrm{d}t_0}{\textrm{d}u}(u,v)=-\frac{1}{v}\frac{2^{m+2j-k}}{\gamma''(2^{-j}t_0(u,v)) },
\end{align}
and
\begin{align}\label{eq:3.73}
\frac{\textrm{d}t_0}{\textrm{d}v}(u,v)=-\frac{1}{v}\frac{\gamma'(2^{-j}t_0(u,v) )}{2^{-j}\gamma''(2^{-j}t_0(u,v)) }=\frac{u}{v^2}\frac{2^{m+2j-k}}{\gamma''(2^{-j}t_0(u,v)) }.
\end{align}
Therefore,
\begin{align}\label{eq:3.74}
\frac{\textrm{d}\Psi}{\textrm{d}v}(u,v)=-2^{k-m} \gamma(2^{-j}t_0(u,v))
\end{align}
and from \eqref{eq:3.71},
\begin{align}\label{eq:3.75}
\frac{\textrm{d}^2\Psi}{\textrm{d}u\textrm{d}v}(u,v)= \frac{u}{v}\frac{\textrm{d}t_0}{\textrm{d}u}(u,v).
\end{align}
Then
\begin{align}\label{eq:3.76}
\frac{\textrm{d}^3\Psi}{\textrm{d}^2u\textrm{d}v}(u,v)= \frac{1}{v^2} \frac{2^{m+2j-k}}{\gamma''(2^{-j}t_0(u,v)) } \left(\frac{\gamma'\gamma'''-(\gamma'')^2}{(\gamma'')^2}\right)\left(2^{-j}t_0(u,v)\right).
\end{align}
From \eqref{eq:1.0}, we have
\begin{align}\label{eq:3.77}
\left|\left(\frac{\gamma'\gamma'''-(\gamma'')^2}{(\gamma'')^2}\right)\left(2^{-j}t_0(u,v)\right)\right|\approx 1.
\end{align}
From \eqref{eq:1.02} and \eqref{eq:1.00}, note that $\gamma$ is either odd or even, $\gamma'$ is increasing on $(0,\infty)$ and $\frac{\gamma'(2^{-j})}{2^{m+j-k}}\in \textrm{supp}~\phi$; thus, we have
\begin{align}\label{eq:3.078}
\left|\frac{2^{m+2j-k}}{\gamma''(2^{-j}t_0(u,v)) }\right|=\left|\frac{2^{m+j-k}}{\gamma'(2^{-j})} \frac{\gamma'(2^{-j})t_0(u,v)}{\gamma'(2^{-j}t_0(u,v))} \frac{\gamma'(2^{-j}t_0(u,v))}{2^{-j}t_0(u,v)\gamma''(2^{-j}t_0(u,v))}\right|\approx  1.
\end{align}
Since $v\in \textrm{supp}~\Psi$, from \eqref{eq:3.77} and \eqref{eq:3.078},
\begin{align}\label{eq:3.78}
\left|\frac{\textrm{d}^3\Psi}{\textrm{d}^2u\textrm{d}v}(u,v)\right|\approx 1.
\end{align}
On the other hand, we write $\frac{\textrm{d}^3\Psi}{\textrm{d}v\textrm{d}u\textrm{d}v}(u,v)$ as
\begin{align}\label{eq:3.79}
\frac{1}{v^2} \frac{\gamma'(2^{-j}t_0(u,v))}{2^{-j}\gamma''(2^{-j}t_0(u,v)) } \left(\frac{\gamma'\gamma'''-(\gamma'')^2}{(\gamma'')^2}\right)\left(2^{-j}t_0(u,v)\right)-\frac{1}{v^2} \frac{\gamma'(2^{-j}t_0(u,v))}{2^{-j}\gamma''(2^{-j}t_0(u,v)) }.
\end{align}
As in \eqref{eq:3.78}, we have
\begin{align}\label{eq:3.80}
\left|\frac{\textrm{d}^3\Psi}{\textrm{d}v\textrm{d}u\textrm{d}v}(u,v)\right|\ls 1.
\end{align}
By the mean value theorem, we rewrite $\frac{\textrm{d}^2Q_\tau}{\textrm{d}u\textrm{d}v}(u,v)$ as
\begin{align*}
\frac{\textrm{d}^3\Psi}{\textrm{d}^2u\textrm{d}v}(u-\theta_1\tau,v)\cdot\tau-\frac{\textrm{d}^3\Psi}{\textrm{d}v\textrm{d}u\textrm{d}v}\left(u-\tau,v+\theta_22^{m+j-k}\tau\right)\cdot2^{m+j-k}\tau,
\end{align*}
where $\theta_1, \theta_2\in[0,1]$. From \eqref{eq:3.78},\eqref{eq:3.80}, and $2^{m+j-k}\approx \gamma'(2^{-j})$ when $j$ is large enough, we have
\begin{align}\label{eq:3.81}
\left|\frac{\textrm{d}^2Q_\tau}{\textrm{d}u\textrm{d}v}(u,v)\right|\gs (1-2^{m+j-k})|\tau|\gs |\tau|.
\end{align}
This finishes the proof of Proposition \ref{proposition 3.4}.
\end{proof}

\subsection{Another estimate by $\sigma$-uniformity and the $TT^*$ argument}\label{subsection 5.2}

In this subsection, we set up
\begin{eqnarray}\label{eq:3.82}
\left\|B^*_{j,m,k}(f,g)\cdot\chi_{\left[2^{k-j}N, 2^{k-j}(N+1)\right]}\right\|_{L^{1}(\mathbb{R})}\ls
\left\{\aligned
2^{-\frac{m}{16}}\|f\|_{L^{2}(\mathbb{R})}\|g\|_{L^{2}(\mathbb{R})},\quad  \textrm{if} \quad j-k+2m\leq 0,\\
\Lambda_{j,m,k}\|f\|_{L^{2}(\mathbb{R})}\|g\|_{L^{2}(\mathbb{R})},\quad  \textrm{if} \quad  j-k+2m> 0,
\endaligned\right.
\end{eqnarray}
where $\Lambda_{j,m,k}:=\max\{2^{\frac{15m}{16}+\frac{j}{2}-\frac{k}{2}},  2^{\frac{m}{16}}(\max\{2^{m+j-k},2^{-\frac{m}{4}}\} )^{\frac{1}{2}}  \}$.

We start by quoting a lemma stated in \cite[Lemma 4.4]{DGR} or \cite[Lemma 3.3]{GX}, which is a slight variant of \cite[Theorem 7.1]{L2} and is called the $\sigma$-uniformity argument. Indeed, this argument can be traced back to Christ et al. \cite{CLTT} and Gowers \cite{G1}. Let $\sigma\in (0,1]$, $\mathbb{I}\subset \mathbb{R}$ be a fixed bounded interval, and $U(\mathbb{I})$ be a nontrivial subset of $L^2(\mathbb{I})$ with $\|u\|_{L^2(\mathbb{I})}\leq C$ uniformly for every element of $u\in U(\mathbb{I})$. We say that a function $f\in L^2(\mathbb{I})$ is \emph{$\sigma$-uniform in $U(\mathbb{I})$} if
\begin{align*}
\left| \int_{\mathbb{I}}f(x)\overline{u(x)}\,\textrm{d}x \right|\leq \sigma\|f\|_{L^{2}(\mathbb{I})}
\end{align*}
for all $u\in U(\mathbb{I})$.
\begin{lemma}\label{lemma 3.1}
\cite[Theorem 7.1]{L2} Let $\mathcal{L}$ be a bounded sublinear functional from $L^2(\mathbb{I})$ to $\mathbb{C}$, $S_\sigma$ be the set of all functions that are $\sigma$-uniform in $U(\mathbb{I})$,
\begin{align*}
 \mathcal{A}_\sigma:=\sup_{f\in S_\sigma} \frac{|\mathcal{L}(f)|}{\|f\|_{L^{2}(\mathbb{I})}} \quad \textrm{and}\quad  \mathcal{M}:=\sup_{u\in U(\mathbb{I})} |\mathcal{L}(u)|.
\end{align*}
Then, for all $f\in L^{2}(\mathbb{I})$, we have
$$|\mathcal{L}(f)|\leq \max\left\{\mathcal{A}_\sigma, 2\sigma^{-1}\mathcal{M}\right\}\|f\|_{L^{2}(\mathbb{I})}.$$
\end{lemma}

We now turn to the proof of \eqref{eq:3.82} by using Lemma \ref{lemma 3.1}. Let $\mathbb{I}:=\textrm{supp}~\Phi$, and for any given $g\in L^{2}(\mathbb{R})$, let $$\mathcal{L}(\chi_{\mathbb{I}}\hat{f}):=\left\|B^*_{j,m,k}(f,g)\cdot\chi_{[2^{k-j}N, 2^{k-j}(N+1)]}\right\|_{L^{1}(\mathbb{R})}.$$

{\bf Step 1: Estimates for $\mathcal{A}_\sigma$}

We split the interval $[2^{k-j}N, 2^{k-j}(N+1)]$ as $\bigcup_{w=1}^{2^m} I_w$, where $$I_w:=[a_w,a_{w+1}]:=\left[2^{k-j}N+\frac{w-1}{2^{m+j-k}}, 2^{k-j}N+\frac{w}{2^{m+j-k}}\right].$$ Furthermore, let us set $$I'_w:=\left[2^{k-j}N+\frac{w-1}{2^{m+j-k}}-2^{2+\frac{2}{C_1}}2^m, 2^{k-j}N+\frac{w}{2^{m+j-k}}+2^{2+\frac{2}{C_1}}2^m\right].$$ It is easy to see that $x-2^k\gamma(2^{-j}t)\in I'_w$ if $x\in I_w$, since $\frac{\gamma'(2t)}{\gamma'(t)}\leq 2^{\frac{2}{C_1}}$ for all $t\in \textrm{supp} ~\rho$. We write $B^*_{j,m,k}(f,g)(x)\cdot\chi_{[2^{k-j}N, 2^{k-j}(N+1)]}(x)$ as
\begin{align*}
2^{\frac{m+j-k}{2}} \phi\left(\frac{\gamma'(2^{-j})}{2^{m+j-k}}\right)\sum_{w=1}^{2^m}\chi_{I_w}(x)\int_{-\infty}^{\infty}\int_{-\infty}^{\infty}\hat{f}(\xi) \Phi(\xi)    \widehat{\dot{g}_w}(\eta)
e^{i2^{m+j-k}\xi x}e^{i\eta x}I_{j,m,k}(\xi,\eta) \,\textrm{d}\xi \,\textrm{d}\eta,
\end{align*}
where $I_{j,m,k}$ can be found in \eqref{eq:3.53} and $\dot{g}_w:=(\chi_{I'_w}\cdot\check{\Phi}\ast g)$.

For $h\in L^{2}(\mathbb{R})$, based on $\xi \in\textrm{supp}~\Phi$ and $2^{m+j-k}|x-a_w|\leq 1$ for all $x\in I_w$, by Taylor's theorem $e^{i2^{m+j-k}\xi (x-a_w)}$, it is safe to split $\int_{-\infty}^{\infty}B^*_{j,m,k}(f,g)(x)\cdot\chi_{[2^{k-j}N, 2^{k-j}(N+1)]}(x)h(x)\,\textrm{d}x$ as the sum of $I$ and $II$, where
$$
\begin{cases}I:=&2^{\frac{m+j-k}{2}} \phi\left(\frac{\gamma'(2^{-j})}{2^{m+j-k}}\right)\sum_{w=1}^{2^m}\sum_{l\in \mathbb{N}}\frac{i^l}{l\mathrm{!}} \int_{-\infty}^{\infty}  \int_{-\infty}^{\infty}\hat{f}(\xi)\Phi(\xi) e^{i2^{m+j-k}\xi a_w} \widehat{\dot{g}_w}(\eta)(1-\mathcal{G}(\eta) )   I_{j,m,k}(\xi,\eta)\xi^l \\
 &\times\mathcal{F}^{-1}\left( \left(2^{m+j-k}(\cdot-a_w)\right)^l\chi_{I_w}(\cdot)h(\cdot)\right)(\eta)   \,\textrm{d}\xi \,\textrm{d}\eta;\\
II:=&2^{\frac{m+j-k}{2}} \phi\left(\frac{\gamma'(2^{-j})}{2^{m+j-k}}\right)\sum_{w=1}^{2^m}\sum_{l\in \mathbb{N}}\frac{i^l}{l\mathrm{!}} \int_{-\infty}^{\infty}  \int_{-\infty}^{\infty}\hat{f}(\xi)\Phi(\xi) e^{i2^{m+j-k}\xi a_w} \widehat{\dot{g}_w}(\eta) \mathcal{G}(\eta) I_{j,m,k}(\xi,\eta)\xi^l \\
 &\times\mathcal{F}^{-1}\left( \left(2^{m+j-k}(\cdot-a_w)\right)^l\chi_{I_w}(\cdot)h(\cdot)\right)(\eta)   \,\textrm{d}\xi \,\textrm{d}\eta,
\end{cases}
$$
$\mathcal{G}$ is a bump function supported on $\{t\in \mathbb{R}:\ 2^{-6-\frac{2}{C_1}}\leq |t|\leq  2^{6+\frac{2}{C_1}}\}$ such that $\mathcal{G}(t)=1$ on $\{t\in \mathbb{R}:\ 2^{-5-\frac{2}{C_1}}\leq |t|\leq 2^{5+\frac{2}{C_1}}\}$, and $\mathcal{F}^{-1}(f)$ means the inverse Fourier transform of $f$.

For $I$. First, we obtain that $|I_{j,m,k}|\leq 2^{-m}$. Indeed, note that $|\eta|\leq 2^{-5-\frac{2}{C_1}}$ or $|\eta|\geq 2^{5+\frac{2}{C_1}}$, and $\xi \in\textrm{supp}~\Phi$, $\frac{\gamma'(2^{-j})}{2^{m+j-k}}\in \textrm{supp}~\phi$; we have $|\varphi_1'(t,\xi,\eta)|\geq |\xi|-|2^{-m-j+k}\gamma'(2^{-j}t)\eta|\geq \frac{1}{16}$ or $|\varphi_1'(t,\xi,\eta)|\geq |2^{-m-j+k}\gamma'(2^{-j}t)\eta|-|\xi|\geq  8$. Note that $\varphi_1''(t,\xi,\eta)=- 2^{-m-j+k}2^{-j}\gamma''(2^{-j}t)\eta$, by the Van der Corput lemma, for example, see \cite[P. 332, Proposition 2]{S}, we have $|I_{j,m}|\leq 2^{-m}$. Therefore, by the $\mathrm{H}\ddot{\mathrm{o}}\mathrm{lder}$ inequality, Plancherel formula and Cauchy-Schwarz inequality, we bound $|I|$ by
\begin{align*}
& 2^{\frac{m+j-k}{2}} \phi\left(\frac{\gamma'(2^{-j})}{2^{m+j-k}}\right)\sum_{w=1}^{2^m}\sum_{l\in \mathbb{N}}\frac{8^l}{l\mathrm{!}} 2^{-m} \left\|\chi_{\mathbb{I}}\hat{f}\right\|_{L^{2}(\mathbb{R})} \|\dot{g}_w\|_{L^{2}(\mathbb{R})}\left\|\left(2^{m+j-k}(\cdot-a_w)\right)^l\chi_{I_w}(\cdot)h(\cdot)\right\|_{L^{2}(\mathbb{R})}\\
\leq& 2^{\frac{m+j-k}{2}} \phi\left(\frac{\gamma'(2^{-j})}{2^{m+j-k}}\right)\sum_{l\in \mathbb{N}}\frac{8^l}{l\mathrm{!}} 2^{-m} \left\|\chi_{\mathbb{I}}\hat{f}\right\|_{L^{2}(\mathbb{R})}
\left[\sum_{w=1}^{2^m}\|\dot{g}_w\|^2_{L^{2}(\mathbb{R})} \right]^{\frac{1}{2}}  \left[\sum_{w=1}^{2^m}\left\|\left(2^{m+j-k}(\cdot-a_w)\right)^l\chi_{I_w}(\cdot)h(\cdot)\right\|^2_{L^{2}(\mathbb{R})}\right]^{\frac{1}{2}}\nonumber\\
\leq& 2^{\frac{-m+j-k}{2}}  \left\|\chi_{\mathbb{I}}\hat{f}\right\|_{L^{2}(\mathbb{R})}
\left[\sum_{w=1}^{2^m}\|\dot{g}_w\|^2_{L^{2}(\mathbb{R})} \right]^{\frac{1}{2}}    \|h\|_{L^{2}(\mathbb{R})},\nonumber
\end{align*}
where the last inequality is a result of the fact that $2^{m+j-k}|x-a_w|\leq 1$ for all $x\in I_w$.
Based on the overlap property of $\{I'_w \}_{w=1}^{2^m}$, we have
$$
\begin{cases}\left[\sum_{w=1}^{2^m}\|\dot{g}_w\|^2_{L^{2}(\mathbb{R})} \right]^{\frac{1}{2}} \ls \|g\|_{L^{2}(\mathbb{R})}, \quad  &\textrm{if} \quad j-k+2m\leq 0;\\
\left[\sum_{w=1}^{2^m}\|\dot{g}_w\|^2_{L^{2}(\mathbb{R})} \right]^{\frac{1}{2}} \ls 2^{\frac{2m+j-k}{2}}\|g\|_{L^{2}(\mathbb{R})},\quad  &\textrm{if} \quad j-k+2m> 0.
\end{cases}
$$
Therefore,
\begin{eqnarray}\label{eq:3.87}
|I|\ls
\left\{\aligned
2^{-\frac{m}{2}} \left\|\chi_{\mathbb{I}}\hat{f}\right\|_{L^{2}(\mathbb{R})}
\|g\|_{L^{2}(\mathbb{R})}  2^{\frac{j-k}{2}} \|h\|_{L^{2}(\mathbb{R})},\quad  \textrm{if} \quad j-k+2m\leq 0,\\
2^{\frac{m+j-k}{2}}\left\|\chi_{\mathbb{I}}\hat{f}\right\|_{L^{2}(\mathbb{R})}
\|g\|_{L^{2}(\mathbb{R})}  2^{\frac{j-k}{2}}\|h\|_{L^{2}(\mathbb{R})},\quad  \textrm{if} \quad  j-k+2m> 0.
\endaligned\right.
\end{eqnarray}
For $II$, we write $II$  as $II_1+II_2$ by using the decomposition in \eqref{eq:3.57}. As in $I$, it is easy to see that
\begin{eqnarray}\label{eq:3.88}
|II_1|\ls
\left\{\aligned
2^{-m} \left\|\chi_{\mathbb{I}}\hat{f}\right\|_{L^{2}(\mathbb{R})}
\|g\|_{L^{2}(\mathbb{R})}  2^{\frac{j-k}{2}} \|h\|_{L^{2}(\mathbb{R})},\quad  \textrm{if} \quad j-k+2m\leq 0,\\
 2^{\frac{j-k}{2}}\left\|\chi_{\mathbb{I}}\hat{f}\right\|_{L^{2}(\mathbb{R})}
\|g\|_{L^{2}(\mathbb{R})}   2^{\frac{j-k}{2}} \|h\|_{L^{2}(\mathbb{R})},\quad  \textrm{if} \quad  j-k+2m> 0.
\endaligned\right.
\end{eqnarray}
$II_2$  can be written as
\begin{align}\label{eq:3.89}
&2^{\frac{m+j-k}{2}} \phi\left(\frac{\gamma'(2^{-j})}{2^{m+j-k}}\right)\sum_{w=1}^{2^m}\sum_{l\in \mathbb{N}}\frac{i^l}{l\mathrm{!}} \int_{-\infty}^{\infty}  \int_{-\infty}^{\infty}\hat{f}(\xi)\Phi(\xi) e^{i2^{m+j-k}\xi a_w} \widehat{\dot{g}_w}(\eta) \mathcal{G}(\eta) e^{i2^m \varphi(t_0(\xi,\eta),\xi,\eta)}\\
&\times\left(\frac{2\pi}{-i2^m\varphi_1''(t_0(\xi,\eta),\xi,\eta)} \right)^{\frac{1}{2}} \rho(t_0(\xi,\eta))\xi^l \mathcal{F}^{-1}\left( \left(2^{m+j-k}(\cdot-a_w)\right)^l\chi_{I_w}(\cdot)h(\cdot)\right)(\eta)   \,\textrm{d}\xi \,\textrm{d}\eta\nonumber\\
 =:&\left(\frac{2\pi}{-i} \right)^{\frac{1}{2}}2^{\frac{j-k}{2}} \phi\left(\frac{\gamma'(2^{-j})}{2^{m+j-k}}\right)\sum_{w=1}^{2^m}\sum_{l\in \mathbb{N}}\frac{i^l}{l\mathrm{!}} \int_{-\infty}^{\infty}\Upsilon(\eta)\widehat{\dot{g}_w}(\eta)\mathcal{F}^{-1}\left( \left(2^{m+j-k}(\cdot-a_w)\right)^l\chi_{I_w}(\cdot)h(\cdot)\right)(\eta)\,\textrm{d}\eta,\nonumber
\end{align}
where
$$
\begin{cases}\Upsilon(\eta):=\int_{\mathbb{I}}\upsilon(\xi,\eta)\hat{f}(\xi)e^{i2^m \varphi(t_0(\xi,\eta),\xi,\eta)}e^{i2^{m+j-k}\xi a_w} \,\textrm{d}\xi;\\
\upsilon(\xi,\eta):=\rho(t_0(\xi,\eta))\left(\frac{1}{\varphi_1''(t_0(\xi,\eta),\xi,\eta)} \right)^{\frac{1}{2}}\mathcal{G}(\eta)\xi^l\Phi(\xi).
\end{cases}
$$
To apply Lemma \ref{lemma 3.1}, we first define
\begin{align*}
U(\mathbb{I}):=\left\{u_{\eta,r}(\xi)\in L^{2}(\mathbb{I}):\ r\in \mathbb{R}, 2^{-6-\frac{2}{C_1}}\leq |\eta|\leq 2^{6+\frac{2}{C_1}}  \right\},
\end{align*}
where $u_{\eta,r}(\xi):=\upsilon(\xi,\eta)e^{-i2^m \varphi(t_0(\xi,\eta),\xi,\eta)}e^{-ir\xi}$. From \eqref{eq:1.00}, \eqref{eq:3.55} and $\frac{\gamma'(2^{-j})}{2^{m+j-k}}\in \textrm{supp}~\phi$, it is easy to see that $\|u_{\eta,r}\|_{L^2(\mathbb{I})}\leq C$ uniformly for every element of $u_{\eta,r}\in U(\mathbb{I})$. To estimate $\mathcal{A}_\sigma$, we first assume that $\chi_{\mathbb{I}}\hat{f}\in L^2(\mathbb{I})$ is $\sigma$-uniform in $U(\mathbb{I})$, which further implies that
\begin{align}\label{eq:3.90}
|\Upsilon(\eta)|\leq \sigma \left\|\chi_{\mathbb{I}}\hat{f}\right\|_{L^{2}(\mathbb{R})}.
\end{align}
Regarding $I$, let $\sigma:=2^{-\frac{m}{16}}$; it is easy to see that
\begin{eqnarray}\label{eq:3.91}
|II_2|\ls
\left\{\aligned
2^{-\frac{m}{16}} \left\|\chi_{\mathbb{I}}\hat{f}\right\|_{L^{2}(\mathbb{R})}
\|g\|_{L^{2}(\mathbb{R})}   2^{\frac{j-k}{2}} \|h\|_{L^{2}(\mathbb{R})}, \quad \textrm{if} \quad j-k+2m\leq 0,\\
2^{\frac{7m}{16}} 2^{\frac{m+j-k}{2}}\left\|\chi_{\mathbb{I}}\hat{f}\right\|_{L^{2}(\mathbb{R})}
\|g\|_{L^{2}(\mathbb{R})}   2^{\frac{j-k}{2}} \|h\|_{L^{2}(\mathbb{R})},  \quad\textrm{if} \quad  j-k+2m> 0.
\endaligned\right.
\end{eqnarray}
From \eqref{eq:3.87}, \eqref{eq:3.88}, and \eqref{eq:3.91}, $|\int_{-\infty}^{\infty}B^*_{j,m,k}(f,g)(x)\cdot\chi_{[2^{k-j}N, 2^{k-j}(N+1)]}(x)h(x)\,\textrm{d}x|$ can be bounded by
\begin{eqnarray*}
\left\{\aligned
2^{-\frac{m}{16}} \left\|\chi_{\mathbb{I}}\hat{f}\right\|_{L^{2}(\mathbb{R})}
\|g\|_{L^{2}(\mathbb{R})}  2^{\frac{j-k}{2}} \|h\|_{L^{2}(\mathbb{R})}, \quad \textrm{if} \quad j-k+2m\leq 0,\\
2^{\frac{7m}{16}}2^{\frac{m+j-k}{2}}\left\|\chi_{\mathbb{I}}\hat{f}\right\|_{L^{2}(\mathbb{R})}
\|g\|_{L^{2}(\mathbb{R})} 2^{\frac{j-k}{2}} \|h\|_{L^{2}(\mathbb{R})},  \quad\textrm{if} \quad  j-k+2m> 0.
\endaligned\right.
\end{eqnarray*}
Let $h:=\textrm{sgn} (B^*_{j,m,k}(f,g))\cdot\chi_{[2^{k-j}N, 2^{k-j}(N+1)]}$.  $\|B^*_{j,m,k}(f,g)\cdot\chi_{[2^{k-j}N, 2^{k-j}(N+1)]}\|_{L^{1}(\mathbb{R})}$ can be bounded by
\begin{eqnarray}\label{eq:3.92}
\left\{\aligned
2^{-\frac{m}{16}} \left\|\chi_{\mathbb{I}}\hat{f}\right\|_{L^{2}(\mathbb{R})}
\|g\|_{L^{2}(\mathbb{R})}, \quad \textrm{if} \quad j-k+2m\leq 0,\\
2^{\frac{7m}{16}}2^{\frac{m+j-k}{2}}\left\|\chi_{\mathbb{I}}\hat{f}\right\|_{L^{2}(\mathbb{R})}
\|g\|_{L^{2}(\mathbb{R})},  \quad\textrm{if} \quad  j-k+2m> 0.\nonumber
\endaligned\right.
\end{eqnarray}
It is easy to see that by the definition of $\mathcal{L}(\chi_{\mathbb{I}}\hat{f})$ that
\begin{eqnarray}\label{eq:3.93}
\mathcal{A}_\sigma\ls
\left\{\aligned
2^{-\frac{m}{16}}
\|g\|_{L^{2}(\mathbb{R})}, \quad \textrm{if} \quad j-k+2m\leq 0,\\
2^{\frac{7m}{16}}2^{\frac{m+j-k}{2}}
\|g\|_{L^{2}(\mathbb{R})},  \quad\textrm{if} \quad  j-k+2m> 0.
\endaligned\right.
\end{eqnarray}

{\bf Step 2: Estimates for $\mathcal{M}$}

For $\varpi\in L^\infty(\mathbb{R})$, let $x:=2^{k-j}x+2^k\gamma(2^{-j}t)$; $\int_{-\infty}^{\infty}B^*_{j,m,k}(f,g)(x)\varpi(x)\,\textrm{d}x$ can be written as
\begin{align}\label{eq:3.94}
2^{\frac{m-j+k}{2}} \phi\left(\frac{\gamma'(2^{-j})}{2^{m+j-k}}\right)\int_{-\infty}^{\infty} \int_{-\infty}^{\infty}& \check{\Phi}\ast f\left(2^mx+2^{m+j}\gamma(2^{-j}t)-2^m t\right) \\
\times&\check{\Phi}\ast g\left(2^{k-j}x\right) \varpi\left(2^{k-j}x+2^k\gamma(2^{-j}t)\right)\rho(t)\,\textrm{d}t\,\textrm{d}x.\nonumber
\end{align}
Therefore,
\begin{align}\label{eq:3.95}
\left|\int_{-\infty}^{\infty}B^*_{j,m,k}(f,g)(x)\varpi(x)\,\textrm{d}x\right|\leq \|g\|_{L^{2}(\mathbb{R})}\|\mathcal{T}(\varpi)\|_{L^{2}(\mathbb{R})},
\end{align}
where $\mathcal{T}(\varpi)(x)$ is defined as
\begin{align*}
2^{\frac{m}{2}}\phi\left(\frac{\gamma'(2^{-j})}{2^{m+j-k}}\right)\int_{-\infty}^{\infty}\check{\Phi}\ast f\left(2^mx+2^{m+j}\gamma(2^{-j}t)-2^m t\right)\varpi\left(2^{k-j}x+2^k\gamma(2^{-j}t)\right)\rho(t)\,\textrm{d}t.
\end{align*}
Let $\chi_{\mathbb{I}}(\xi)\hat{f}(\xi):=\upsilon(\xi,\eta)e^{-i2^m \varphi(t_0(\xi,\eta),\xi,\eta)}e^{-ir\xi}$ and $x:=x+2^{-m}r$, $\|\mathcal{T}(\varpi)\|^2_{L^{2}(\mathbb{R})}$ equals to
\begin{align}\label{eq:3.97}
2^m \phi^2\left(\frac{\gamma'(2^{-j})}{2^{m+j-k}}\right)\int_{-\infty}^{\infty} \left| \int_{-\infty}^{\infty}    \kappa(x,t) \varpi\left(2^{k-j}x+2^{k-j-m}r+2^k\gamma(2^{-j}t)\right) \rho(t)\,\textrm{d}t\right|^2 \,\textrm{d}x,
\end{align}
where $$\kappa(x,t):=\int_{-\infty}^{\infty}\Phi(\xi) \upsilon(\xi,\eta) e^{i2^m\left(x\xi+2^j\gamma(2^{-j}t)\xi-t\xi- \varphi(t_0(\xi,\eta),\xi,\eta)\right)}\,\textrm{d}\xi.$$ Furthermore, let us set
\begin{align}\label{eq:3.98}
K(\xi,x,t):=x\xi+2^j\gamma(2^{-j}t)\xi-t\xi- \varphi(t_0(\xi,\eta),\xi,\eta),
\end{align}
we have $K'_1(\xi,x,t):=x+2^j\gamma(2^{-j}t)-t+t_0(\xi,\eta)$. There exists a positive constant $\delta$ such that $|2^j\gamma(2^{-j}t)-t+t_0(\xi,\eta)|\leq \frac{\delta}{4}$ for $j$ large enough. Therefore, we can split \eqref{eq:3.97} as
\begin{align}\label{eq:3.99}
\|\mathcal{T}(\varpi)\|^2_{L^{2}(\mathbb{R})}=:\|\mathcal{T}(\varpi),a\|^2_{L^{2}(\mathbb{R})}+\|\mathcal{T}(\varpi),b\|^2_{L^{2}(\mathbb{R})}
\end{align}
by $1=(1-\Delta(x))+\Delta(x)$ on the right-hand side of \eqref{eq:3.97}, where $\Delta$ is a bump function supported on $\{x\in \mathbb{R}:\ |x|\leq \delta\}$ such that $\Delta(x)=1$ on $\{x\in \mathbb{R}:\ |x|\leq \frac{\delta}{2}\}$. For $\|\mathcal{T}(\varpi),a\|^2_{L^{2}(\mathbb{R})}$, we have $|K'_1(\xi,x,t)|\geq |x|-|2^j\gamma(2^{-j}t)-t+t_0(\xi,\eta)|\geq |x|-\frac{\delta}{4}\geq \frac{|x|}{2}$. Note that $K''_1(\xi,x,t)=-\frac{1}{\eta}\frac{2^{m+2j-k}}{\gamma''(2^{-j}t_0(\xi,\eta)) }$, by the Van der Corput lemma, for example, see \cite[P. 332, Proposition 2]{S}, we have that $|\kappa(x,t)|\leq 2^{-m}|x|^{-1}$. Then
\begin{align}\label{eq:3.100}
\|\mathcal{T}(\varpi),a\|^2_{L^{2}(\mathbb{R})}\ls 2^{m} \left[ \int_{|x|\geq \frac{\delta}{2}} (2^{-m}|x|^{-1})^2 \,\textrm{d}x\right] \cdot \|\varpi\|^2_{L^\infty(\mathbb{R})}\ls 2^{-m}\|\varpi\|^2_{L^\infty(\mathbb{R})}.
\end{align}
We now turn to $\|\mathcal{T}(\varpi),b\|^2_{L^{2}(\mathbb{R})}$; it has been defined as
\begin{align}\label{eq:3.101}
2^m \phi^2\left(\frac{\gamma'(2^{-j})}{2^{m+j-k}}\right)\int_{-\infty}^{\infty} \left| \int_{-\infty}^{\infty}    \kappa(x,t)\varpi\left(2^{k-j}x+2^{k-j-m}r+2^k\gamma(2^{-j}t)\right) \rho(t)\,\textrm{d}t\right|^2 \Delta(x)\,\textrm{d}x.
\end{align}
For $\kappa(x,t)$, whose phase function is $2^mK(\xi,x,t)$. Let $\xi(x,t)$ satisfy $K'_1(\xi(x,t),x,t)=0$; in other words,
\begin{align}\label{eq:3.102}
x+2^j\gamma(2^{-j}t)-t+t_0(\xi(x,t),\eta)=0.
\end{align}
By the stationary phase method, we assert that
\begin{align}\label{eq:3.103}
\kappa(x,t)=e^{i2^m K(\xi(x,t),x,t)}\left(\frac{2\pi}{-i2^mK_1''(\xi(x,t),x,t)} \right)^{\frac{1}{2}} \Phi(\xi(x,t)) \cdot\upsilon(\xi(x,t),\eta)+ O(2^{-\frac{3}{2}m}).
\end{align}
From \eqref{eq:3.103}, we split $\|\mathcal{T}(\varpi),b\|^2_{L^{2}(\mathbb{R})}$ as
\begin{align}\label{eq:3.104}
\|\mathcal{T}(\varpi),b\|^2_{L^{2}(\mathbb{R})}=:\|\mathcal{T}(\varpi),b,I\|^2_{L^{2}(\mathbb{R})}+\|\mathcal{T}(\varpi),b,II\|^2_{L^{2}(\mathbb{R})}.
\end{align}
Furthermore, it is easy to see that
\begin{align}\label{eq:3.105}
\|\mathcal{T}(\varpi),b,II\|^2_{L^{2}(\mathbb{R})}\ls 2^{-2m}\|\varpi\|^2_{L^\infty(\mathbb{R})}.
\end{align}

For $\|\mathcal{T}(\varpi),b,I\|^2_{L^{2}(\mathbb{R})}$, which can be written as
\begin{align}\label{eq:3.106}
& \phi^2\left(\frac{\gamma'(2^{-j})}{2^{m+j-k}}\right)\int_{-\infty}^{\infty} \Bigg| \int_{-\infty}^{\infty}    e^{i2^m K(\xi(x,t),x,t)}\left(\frac{1}{K_1''(\xi(x,t),x,t)} \right)^{\frac{1}{2}}\Phi(\xi(x,t)) \\ \times & \upsilon(\xi(x,t),\eta)\varpi\left(2^{k-j}x+2^{k-j-m}r+2^k\gamma(2^{-j}t)\right) \rho(t)\,\textrm{d}t\Bigg|^2 \Delta(x)\,\textrm{d}x
=:\|T(\varpi)\|^2_{L^{2}(\mathbb{R})},\nonumber
\end{align}
where $T(\varpi)(x)$ is defined as
\begin{align}\label{eq:3.107}
&\phi\left(\frac{\gamma'(2^{-j})}{2^{m+j-k}}\right)\int_{-\infty}^{\infty}  \sqrt{\Delta(x)}  e^{i2^m K(\xi(x,t),x,t)}\left(\frac{1}{K_1''(\xi(x,t),x,t)} \right)^{\frac{1}{2}} \Phi(\xi(x,t)) \cdot\upsilon(\xi(x,t),\eta)\\
\times&\varpi\left(2^{k-j}x+2^{k-j-m}r+2^k\gamma(2^{-j}t)\right) \rho(t)\,\textrm{d}t.\nonumber
\end{align}
Let
\begin{align}\label{eq:3.0107}
y(x,t):=x+2^j\gamma(2^{-j}t)-t.
\end{align}
Note that from \eqref{eq:3.54}, \eqref{eq:3.102}, and \eqref{eq:3.98}, we have $K(\xi(x,t),x,t)=2^{k-m}\eta \gamma(-2^{-j}y(x,t))$. Let $Q_j(t):=\frac{2^j\gamma(2^{-j}t)}{\gamma'(2^{-j})}$ and $u:=Q_j(t)$. We rewrite $T(\varpi)(x)$ as
\begin{align*}
&\phi\left(\frac{\gamma'(2^{-j})}{2^{m+j-k}}\right)\int_{-\infty}^{\infty}  \sqrt{\Delta(x)}  e^{i2^m2^{k-m-j}\gamma'(2^{-j}) \eta Q_j\left(-y\left(x,Q_j^{-1}(u)\right)\right)}\left(\frac{1}{K_1''\left(\xi\left(x,Q_j^{-1}(u)\right),x,Q_j^{-1}(u)\right)} \right)^{\frac{1}{2}} \\ \times&\Phi\left(\xi\left(x,Q_j^{-1}(u)\right)\right)\cdot\upsilon\left(\xi\left(x,Q_j^{-1}(u)\right),\eta\right)\varpi\left(2^{k-j}x+2^{k-m-j}r+2^m2^{k-m-j}
\gamma'(2^{-j})u\right)\frac{\rho\left(Q_j^{-1}(u)\right)}{(Q_j)'\left(Q_j^{-1}(u)\right)}\,\textrm{d}u.\nonumber
\end{align*}
Furthermore, denote $\mathcal{K}(x,u)$ as
\begin{align}\label{eq:3.0109}
\frac{\phi\left(\frac{\gamma'(2^{-j})}{2^{m+j-k}}\right)\sqrt{\Delta(x)}}{\left[K_1''\left(\xi\left(x,Q_j^{-1}(u)\right),x,Q_j^{-1}(u)\right)\right]^{\frac{1}{2}}}
\Phi\left(\xi\left(x,Q_j^{-1}(u)\right)\right)\cdot\upsilon\left(\xi\left(x,Q_j^{-1}(u)\right),\eta\right) \frac{\rho\left(Q_j^{-1}(u)\right)}{(Q_j)'\left(Q_j^{-1}(u)\right)},
\end{align}
 $T(\varpi)(x)$ can be written as
\begin{align}\label{eq:3.109}
\int_{-\infty}^{\infty} \mathcal{K}(x,u)  e^{i2^m2^{k-m-j}\gamma'(2^{-j}) \eta Q_j\left(-y\left(x,Q_j^{-1}(u)\right)\right)}\varpi\left(2^{k-j}x+2^{k-m-j}r+2^m2^{k-m-j}\gamma'(2^{-j})u\right)\,\textrm{d}u.
\end{align}
By the $TT^*$ argument, $\|T(\varpi)\|^2_{L^{2}(\mathbb{R})}$ is equal to
\begin{align*}
&\int_{-\infty}^{\infty}\Bigg[\int_{-\infty}^{\infty}\int_{-\infty}^{\infty}\mathcal{K}(x,u_1)\overline{\mathcal{K}(x,u_2)} \\
\times &\varpi\left(2^{k-j}x+2^{k-m-j}r+2^m2^{k-m-j}\gamma'(2^{-j})u_1\right)\overline{\varpi\left(2^{k-j}x+2^{k-m-j}r+2^m2^{k-m-j}\gamma'(2^{-j})u_2\right)} \nonumber\\
\times &e^{i2^m2^{k-m-j}\gamma'(2^{-j}) \eta Q_j\left(-y\left(x,Q_j^{-1}(u_1)\right)\right)}e^{-i2^m2^{k-m-j}\gamma'(2^{-j}) \eta Q_j\left(-y\left(x,Q_j^{-1}(u_2)\right)\right)} \,\textrm{d}u_1 \,\textrm{d}u_2\Bigg]\,\textrm{d}x.\nonumber
\end{align*}
By changing the variables $u_1:=v+\tau$, $u_2:=v$ and $x:=x-\gamma'(2^{-j})v$, which can be further written as
\begin{align}\label{eq:3.110}
\int_{-\infty}^{\infty}\int_{-\infty}^{\infty}W_\tau(x)\left(\int_{-\infty}^{\infty} K_{\tau,x}(v)  e^{i2^m 2^{k-m-j}\gamma'(2^{-j})\eta P_{\tau,x}(v)}\,\textrm{d}v \right) \,\textrm{d}x\,\textrm{d}\tau,
\end{align}
where
$$
\begin{cases}K_{\tau,x}(v):=\mathcal{K}(x-\gamma'(2^{-j})v,v+\tau)\cdot\overline{\mathcal{K}(x-\gamma'(2^{-j})v,v)};\\
W_\tau(x):=\varpi\left(2^{k-j}x+2^{k-m-j}r+2^m 2^{k-m-j}\gamma'(2^{-j})\tau\right)\cdot\overline{\varpi\left(2^{k-j}x+2^{k-m-j}r\right)};\\
P_{\tau,x}(v):=  Q_j\left(-y\left(x-\gamma'(2^{-j})v,Q_j^{-1}(v+\tau)\right)\right)- Q_j\left(-y\left(x-\gamma'(2^{-j})v,Q_j^{-1}(v)\right)\right).
\end{cases}
$$
From the definition of $y(x,t)$ in \eqref{eq:3.0107}, we have $-y(x-\gamma'(2^{-j})v,Q_j^{-1}(v))=-x+Q_j^{-1}(v)$ and $-y(x-\gamma'(2^{-j})v,Q_j^{-1}(v+\tau))=-x- \gamma'(2^{-j})\tau +Q_j^{-1}(v+\tau)$. Furthermore, let $P(x,v):=Q_j(-x+Q_j^{-1}(v))$. Then
\begin{align*}
P_{\tau,x}(v)=P(x+\gamma'(2^{-j})\tau,v+\tau)-P(x,v).
\end{align*}

We now turn to $\|T(\varpi)\|^2_{L^{2}(\mathbb{R})}$. Indeed, from the definition of $\mathcal{K}$ in \eqref{eq:3.0109}, we have $Q_j^{-1}(v)\in \textrm{supp}~\rho$,  and $|v|\ls 1$ for $j$ large enough. At the same time, we have $|v+\tau|\ls 1$. Therefore, there exists a positive constant $\epsilon_0>1$ such that
\begin{align}\label{eq:3.111}
|\tau|\leq \epsilon_0.
\end{align}
Furthermore, we also have $|K_{\tau,x}(v)|\ls 1$, which leads to
\begin{align}\label{eq:3.112}
\left|\int_{-\infty}^{\infty} K_{\tau,x}(v)  e^{i2^m 2^{k-m-j}\gamma'(2^{-j})\eta P_{\tau,x}(v)}\,\textrm{d}v \right|\ls 1.
\end{align}
We now consider two cases, i.e., $\frac{\gamma'(2^{-j})}{|x|}\geq \frac{1}{2\epsilon_0}$ and $\frac{\gamma'(2^{-j})}{|x|}< \frac{1}{2\epsilon_0}$. For the first case, note that $\frac{\gamma'(2^{-j})}{2^{m+j-k}}\in \textrm{supp}~\phi$; it is easy to see that
\begin{align}\label{eq:3.113}
\|T(\varpi)\|^2_{L^{2}(\mathbb{R})}\ls 2^{m+j-k}\|\varpi\|^2_{L^\infty(\mathbb{R})}.
\end{align}
For the second case, i.e., $\frac{\gamma'(2^{-j})}{|x|}< \frac{1}{2\epsilon_0}$. We need the following Lemma \ref{lemma 3.2} and Proposition \ref{proposition 3.5}.

\begin{lemma}\label{lemma 3.2}
\cite[Lemma 2.1]{YL} Suppose $\phi$ is real-valued and smooth in $(a,b)$ and that both $|\phi'(x)|\geq \sigma_1$ and $|\phi''(x)|\leq \sigma_2$ for any $x\in (a,b)$. Then,
$$\left|\int_a^b e^{i\phi (t)}\,\textrm{d}t\right|\leq \frac{2}{\sigma_1}+(b-a)\frac{\sigma_2}{\sigma_1^2}.$$
\end{lemma}

\begin{proposition}\label{proposition 3.5}
Let all of the variables $x$, $\tau$, $v$, and $\eta$ be the same as in \eqref{eq:3.110}. If $\frac{\gamma'(2^{-j})}{|x|}< \frac{1}{2\epsilon_0}$, then there exists a positive constant $C$ such that
\begin{align}\label{eq:3.124}
\left|\frac{\textrm{d}P_{\tau,x}}{\textrm{d}v}(v)\right|\geq C |x||\tau| \quad \textrm{and}\quad \left|\frac{\textrm{d}^2P_{\tau,x}}{\textrm{d}^2v}(v)\right|\leq C |x|
\end{align}
where $j$ is large enough and $\epsilon_0$ can be found in \eqref{eq:3.111}.
\end{proposition}

Let us postpone the proof of Proposition \ref{proposition 3.5} for the moment.
By Lemma \ref{lemma 3.2}, similar to the Corollary on P. $334$ in Stein's book \cite{S}, with \eqref{eq:3.124} in Proposition \ref{proposition 3.5} and the fact that $\frac{\gamma'(2^{-j})}{2^{m+j-k}}\in \textrm{supp}~\phi$ and $ 2^{-6-\frac{2}{C_1}}\leq |\eta|\leq  2^{6+\frac{2}{C_1}}$, we have
\begin{align}\label{eq:3.114}
\left|\int_{-\infty}^{\infty} K_{\tau,x}(v)  e^{i2^m 2^{k-m-j}\gamma'(2^{-j})\eta P_{\tau,x}(v)}\,\textrm{d}v \right|\ls \frac{1}{2^m|x||\tau|}+\frac{1}{2^m|x||\tau|^2}.
\end{align}
From \eqref{eq:3.112} and \eqref{eq:3.114}, we can estimate
\begin{align}\label{eq:3.115}
\left|\int_{-\infty}^{\infty} K_{\tau,x}(v)  e^{i2^m 2^{k-m-j}\gamma'(2^{-j})\eta P_{\tau,x}(v)}\,\textrm{d}v \right|
\ls \left(\frac{1}{2^m|x||\tau|}\right)^{\frac{1}{4}}+\left(\frac{1}{2^m|x||\tau|^2}\right)^{\frac{1}{4}}.
\end{align}
Note that $|v|\ls 1$, from the definition of $\mathcal{K}$, we have $|x|\ls 1$ for $j$ large enough. From \eqref{eq:3.111} and \eqref{eq:3.115}, we have
\begin{align}\label{eq:3.116}
\|T(\varpi)\|^2_{L^{2}(\mathbb{R})}\ls \int_{|\tau|\leq \epsilon_0}\int_{|x|\ls 1}\left(\frac{1}{2^m|x||\tau|}\right)^{\frac{1}{4}}+\left(\frac{1}{2^m|x||\tau|^2}\right)^{\frac{1}{4}} \,\textrm{d}x\,\textrm{d}\tau\cdot\|\varpi\|^2_{L^\infty(\mathbb{R})}\ls 2^{-\frac{m}{4}}\|\varpi\|^2_{L^\infty(\mathbb{R})}.
\end{align}
By combining \eqref{eq:3.113}, \eqref{eq:3.106}, \eqref{eq:3.104}, \eqref{eq:3.105}, \eqref{eq:3.99} and \eqref{eq:3.100}, we have
\begin{align}\label{eq:3.120}
\|\mathcal{T}(\varpi)\|^2_{L^{2}(\mathbb{R})}\ls \left(2^{m+j-k}+2^{-\frac{m}{4}}+2^{-2m}+2^{-m}\right)\|\varpi\|^2_{L^\infty(\mathbb{R})}.
\end{align}
By \eqref{eq:3.95}, we have
\begin{align}\label{eq:3.121}
\left|\int_{-\infty}^{\infty}B^*_{j,m}(f,g)(x)\varpi(x)\,\textrm{d}x\right|\leq \left(2^{m+j-k}+2^{-\frac{m}{4}}+2^{-2m}+2^{-m}\right)^{\frac{1}{2}}\|g\|_{L^{2}(\mathbb{R})}\|\varpi\|_{L^\infty(\mathbb{R})}.
\end{align}
Let $\varpi:=\textrm{sgn} \left(B^*_{j,m}(f,g)\right)\cdot\chi_{[2^{k-j}N, 2^{k-j}(N+1)]}$. We see that
\begin{align}\label{eq:3.122}
\left\|B^*_{j,m,k}(f,g)\cdot\chi_{[2^{k-j}N, 2^{k-j}(N+1)]}\right\|_{L^{1}(\mathbb{R})}\leq \left(2^{m+j-k}+2^{-\frac{m}{4}}+2^{-2m}+2^{-m}\right)^{\frac{1}{2}}\|g\|_{L^{2}(\mathbb{R})}.
\end{align}
Furthermore, it is clear by the definition of $\mathcal{L}(\chi_{\mathbb{I}}\hat{f})$ that
\begin{eqnarray}\label{eq:3.123}
\mathcal{M}\ls
\left\{\aligned
2^{-\frac{m}{8}}
\|g\|_{L^{2}(\mathbb{R})}, \quad \textrm{if} \quad j-k+2m\leq 0,\\
\left(\max\left\{2^{m+j-k},2^{-\frac{m}{4}}\right\}\right)^{\frac{1}{2}}
\|g\|_{L^{2}(\mathbb{R})},  \quad\textrm{if} \quad  j-k+2m> 0.
\endaligned\right.
\end{eqnarray}

We now state the proof of Proposition \ref{proposition 3.5}.

\begin{proof}
By simple calculation, we have
\begin{align}\label{eq:3.125}
\frac{\textrm{d}P}{\textrm{d}v}(x,v)=(Q_j)'\left(Q_j^{-1}(v)-x\right)\cdot(Q_j^{-1})'(v)=\frac{(Q_j)'\left(Q_j^{-1}(v)-x\right)}{(Q_j)'\left(Q_j^{-1}(v)\right)}.
\end{align}
Here, we used the fact that $(Q_j^{-1})'(v)\cdot(Q_j)'(Q_j^{-1}(v))=1$. Furthermore, we have
\begin{align}\label{eq:3.126}
\frac{\textrm{d}^2P}{\textrm{d}x\textrm{d}v}(x,v)=-\frac{(Q_j)''\left(Q_j^{-1}(v)-x\right)}{(Q_j)'\left(Q_j^{-1}(v)\right)}
=-\frac{2^{-j}\gamma''\left(2^{-j}\left(Q_j^{-1}(v)-x\right) \right)}{\gamma'\left(2^{-j}Q_j^{-1}(v)\right)}.
\end{align}
From \eqref{eq:3.102}, we have $Q_j^{-1}(v)-x=t_0(\xi(x-\gamma'(2^{-j})v, Q_j^{-1}(v)),\eta)$, which implies that $\frac{1}{2}\leq|Q_j^{-1}(v)-x|\leq2$. It is also easy to see that $\frac{1}{2}\leq|Q_j^{-1}(v)|\leq2$. From \eqref{eq:1.00} and \eqref{eq:1.02}, $|\frac{\textrm{d}^2P}{\textrm{d}x\textrm{d}v}(x,v) |$ can be bounded by
\begin{align}\label{eq:3.127}
\left|\frac{2^{-j}}{2^{-j}\left(Q_j^{-1}(v)-x\right)} \frac{2^{-j}\left(Q_j^{-1}(v)-x\right)\gamma''\left(2^{-j}\left(Q_j^{-1}(v)-x\right) \right)}{\gamma'\left(2^{-j}\left(Q_j^{-1}(v)-x\right) \right)}        \frac{ \gamma'\left(2^{-j}\left(Q_j^{-1}(v)-x\right) \right)}{\gamma'\left(2^{-j}Q_j^{-1}(v)\right)}\right|
\ls 1.
\end{align}
On the other hand, by the mean value theorem,
\begin{align}\label{eq:3.128}
\frac{\textrm{d}^2P}{\textrm{d}^2v}(x,v)=-\frac{\gamma'\left(2^{-j}\left(Q_j^{-1}(v)-x\right)\right)\gamma'(2^{-j})2^{-j}}{\left(\gamma'\left(2^{-j}Q_j^{-1}(v)\right)\right)^2}\left(\frac{\gamma''}{\gamma'}\right)'\left(2^{-j}\left(Q_j^{-1}(v)-\vartheta x\right)\right)\cdot2^{-j}x,
\end{align}
where $\vartheta\in[0,1]$. Noting that $(\frac{\gamma''}{\gamma'})'(t)=-(\frac{\gamma'}{\gamma''})'(t) (\frac{t\gamma''(t)}{\gamma'(t)})^2\frac{1}{t^2}$, which, together with \eqref{eq:1.0}, \eqref{eq:1.00} and \eqref{eq:1.02}, shows that
\begin{align}\label{eq:3.129}
\left|\frac{\textrm{d}^2P}{\textrm{d}^2v}(x,v)\right|\approx |x|.
\end{align}
By the mean value theorem, we write
\begin{align*}
\frac{\textrm{d}P_{\tau,x}}{\textrm{d}v}(v)=
\frac{\textrm{d}^2P}{\textrm{d}x\textrm{d}v}\left(x+\vartheta_1\gamma'(2^{-j})\tau,v+\tau\right)\cdot\gamma'(2^{-j})\tau+\frac{\textrm{d}^2P}{\textrm{d}^2v}(x,v+\vartheta_2\tau)\cdot\tau,
\end{align*}
where $\vartheta_1, \vartheta_2\in[0,1]$. From \eqref{eq:3.127} and \eqref{eq:3.129}, note that $\frac{\gamma'(2^{-j})}{|x|}< \frac{1}{2\epsilon_0}<\frac{1}{2}$; it follows that
\begin{align}\label{eq:3.130}
\left|\frac{\textrm{d}P_{\tau,x}}{\textrm{d}v}(v)\right|\geq  \left|\frac{\textrm{d}^2P}{\textrm{d}^2v}(x,v+\vartheta_2\tau)\right|\cdot|\tau|-\left|\frac{\textrm{d}^2P}{\textrm{d}x\textrm{d}v}\left(x+\vartheta_1\gamma'(2^{-j})\tau,v+\tau\right)\right|\cdot\gamma'(2^{-j})|\tau|
\gs \frac{|x||\tau|}{2}.
\end{align}
This is the desired result regarding $|\frac{\textrm{d}P_{\tau,x}}{\textrm{d}v}(v)|$. For $|\frac{\textrm{d}^2P_{\tau,x}}{\textrm{d}^2v}(v)|$, from \eqref{eq:3.129} and $\frac{\gamma'(2^{-j})}{|x|}<\frac{1}{2}$, $|\tau|\ls 1$, it implies
\begin{align}\label{eq:3.131}
\left|\frac{\textrm{d}^2P_{\tau,x}}{\textrm{d}^2v}(v)\right|\leq \left|\frac{\textrm{d}^2P}{\textrm{d}^2v}\left(x+\gamma'(2^{-j})\tau,v+\tau\right)\right|+\left|\frac{\textrm{d}^2P}{\textrm{d}^2v}(x,v)\right|\ls |x|.
\end{align}
This finishes the proof of Proposition \ref{proposition 3.5}.
\end{proof}

{\bf Step~3: Estimates for \eqref{eq:3.82}}

From \eqref{eq:3.93} and \eqref{eq:3.123}, note that $\sigma=2^{-\frac{m}{16}}$ and $\|\chi_{\mathbb{I}}\hat{f}\|_{L^{2}(\mathbb{I})}\leq \|f\|_{L^{2}(\mathbb{R})}$; by Lemma \ref{lemma 3.1} and the definition of $\mathcal{L}(\chi_{\mathbb{I}}\hat{f})$, for any $\chi_{\mathbb{I}}\hat{f}\in L^{2}(\mathbb{I})$, it is easy to obtain \eqref{eq:3.82}.

\subsection{Proof of proposition \ref{proposition 3.3}}\label{subsection 5.3}

In Subsection \ref{subsection 5.1}, we obtained (see \eqref{eq:3.50})
\begin{align}\label{eq:3.132}
\left\|B^*_{j,m,k}(f,g)\cdot\chi_{\left[2^{k-j}N, 2^{k-j}(N+1)\right]}\right\|_{L^{1}(\mathbb{R})}\ls  2^{-\frac{2m+j-k}{6}}\|f\|_{L^{2}(\mathbb{R})}\|g\|_{L^{2}(\mathbb{R})}.
\end{align}
In Subsection \ref{subsection 5.2}, we bounded $\|B^*_{j,m,k}(f,g)\cdot\chi_{[2^{k-j}N, 2^{k-j}(N+1)]}\|_{L^{1}(\mathbb{R})}$  (see \eqref{eq:3.82}) by
\begin{eqnarray}\label{eq:3.133}
\left\{\aligned
2^{-\frac{m}{16}}\|f\|_{L^{2}(\mathbb{R})}\|g\|_{L^{2}(\mathbb{R})},\quad  \textrm{if} \quad j-k+2m\leq 0,\\
\Lambda_{j,m,k}\|f\|_{L^{2}(\mathbb{R})}\|g\|_{L^{2}(\mathbb{R})},\quad  \textrm{if} \quad  j-k+2m> 0,
\endaligned\right.\nonumber
\end{eqnarray}
where $\Lambda_{j,m,k}:=\max\{2^{\frac{15m}{16}+\frac{j}{2}-\frac{k}{2}},  2^{\frac{m}{16}}(\max\{2^{m+j-k},2^{-\frac{m}{4}}\} )^{\frac{1}{2}}  \}$. Our goal is to prove that (see Proposition \ref{proposition 3.3}, \eqref{eq:3.48})
\begin{align}\label{eq:3.134}
\left\|B^*_{j,m,k}(f,g)\cdot\chi_{\left[2^{k-j}N, 2^{k-j}(N+1)\right]}\right\|_{L^{1}(\mathbb{R})}\ls 2^{-\varepsilon_0' m}\|f\|_{L^{2}(\mathbb{R})}\|g\|_{L^{2}(\mathbb{R})}.
\end{align}
Indeed, for the case that $j-k+2m\leq 0$, \eqref{eq:3.134} is true with $\varepsilon_0':=\frac{1}{16}$. For the case that $j-k+2m>0$, we have
\begin{align}\label{eq:3.135}
\left\|B^*_{j,m,k}(f,g)\cdot\chi_{\left[2^{k-j}N, 2^{k-j}(N+1)\right]}\right\|_{L^{1}(\mathbb{R})}\ls 2^{\frac{15m}{16}+\frac{j}{2}-\frac{k}{2}}\|f\|_{L^{2}(\mathbb{R})}\|g\|_{L^{2}(\mathbb{R})}.
\end{align}
Thus, combined with \eqref{eq:3.132}, we have $\|B^*_{j,m}(f,g)\cdot\chi_{(\frac{2^m}{\gamma'(2^{-j})}N, \frac{2^m}{\gamma'(2^{-j})}(N+1))}\|_{L^{1}(\mathbb{R})}$ is bounded by
\begin{align}\label{eq:3.136} \left(2^{\frac{15m}{16}+\frac{j}{2}-\frac{k}{2}}\right)^{\frac{1}{4}}\cdot\left(2^{-\frac{2m+j-k}{6}}\right)^{\frac{3}{4}}\|f\|_{L^{2}(\mathbb{R})}\|g\|_{L^{2}(\mathbb{R})}
\ls 2^{-\frac{m}{64}}\|f\|_{L^{2}(\mathbb{R})}\|g\|_{L^{2}(\mathbb{R})}.
\end{align}
This is \eqref{eq:3.134} with $\varepsilon_0':=\frac{1}{64}$. These are the proof of Proposition \ref{proposition 3.3}.

We obtain the $L^2(\mathbb{R})\times L^2(\mathbb{R})\rightarrow L^1(\mathbb{R})$  estimate for  $H^3_{\gamma}(f,g)$.

\section{Weak-$L^p(\mathbb{R})\times L^q(\mathbb{R})\rightarrow L^r(\mathbb{R})$ boundedness of $H_{m}(f,g)$}\label{section 6}

We begin with a lemma which can be found in \cite[Lemma 5.4]{AHMTT}:

\begin{lemma}\label{lemma 6.1}
\cite[Lemma 5.4]{AHMTT} Let $0<p<\infty$ and $A>0$. Then, the following statements are equivalent:
\begin{enumerate}
  \item[\rm(i)] $\|f\|_{L^{p,\infty}(\mathbb{R})}\leq A$.
  \item[\rm(ii)] For every Lebesgue measurable set $E$ with $0<|E|<\infty$, there exists a subset $E'\subset E$ with $|E'|\geq \frac{|E|}{2}$ such that $|\langle f,\chi_{E'}\rangle|\ls A|E|^{\frac{1}{p'}}$.
\end{enumerate}
\end{lemma}

Recall that $\psi_\lambda(\xi)=2^\lambda\psi(2^\lambda\xi)$; we rewrite $H_{\gamma}^3(f,g)$ as
\begin{align}\label{eq:6.1}
H_{\gamma}^3(f,g)(x)=\sum_{j\in \mathbb{Z}}\sum_{m\in \mathbb{N}}\sum_{k\in \mathbb{Z}}\phi\left(\frac{\gamma'(2^{-j})}{2^{m+j-k}}\right)\int_{-\infty}^{\infty} \check{\phi}_{m+j}\ast f\left(x-2^{-j} t\right) \check{\phi}_k\ast g\left(x-\gamma(2^{-j}t)\right) \rho(t)\,\textrm{d}t.
\end{align}
Let
\begin{align}\label{eq:6.2}
H_m(f,g)(x):=\sum_{j\in \mathbb{Z}}\sum_{k\in \mathbb{Z}}\phi\left(\frac{\gamma'(2^{-j})}{2^{m+j-k}}\right)\int_{-\infty}^{\infty} \check{\phi}_{m+j}\ast f\left(x-2^{-j} t\right) \check{\phi}_k\ast g\left(x-\gamma(2^{-j}t)\right) \rho(t)\,\textrm{d}t.
\end{align}
In this section, we show that
\begin{align}\label{eq:6.3}
\|H_m(f,g)\|_{L^{r,\infty}(\mathbb{R})}\ls m\|f\|_{L^{p}(\mathbb{R})}\|g\|_{L^{q}(\mathbb{R})}
\end{align}
for $r>\frac{1}{2}$. We put the absolute value inside the integral and define
\begin{align}\label{eq:6.4}
|H_m|(f,g)(x):=\sum_{j\in \mathbb{Z}}\sum_{k\in \mathbb{Z}}\phi\left(\frac{\gamma'(2^{-j})}{2^{m+j-k}}\right)\int_{-\infty}^{\infty} \left|\check{\phi}_{m+j}\ast f\left(x-2^{-j} t\right) \check{\phi}_k\ast g\left(x-\gamma(2^{-j}t)\right) \rho(t)\right|\,\textrm{d}t.
\end{align}
Indeed, our goal is to obtain
\begin{align}\label{eq:6.5}
\||H_m|(f,g)\|_{L^{r,\infty}(\mathbb{R})}\ls m\|f\|_{L^{p}(\mathbb{R})}\|g\|_{L^{q}(\mathbb{R})}.
\end{align}
The main tool is Lemma \ref{lemma 6.1} above. Therefore, we may assume that $f:=\chi_{F_1}$ and $g:=\chi_{F_2}$ throughout this section, where $F_1$ and $F_2$ are a Lebesgue measurable set satisfying $0<|F_1|, |F_2|<\infty$. Furthermore, let us set $F_3$ as a Lebesgue measurable set with $0<|F_3|<\infty$ and define
\begin{align}\label{eq:6.6}
\Omega:=\left\{x\in \mathbb{R}:\ M\chi_{F_1}(x)>C \frac{|F_1|}{|F_3|}\right\}\bigcup \left\{x\in \mathbb{R}:\ M\chi_{F_2}(x)>C \frac{|F_2|}{|F_3|}\right\},
\end{align}
where $M$ is the \emph{Hardy-Littlewood maximal function}. From the weak-$(1,1)$ boundedness of the uncentered Hardy-Littlewood maximal function $M$, we may assume that $|\Omega|<\frac{|F_3|}{2}$ with $C$ large enough. Let $F'_3:=F_3\backslash \Omega$; we have $|F'_3|>\frac{|F_3|}{2}$. By Lemma \ref{lemma 6.1}, for $r>\frac{1}{2}$, it suffices to prove
\begin{align}\label{eq:6.7}
|\langle |H_m|(f,g),\chi_{E'}\rangle|\ls m|F_1|^{\frac{1}{p}}|F_2|^{\frac{1}{q}}|F_3|^{\frac{1}{r'}}.
\end{align}

Recall that $\psi$ is a nonnegative Schwartz function such that $\hat{\psi}$ is supported on $\{t\in \mathbb{R}:\ |t|\leq \frac{1}{100}\}$ and satisfies $\hat{\psi}(0)=1$, and $\psi_\lambda(x)=2^\lambda\psi(2^\lambda x) $. Furthermore, let $\Omega_\lambda:=\left\{x\in \Omega:\ \textrm{dist}(x, \Omega^{\complement})\geq 2^{-\lambda} \right\}$ and $\tilde{\psi}_\lambda:=\chi_{\Omega_\lambda^{\complement}}\ast \psi_\lambda$. Therefore, we may split $\check{\phi}_{m+j}\ast f$ as:
\begin{align*}
F_{m,j}(x,t):=\left(\tilde{\psi}_{m+j}\cdot\check{\phi}_{m+j}\ast f\right)\left(x-2^{-j} t\right)\quad \textrm{and}\quad F^{\complement}_{m,j}(x,t):=\left(\left(1-\tilde{\psi}_{m+j}\right)\cdot\check{\phi}_{m+j}\ast f\right)\left(x-2^{-j} t\right).
\end{align*}
Similarly, we split $\check{\phi}_k\ast g$ as:
\begin{align*}
G_{k,j}(x,t):=\left(\tilde{\psi}_{k}\cdot\check{\phi}_k\ast g\right)\left(x-\gamma(2^{-j}t)\right)\quad \textrm{and}\quad G^{\complement}_{k,j}(x,t):=\left(\left(1-\tilde{\psi}_{k}\right)\cdot\check{\phi}_k\ast g\right)\left(x-\gamma(2^{-j}t)\right).
\end{align*}
Then, $|H_m|(f,g)$ can be split into three error terms:
\begin{align}\label{eq:6.9}
\begin{cases}|H^1_m|(f,g)(x):=\sum_{j\in \mathbb{Z}}\sum_{k\in \mathbb{Z}}\phi\left(\frac{\gamma'(2^{-j})}{2^{m+j-k}}\right)\int_{-\infty}^{\infty} \left|F^{\complement}_{m,j}(x,t)G_{k,j}(x,t) \rho(t)\right|\,\textrm{d}t;\\
|H^2_m|(f,g)(x):=\sum_{j\in \mathbb{Z}}\sum_{k\in \mathbb{Z}}\phi\left(\frac{\gamma'(2^{-j})}{2^{m+j-k}}\right)\int_{-\infty}^{\infty} \left|F^{\complement}_{m,j}(x,t)G^{\complement}_{k,j}(x,t) \rho(t)\right|\,\textrm{d}t;\\
|H^3_m|(f,g)(x):=\sum_{j\in \mathbb{Z}}\sum_{k\in \mathbb{Z}}\phi\left(\frac{\gamma'(2^{-j})}{2^{m+j-k}}\right)\int_{-\infty}^{\infty} \left|F_{m,j}(x,t)G^{\complement}_{k,j}(x,t) \rho(t)\right|\,\textrm{d}t,
\end{cases}
\end{align}
and a major term:
\begin{align}\label{eq:6.12}
|H^4_m|(f,g)(x):=\sum_{j\in \mathbb{Z}}\sum_{k\in \mathbb{Z}}\phi\left(\frac{\gamma'(2^{-j})}{2^{m+j-k}}\right)\int_{-\infty}^{\infty} \left|F_{m,j}(x,t)G_{k,j}(x,t) \rho(t)\right|\,\textrm{d}t.
\end{align}

\subsection{Error terms $|H^1_m|(f,g)$, $|H^2_m|(f,g)$, $|H^3_m|(f,g)$}\label{subsection 6.1}

In this subsection, we want to prove that
\begin{align}\label{eq:6.13}
|\langle |H^i_m|(f,g),\chi_{E'}\rangle|\ls m|F_1|^{\frac{1}{p}}|F_2|^{\frac{1}{q}}|F_3|^{\frac{1}{r'}}
\end{align}
for all $i\in\{1,2,3\}$ and $r>\frac{1}{2}$.

We will set up \eqref{eq:6.13} for $|H^1_m|(f,g)$. The proofs for $|H^2_m|(f,g)$ and $|H^3_m|(f,g)$ are similar. Let $K\in \mathbb{N}$ be large enough; for any $x,y\in \mathbb{R}$ and any Lebesgue measurable set $E \subset \mathbb{R}$, define
\begin{align*}
\delta_{j,K}(x, y):=\frac{1}{\left(1+2^{j+m}|x-y|\right)^{K}}\quad \textrm{and}\quad \delta_{j,K}(x, E):=\frac{1}{\left(1+2^{j+m}\textrm{dist}(x, E)\right)^K}.
\end{align*}
Noting that $\tilde{\psi}_{m+j}(x-2^{-j} t)=\int_{\Omega_{j+m}^{\complement}}2^{j+m}\psi(2^{j+m}(x-2^{-j} t-y))\,\textrm{d}y$ and $\int_{\Omega_{j+m}^{\complement}\bigcup \Omega_{j+m}}2^{j+m}\psi(2^{j+m}(x-2^{-j} t-y))\,\textrm{d}y=1$ since $\hat{\psi}(0)=1$, we have
$(1-\tilde{\psi}_{m+j})(x-2^{-j} t)=\int_{\Omega_{j+m}}2^{j+m}\psi(2^{j+m}(x-2^{-j} t-y))\,\textrm{d}y\ls\int_{\Omega_{j+m}}2^{j+m} \delta_{j,K}(x-2^{-j} t, y) \,\textrm{d}y$. Therefore, $|\langle |H^1_m|(f,g),\chi_{E'}\rangle|$ can be bounded by
\begin{align}\label{eq:6.16}
\sum_{j\in \mathbb{Z}}\sum_{k\in \mathbb{Z}}\phi\left(\frac{\gamma'(2^{-j})}{2^{m+j-k}}\right)&\int_{\Omega^{\complement}} \int_{-\infty}^{\infty}\int_{\Omega_{j+m}} 2^{j+m} \delta_{j,K}(x-2^{-j} t, y) \\
  \times&\left|\check{\phi}_{m+j}\ast f\left(x-2^{-j} t\right)\check{\phi}_k\ast g\left(x-\gamma(2^{-j}t)\right) \rho(t)\right|\,\textrm{d}y\,\textrm{d}t\,\textrm{d}x.\nonumber
\end{align}
There are two cases: $x-2^{-j} t\in \Omega^{\complement}$ and $x-2^{-j} t\in \Omega$.

{\bf Case I:} $x-2^{-j} t\in \Omega^{\complement}$

It is easy to see that $Mf(x)=M\chi_{F_1}(x)\leq 1$ for all $x\in \mathbb{R}$; from the definition of $\Omega$, it implies $$|\check{\phi}_{m+j}\ast f(x-2^{-j} t)|\ls \left(\frac{|F_1|}{|F_3|}\right)^{\frac{1}{p}}.$$ Furthermore, let us set $\textbf{tr}_j(t):=2^{-j}t-\gamma(2^{-j}t)$ and $u:=x-2^{-j}t$ for any given $t$, then $x-\gamma(2^{-j}t)=u+\textbf{tr}_j(t)$. We also have $$\check{\phi}_k\ast g\ls M\chi_{F_2}\quad\text{and}\quad\delta_{j,K}(x-2^{-j}t, y)\leq \delta_{j,\frac{K}{2}}(y,\Omega^{\complement})\cdot\delta_{j,\frac{K}{2}}(u, y).$$ Therefore, \eqref{eq:6.16} is dominated by
\begin{align*}
\left(\frac{|F_1|}{|F_3|}\right)^{\frac{1}{p}}\sum_{j\in \mathbb{Z}}\sum_{k\in \mathbb{Z}}\phi\left(\frac{\gamma'(2^{-j})}{2^{m+j-k}}\right)\int_{\Omega_{j+m}} \delta_{j,\frac{K}{2}}(y,\Omega^{\complement})\int_{\Omega^{\complement}} & 2^{j+m}\delta_{j,\frac{K}{2}}(u, y)\left(\int_{-\infty}^{\infty}M\chi_{F_2}\left(u+\textbf{tr}_j(t)\right) |\rho(t)|\,\textrm{d}t\right)\,\textrm{d}u\,\textrm{d}y.
\end{align*}
For any given $u\in \Omega^{\complement}$, let $\tau:=u+\textbf{tr}_j(t)$. Then, for $j>0$ large enough, we have $$\textrm{d}t\leq\frac{\textrm{d}\tau}{2^{-j}-2^{-j}\gamma'(2\cdot2^{-j})}\ls2^j\,\textrm{d}\tau.$$ On the other hand, note that $\gamma(t)$ and $\frac{\gamma(t)}{t}$ are increasing on $(0,\infty)$ and \eqref{eq:1.03}; we have $$|2^{-j}t-\gamma(2^{-j}t)|\leq 2\cdot2^{-j}+2^{1+\frac{2}{C_1}}\gamma(2^{-j})\ls 2^{-j}.$$
Therefore, by the fact that $\left(M\chi_{F_2}\right)^{\frac{1}{q}}$ is an $A_1$ weight, we conclude that
\begin{align}\label{eq:6.019}
\int_{-\infty}^{\infty}M\chi_{F_2}\left(u+\textbf{tr}_j(t)\right) |\rho(t)|\,\textrm{d}t \ls 2^j\int_{u-2^{-j}}^{u+2^{-j}}\left(M\chi_{F_2}(\tau)\right)^{\frac{1}{q}}\,\textrm{d}\tau \ls \left(M\chi_{F_2}(u)\right)^{\frac{1}{q}}.
\end{align}
Note that $u\in \Omega^{\complement}$, which further implies that $$\int_{-\infty}^{\infty}M\chi_{F_2}(u+\textbf{tr}_j(t)) |\rho(t)|\,\textrm{d}t \ls \left(\frac{|F_2|}{|F_3|}\right)^{\frac{1}{q}}.$$ Then, \eqref{eq:6.16} can be bounded by
\begin{align}\label{eq:6.19}
\left(\frac{|F_1|}{|F_3|}\right)^{\frac{1}{p}}\left(\frac{|F_2|}{|F_3|}\right)^{\frac{1}{q}}\sum_{j\in \mathbb{Z}}\sum_{k\in \mathbb{Z}}\phi\left(\frac{\gamma'(2^{-j})}{2^{m+j-k}}\right)\int_{\Omega_{j+m}} \delta_{j,\frac{K}{2}}(y,\Omega^{\complement})\int_{\Omega^{\complement}}  2^{j+m}\delta_{j,\frac{K}{2}}(u, y)\,\textrm{d}u\,\textrm{d}y.
\end{align}

When $j<0$ and $|j|$ are large enough, note that $\frac{\gamma(t)}{t}$ is strictly increasing on $(0,\infty)$ and $\gamma'(0)=0$; we have $\frac{\gamma(2^{-j})}{2^{-j}}>\frac{\gamma(J_0)}{J_0}$ with $\frac{\gamma(J_0)}{J_0}>2^{2+\frac{2}{C_1}}/(1+\frac{1}{2C_2})$ for $|j|$ large enough. From \eqref{eq:1.03} and \eqref{eq:1.01}, we have $$\textrm{d}t\leq\frac{\textrm{d}\tau}{\left|\frac{2^{-j}t\gamma'(2^{-j}t)}{\gamma(2^{-j}t)}\cdot\frac{\gamma(2^{-j}t)}{t}\right|-2^{-j}}
\leq\frac{\textrm{d}\tau}{(1+\frac{1}{2C_2})2^{-2-\frac{2}{C_1}}\gamma(2^{-j})-\frac{J_0\gamma(2^{-j})}{\gamma(J_0)}}\ls\frac{\textrm{d}\tau}{\gamma(2^{-j})}.$$ It is also easy to see that $|2^{-j}t-\gamma(2^{-j}t)|\ls \gamma(2^{-j})$. Therefore, we can control \eqref{eq:6.16} by \eqref{eq:6.19}, as in the case $j>0$.

As in \eqref{eq:3.6} and noting that $\int_{\Omega^{\complement}}   \frac{2^{j+m}}{(1+2^{j+m}|u-y|)^{\frac{K}{2}}} \,\textrm{d}u\ls1$, \eqref{eq:6.19}  can be bounded by
\begin{align*}
&\left(\frac{|F_1|}{|F_3|}\right)^{\frac{1}{p}}\left(\frac{|F_2|}{|F_3|}\right)^{\frac{1}{q}}\sum_{j\in \mathbb{Z}}\int_{\Omega_{j+m}} \frac{1}{\left(1+2^{j+m}\textrm{dist}(y, \Omega^{\complement})\right)^{\frac{K}{2}}} \int_{\Omega^{\complement}}   \frac{2^{j+m}}{\left(1+2^{j+m}|u-y|\right)^{\frac{K}{2}}} \,\textrm{d}u\,\textrm{d}y\\
\leq&\left(\frac{|F_1|}{|F_3|}\right)^{\frac{1}{p}}\left(\frac{|F_2|}{|F_3|}\right)^{\frac{1}{q}}\sum_{j\in \mathbb{Z}}\int_{\Omega_{j+m}} \frac{1}{\left(1+2^{j+m}\textrm{dist}(y, \Omega^{\complement})\right)^{\frac{K}{2}}} \,\textrm{d}y
\leq\left(\frac{|F_1|}{|F_3|}\right)^{\frac{1}{p}}\left(\frac{|F_2|}{|F_3|}\right)^{\frac{1}{q}}|\Omega|.\nonumber
\end{align*}
This is the desired estimate, since $|\Omega|<\frac{|F_3|}{2}$. The last inequality  above  follows from the fact
\begin{align*}
\sum_{j\in \mathbb{Z}}\int_{\Omega_{j+m}} \frac{1}{\left(1+2^{j+m}\textrm{dist}(y, \Omega^{\complement})\right)^{\frac{K}{2}}} \,\textrm{d}y
\leq& \int_{\Omega}\sum_{l\in \mathbb{N}}\sum_{j\in \mathbb{Z}:\ 2^{-j-m+l}\leq \textrm{dist}(y, \Omega^{\complement})\leq 2^{-j-m+l+1}  } \frac{1}{\left(1+2^{j+m}\textrm{dist}(y, \Omega^{\complement})\right)^{\frac{K}{2}}} \,\textrm{d}y\\
 \ls& |\Omega|\sum_{l\in \mathbb{N}} \frac{1}{\left(1+2^{l}\right)^{\frac{K}{2}}}
  \ls |\Omega|.\nonumber
\end{align*}
Here, we used the fact that $\sharp\left\{j\in \mathbb{Z}:\ 2^{-j-m+l}\leq \textrm{dist}(y, \Omega^{\complement})\leq 2^{-j-m+l+1}\right\}\ls 1$ for any given $y\in \Omega$ and $l,m\in \mathbb{N}$.

{\bf Case II:} $x-2^{-j} t\in \Omega$

In this case, we use the Whitney decomposition theorem to the open set $\Omega$ (see, for example, \cite[P. 609]{G}). Let $\mathcal{F}$ be a collection of pairwise disjoint dyadic interval $J$'s such that $\Omega=\bigcup_{J\in \mathcal{F}}J$, Then, for each $J\in \mathcal{F}$, we have $|J|\leq\textrm{dist}(J, \Omega^{\complement})\leq 4 |J|$; thus, $10J$ meets $\Omega^{\complement}$. Furthermore, for each $J\in \mathcal{F}$ and $i\in\{1,2\}$, we have
\begin{align*}
\frac{1}{|10J|}\int_{10J}\chi_{F_i}(x)\,\textrm{d}x\ls \frac{|F_i|}{|F_3|}.
\end{align*}
We now introduce the following two lemmas which can be found in \cite{LX}:

\begin{lemma}\label{lemma 6.2}
\cite[Lemma 8.1]{LX} Let $I_1$, $I_2$ be two intervals in $\mathcal{F}$. Suppose that $a\in I_1$ and $b\in I_2$. If $\textrm{dist}(I_1,I_2)\geq 100 \min\{|I_1|,|I_2|\}$, then
$$\delta_{j,K}(a,b)\ls\delta_{j,\frac{K}{2}}(a,\Omega^{\complement})\cdot\delta_{j,\frac{K}{2}}(b,\Omega^{\complement}). $$
\end{lemma}

\begin{lemma}\label{lemma 6.3}
\cite[Lemma 8.2]{LX} Let $I_1$, $I_2$ be two intervals in $\mathcal{F}$. Suppose that $\textrm{dist}(I_1,I_2)\leq 100 \min\{|I_1|,|I_2|\}$, then
$$\frac{|I_2|}{2000}\leq |I_1|\leq 2000 |I_2|. $$
\end{lemma}

We now turn to \eqref{eq:6.16}; we rewrite it as
\begin{align}\label{eq:6.21}
\sum_{j\in \mathbb{Z}}\sum_{k\in \mathbb{Z}}\phi\left(\frac{\gamma'(2^{-j})}{2^{m+j-k}}\right)\sum_{I_1,I_2\in \mathcal{F}} &\int_{\Omega^{\complement}} \int_{-\infty}^{\infty}\int_{\Omega_{j+m}} 2^{j+m} \chi_{I_1}(y)\chi_{I_2}(x-2^{-j} t)\delta_{j,K}(x-2^{-j} t, y) \\
  \times&\left|\check{\phi}_{m+j}\ast f\left(x-2^{-j} t\right)\check{\phi}_k\ast g\left(x-\gamma(2^{-j}t)\right) \rho(t)\right|\,\textrm{d}y\,\textrm{d}t\,\textrm{d}x.\nonumber
\end{align}

{\bf Case IIa:} $\textrm{dist}(I_1,I_2)\geq 100 \min\{|I_1|,|I_2|\}$

In this case, by Lemma \ref{lemma 6.2}, we have $\delta_{j,K}(x-2^{-j} t, y)\ls\delta_{j,\frac{K}{2}}(x-2^{-j} t,\Omega^{\complement})\cdot\delta_{j,\frac{K}{2}}(y,\Omega^{\complement})$. Then, \eqref{eq:6.21} can be bounded by
\begin{align}\label{eq:6.22}
&\sum_{j\in \mathbb{Z}}\sum_{k\in \mathbb{Z}}\phi\left(\frac{\gamma'(2^{-j})}{2^{m+j-k}}\right)\sum_{I_1,I_2\in \mathcal{F}} \int_{\Omega^{\complement}} \int_{-\infty}^{\infty}\left(\int_{\Omega_{j+m}} 2^{j+m} \chi_{I_1}(y)\delta_{j,\frac{K}{2}}(y,\Omega^{\complement}) \,\textrm{d}y\right)\\
  \times&\chi_{I_2}(x-2^{-j} t)\cdot\delta_{j,\frac{K}{2}}(x-2^{-j} t,\Omega^{\complement})\left|\check{\phi}_{m+j}\ast f\left(x-2^{-j} t\right)\check{\phi}_k\ast g\left(x-\gamma(2^{-j}t)\right) \rho(t)\right|\,\textrm{d}t\,\textrm{d}x.\nonumber
\end{align}
For each $x-2^{-j} t\in I_2$, we choose $z\in \Omega^{\complement}$ such that $\textrm{dist}(x-2^{-j} t, \Omega^{\complement})\approx |x-2^{-j} t-z|$. The definition of $\Omega$ implies that $\delta_{j,\frac{K}{2}}(x-2^{-j} t,\Omega^{\complement})\cdot|\check{\phi}_{m+j}\ast f(x-2^{-j} t)|$ can be bounded by
\begin{align}\label{eq:6.023}
&\int_{-\infty}^{\infty} f(w)\frac{2^{j+m}}{\left(1+2^{j+m}|x-2^{-j} t-w|\right)^{\frac{K}{2}}}\,\textrm{d}w\cdot \frac{1}{\left(1+2^{j+m}|x-2^{-j} t-z|\right)^{\frac{K}{2}}}\\
\ls &\int_{-\infty}^{\infty} f(w)\frac{2^{j+m}}{\left(1+2^{j+m}|z-w|\right)^{\frac{K}{2}}}\,\textrm{d}w
\ls Mf(z)
\ls\left(\frac{|F_1|}{|F_3|}\right)^{\frac{1}{p}}.\nonumber
\end{align}
Furthermore, let $u:=x-2^{-j}t$; we have $x-\gamma(2^{-j}t)=u+\textbf{tr}_j(t)$. Note that  $\check{\phi}_k\ast g\ls M\chi_{F_2}$; we can bound \eqref{eq:6.22} by
\begin{align}\label{eq:6.24}
\left(\frac{|F_1|}{|F_3|}\right)^{\frac{1}{p}}\sum_{j\in \mathbb{Z}}\sum_{k\in \mathbb{Z}}\phi\left(\frac{\gamma'(2^{-j})}{2^{m+j-k}}\right)\sum_{I_1,I_2\in \mathcal{F}} \int_{-\infty}^{\infty} \int_{-\infty}^{\infty}&\left[\int_{\Omega_{j+m}} 2^{j+m} \chi_{I_1}(y)\delta_{j,\frac{K}{2}}(y,\Omega^{\complement}) \,\textrm{d}y\right]\\
  \times&\chi_{I_2}(u)\cdot M\chi_{F_2}\left(u+\textbf{tr}_j(t)\right)| \rho(t)|\,\textrm{d}t\,\textrm{d}u.\nonumber
\end{align}
Note that \eqref{eq:6.019} holds for all $j\in \mathbb{Z}$; we have
\begin{align}\label{eq:6.26}
&\sum_{I_2\in \mathcal{F}} \int_{-\infty}^{\infty}\int_{-\infty}^{\infty}\chi_{I_2}(u)\cdot M\chi_{F_2}\left(u+\textbf{tr}_j(t)\right)| \rho(t)|\,\textrm{d}t\,\textrm{d}u\\
\ls&\sum_{I_2\in \mathcal{F}} |10I_2| \inf_{w\in 10I_2} M\left(\left(M\chi_{F_2}\right)^{\frac{1}{q}}\right)(w)
\ls\sum_{I_2\in \mathcal{F}} |10I_2| \inf_{w\in 10I_2} \left(M\chi_{F_2}\right)^{\frac{1}{q}}(w)
\ls|\Omega| \left(\frac{|F_2|}{|F_3|}\right)^{\frac{1}{q}}.\nonumber
\end{align}

On the other hand, $|I_1|\leq\textrm{dist}(I_1, \Omega^{\complement})\leq 4 |I_1|$ implies that $|I_1|\leq\textrm{dist}(y, \Omega^{\complement})\leq 5 |I_1|$ for all $y\in I_1$. Combining it with the fact that $\sharp\left\{j\in \mathbb{Z}:\ 2^{-j-m+l}\leq \textrm{dist}(y, \Omega^{\complement})\leq 2^{-j-m+l+1}\right\}\ls 1$ and $I_1$ is a dyadic interval, we obtain
\begin{align}\label{eq:6.27}
&\sum_{j\in \mathbb{Z}}\sum_{I_1\in \mathcal{F}}\int_{\Omega_{j+m}} 2^{j+m} \chi_{I_1}(y)\cdot\delta_{j,\frac{K}{2}}(y,\Omega^{\complement}) \,\textrm{d}y\\
\ls& \sum_{I_1\in \mathcal{F}}\int_{I_1} \sum_{l\in \mathbb{N}}\sum_{j\in \mathbb{Z}:\ 2^{-j-m+l}\leq \textrm{dist}(y, \Omega^{\complement})\leq 2^{-j-m+l+1}  } \frac{1}{\left(1+2^{j+m}\cdot2^{-j-m+l}\right)^{\frac{K}{4}}}\frac{2^{j+m}}{\left(1+2^{j+m}|I_1|\right)^{\frac{K}{4}}}
\,\textrm{d}y\nonumber\\
\ls& \sum_{l\in \mathbb{N}} \frac{1}{\left(1+2^{l}\right)^{\frac{K}{4}}}\sum_{j\in \mathbb{Z}:\ \sharp\{j\}\ls 1 } \sum_{I_1\in \mathcal{F}} \frac{2^{j+m}|I_1|}{\left(1+2^{j+m}|I_1|\right)^{\frac{K}{4}}}
\ls 1 .\nonumber
\end{align}
Therefore, for this case, we also have $$|\langle |H^1_m|(f,g),\chi_{E'}\rangle|\ls \left(\frac{|F_1|}{|F_3|}\right)^{\frac{1}{p}}\left(\frac{|F_2|}{|F_3|}\right)^{\frac{1}{q}}|\Omega|.$$

{\bf Case IIb:} $\textrm{dist}(I_1,I_2)\leq 100 \min\{|I_1|,|I_2|\}$

By Lemma \ref{lemma 6.3}, we have $\frac{|I_2|}{2000}\leq |I_1|\leq 2000 |I_2|$ in this case. We may assume that $\max\{|I_1|,|I_2|\}\leq 2^{50}2^{-j}$. Otherwise, we have $|I_2|\geq2^{8}2^{-j}$. Since $x-2^{-j} t\in I_2$ and $x\in \Omega^{\complement}$, we have $$\textrm{dist}(x, I_2)\leq |x-(x-2^{-j} t)|\leq 2\cdot2^{-j}\leq \frac{|I_2|}{2^7}$$ which further implies that $\textrm{dist}(\Omega^{\complement}, I_2)\leq \frac{|I_2|}{2^7}$. This is a contradiction from $I_2\in \mathcal{F}$.

 On the other hand, we may also assume that $|I_1|\geq 2^{-10}2^{-j-m}$. If not, $y\in I_1\bigcap \Omega_{j+m}$  implies that $2^{-j-m}\leq \textrm{dist}(y,\Omega^{\complement})\leq 4|I_1|\leq 2^{-8}2^{-j-m}$. It is a contradiction. Therefore, for any $I_1\in \mathcal{F}$, we have
\begin{align}\label{eq:6.30}
\frac{1}{2^{m+10}|I_1|}\leq 2^j\leq \frac{2^{50}}{|I_1|}.
\end{align}
Then, for any given $I_1\in \mathcal{F}$, there are at most $10m$ many $j$'s when $m$ is large enough.

Furthermore, for any given $I_1\in \mathcal{F}$, from $\frac{|I_2|}{2000}\leq |I_1|\leq 2000 |I_2|$, the number of interval $I_2$'s with $\textrm{dist}(I_1,I_2)\leq 100 \min\{|I_1|,|I_2|\}$ is finite and independent of $m$. Let $u:=x-2^{-j}t$, then $x-\gamma(2^{-j}t)=u+\textbf{tr}_j(t)$. Without loss of generality, \eqref{eq:6.21} can be controlled by
\begin{align}\label{eq:6.32}
&\sum_{I_1\in \mathcal{F}} \sum_{j\in \mathbb{Z}:\ \frac{1}{2^{m+10}|I_1|}\leq 2^j\leq \frac{2^{50}}{|I_1|} }\sum_{k\in \mathbb{Z}}\phi\left(\frac{\gamma'(2^{-j})}{2^{m+j-k}}\right)\int_{-\infty}^{\infty} \int_{-\infty}^{\infty}\int_{\Omega_{j+m}} 2^{j+m} \chi_{I_1}(y)\chi_{I_1}(u)\chi_{\Omega^{\complement}}(u+2^{-j}t) \\
  \times&\delta_{j,K}(u, y)\left|\check{\phi}_{m+j}\ast f(u)\cdot\check{\phi}_k\ast g\left(u+\textbf{tr}_j(t)\right) \rho(t)\right|\,\textrm{d}y\,\textrm{d}t\,\textrm{d}u.\nonumber
\end{align}
As in \eqref{eq:6.023}, we have $\delta_{j,\frac{K}{2}}(u, y)\cdot|\check{\phi}_{m+j}\ast f(u)|\ls (M\chi_{F_1})^{\frac{1}{p}}(y)$. Noting that $\check{\phi}_k\ast g\ls (M\chi_{F_2})^{\frac{1}{q}}$ and \eqref{eq:3.6}, we bound \eqref{eq:6.32} by
\begin{align}\label{eq:6.33}
&\sum_{I_1\in \mathcal{F}} \sum_{j\in \mathbb{Z}:\ \frac{1}{2^{m+10}|I_1|}\leq 2^j\leq \frac{2^{50}}{|I_1|} }\int_{-\infty}^{\infty} \int_{-\infty}^{\infty}\int_{\Omega_{j+m}} 2^{j+m} \chi_{I_1}(y)\chi_{I_1}(u)\chi_{\Omega^{\complement}}(u+2^{-j}t) \\
  \times&\delta_{j,\frac{K}{2}}(u, y)\cdot(M\chi_{F_1})^{\frac{1}{p}}(y)\cdot\left(M\chi_{F_2}\right)^{\frac{1}{q}}\left(u+\textbf{tr}_j(t)\right) \cdot|\rho(t)|\,\textrm{d}y\,\textrm{d}t\,\textrm{d}u.\nonumber
\end{align}
As in \eqref{eq:6.019}, by the fact that $2^{-j}t\ls 2^{-j}$ for $j>0$ and $2^{-j}t\ls \gamma(2^{-j})$ for $j\leq0$, it is easy to see that
$$\int_{-\infty}^{\infty}\chi_{\Omega^{\complement}}(u+2^{-j}t)\cdot(M\chi_{F_2})^{\frac{1}{q}}(u+\textbf{tr}_j(t)) \cdot |\rho(t)|\,\textrm{d}t\ls(M\chi_{F_2})^{\frac{1}{q}}(w)$$ for some $w\in \Omega^{\complement}.$ From the definition of $\Omega$, it implies that $$\int_{-\infty}^{\infty}\chi_{\Omega^{\complement}}(u+2^{-j}t)\cdot(M\chi_{F_2})^{\frac{1}{q}}(u+\textbf{tr}_j(t)) \cdot|\rho(t)|\,\textrm{d}t\ls \left(\frac{|F_2|}{|F_3|}\right)^{\frac{1}{q}}.$$ We also have $$\int_{-\infty}^{\infty}2^{j+m}\chi_{I_1}(u)\cdot\delta_{j,\frac{K}{2}}(u, y)\,\textrm{d}u\ls M\chi_{I_1}(y)\ls 1.$$ Therefore, as in \eqref{eq:6.26}, we bound \eqref{eq:6.33} by
\begin{align*}
&\left(\frac{|F_2|}{|F_3|}\right)^{\frac{1}{q}}\sum_{I_1\in \mathcal{F}} \sum_{j\in \mathbb{Z}:\ \frac{1}{2^{m+10}|I_1|}\leq 2^j\leq \frac{2^{50}}{|I_1|} } \int_{\Omega_{j+m}} \chi_{I_1}(y)(M\chi_{F_1})^{\frac{1}{p}}(y)\,\textrm{d}y\\
\leq& m \left(\frac{|F_2|}{|F_3|}\right)^{\frac{1}{q}} \sum_{I_1\in \mathcal{F}}|10I_1|\frac{1}{|10I_1|}  \int_{10I_1 } (M\chi_{F_1})^{\frac{1}{p}}(y)\,\textrm{d}y
\leq m |\Omega|\left(\frac{|F_1|}{|F_3|}\right)^{\frac{1}{p}}\left(\frac{|F_2|}{|F_3|}\right)^{\frac{1}{q}}. \nonumber
\end{align*}
This is the desired estimate, since $|\Omega|<\frac{|F_3|}{2}$. Therefore, we obtain \eqref{eq:6.13}.

\subsection{Major term $|H^4_m|(f,g)$}\label{subsection 6.2}

In this subsection, we want to prove that
\begin{align}\label{eq:6.37}
|\langle |H^4_m|(f,g),\chi_{E'}\rangle|\ls m|F_1|^{\frac{1}{p}}|F_2|^{\frac{1}{q}}|F_3|^{\frac{1}{r'}}
\end{align}
for $r>\frac{1}{2}$. It is easy to see that $|\langle |H^4_m|(f,g),\chi_{E'}\rangle|$ can be bounded by
\begin{align*}
\sum_{j\in \mathbb{Z}}\sum_{k\in \mathbb{Z}}\phi\left(\frac{\gamma'(2^{-j})}{2^{m+j-k}}\right)\int_{-\infty}^{\infty} \int_{-\infty}^{\infty} \left|\left(\tilde{\psi}_{m+j}\cdot\check{\phi}_{m+j}\ast f\right)\left(x-2^{-j} t\right)\left(\tilde{\psi}_{k}\cdot\check{\phi}_k\ast g\right)\left(x-\gamma(2^{-j}t)\right) \rho(t)\right|\,\textrm{d}t\,\textrm{d}x.
\end{align*}
It will be bounded by $m|F_1|^{\frac{1}{p}}|F_2|^{\frac{1}{q}}|F_3|^{\frac{1}{r'}}$. In what follows, we give the proof for the case $j> 0$. The case $j\leq0$ can be handled similarly. According to the value range of $r$, we consider the following two cases: $r\geq 1$ and $\frac{1}{2}<r<1$.
The case $r\geq 1$ follows from  the following proposition.

\begin{proposition}\label{proposition 6.1}
There exists a positive constant $C$ such that
$$|\langle |H^4_m|(f,g),\chi_{E'}\rangle|\leq C |F_1|^{\frac{1}{p}}|F_2|^{\frac{1}{q}}|F_3|^{\frac{1}{r'}}$$
holds for $\frac{1}{p}+\frac{1}{q}=\frac{1}{r}$, and $r\geq 1$, $p>1$, $q>1$.
\end{proposition}
\begin{proof}
Let $u:=x-\gamma(2^{-j}t)$; then $x-2^{-j} t=u-\textbf{tr}_j(t)$. We can bound $|\langle |H^4_m|(f,g),\chi_{E'}\rangle|$ by
\begin{align*}
\sum_{j\in \mathbb{Z}}\sum_{k\in \mathbb{Z}}\phi\left(\frac{\gamma'(2^{-j})}{2^{m+j-k}}\right)\int_{-\infty}^{\infty} \int_{-\infty}^{\infty} \left|\left(\tilde{\psi}_{m+j}\cdot\check{\phi}_{m+j}\ast f\right)\left(u-\textbf{tr}_j(t)\right)\left(\tilde{\psi}_{k}\cdot\check{\phi}_k\ast g\right)(u) \rho(t)\right|\,\textrm{d}t\,\textrm{d}u.
\end{align*}
Noting that $|\textbf{tr}_j(t)|\ls |t|$, we have $$\int_{-\infty}^{\infty} | (\tilde{\psi}_{m+j}\cdot\check{\phi}_{m+j}\ast f)(u-\textbf{tr}_j(t))\rho(t) |\,\textrm{d}t\ls  M(\tilde{\psi}_{m+j}\cdot\check{\phi}_{m+j}\ast f)(u).$$ By the Cauchy-Schwarz inequality, $|\langle |H^4_m|(f,g),\chi_{E'}\rangle|$ can be bounded by
\begin{align*}
\int_{-\infty}^{\infty} \left[ \sum_{j\in T}\sum_{k\in \mathbb{Z}}\phi\left(\frac{\gamma'(2^{-j})}{2^{m+j-k}}\right)\left|M\left(\tilde{\psi}_{m+j}\cdot\check{\phi}_{m+j}\ast f\right)(u)\right|^2\right]^{\frac{1}{2}}\cdot \left[ \sum_{j\in T}\sum_{k\in \mathbb{Z}}\phi\left(\frac{\gamma'(2^{-j})}{2^{m+j-k}}\right)\left|\left(\tilde{\psi}_{k}\cdot \check{\phi}_k\ast g\right)(u)\right|^2\right]^{\frac{1}{2}}     \,\textrm{d}u.
\end{align*}
Furthermore, we also have $|\tilde{\psi}_{m+j}(u)|\ls\sup_{k\in \mathbb{Z}}|\tilde{\psi}_{k}(u)|\ls M(\chi_\Omega)(u)\leq 1$. By the $\mathrm{H}\ddot{\mathrm{o}}\mathrm{lder}$ inequality, then $|\langle |H^4_m|(f,g),\chi_{E'}\rangle|$ is controlled by
\begin{align*}
\left\| \left[ \sum_{j\in T}\sum_{k\in \mathbb{Z}}\phi\left(\frac{\gamma'(2^{-j})}{2^{m+j-k}}\right)\left|M(\check{\phi}_{m+j}\ast f)\right|^2\right]^{\frac{1}{2}} \right\|_{L^{p}(\mathbb{R})} \left\| \left[ \sum_{j\in T}\sum_{k\in \mathbb{Z}}\phi\left(\frac{\gamma'(2^{-j})}{2^{m+j-k}}\right)\left| \check{\phi}_k\ast g\right|^2\right]^{\frac{1}{2}}  \right\|_{L^{q}(\mathbb{R})} \|M(\chi_\Omega)\|_{L^{r'}(\mathbb{R})},
\end{align*}
which can be further bounded by
$\|f\|_{L^{p}(\mathbb{R})}\|g\|_{L^{q}(\mathbb{R})}\|M(\chi_\Omega)\|_{L^{r'}(\mathbb{R})}$ by the Fefferman-Stein inequality, the Littlewood-Paley Theory, \eqref{eq:3.7} and \eqref{eq:3.6}. Noting that $f=\chi_{F_1}$, $g=\chi_{F_2}(x)$ and $|\Omega|<\frac{|F_3|}{2}$, from the $L^{r'}(\mathbb{R})$ boundedness of $M$ for all $1<r'\leq\infty$, $|\langle |H^4_m|(f,g),\chi_{E'}\rangle|$ can be bounded by $|F_1|^{\frac{1}{p}}|F_2|^{\frac{1}{q}}|F_3|^{\frac{1}{r'}}$. Therefore, we complete the proof of Proposition \ref{proposition 6.1}.
\end{proof}

The rest of this subsection is devoted to the case $\frac{1}{2}<r<1$. Let $\theta$ be a nonnegative Schwartz function such that $\hat{\theta}$ is supported on $\{t\in \mathbb{R}:\  |t|\leq 2^{-10}\}$ and $\hat{\theta}(0)=1$, $\theta_\lambda(x):=2^\lambda\theta(2^\lambda x)$, $\lambda\in \mathbb{Z}$. Let $I_{n,j}:=[\frac{n}{2^j}, \frac{n+1}{2^j}]$ and $\chi^*_{I_{n,j}}:=\chi_{I_{n,j}}\ast \theta_{j+m}$. We can make a partition of unity $1=\sum_{n\in \mathbb{Z}}\chi^*_{I_{n,j}}(x)$. Denote
$$F_{n,m,j}(x,t):=(\chi^*_{I_{n,j}}\cdot\tilde{\psi}_{m+j}\cdot\check{\phi}_{m+j}\ast f)(x-2^{-j} t)\quad \textrm{and}\quad G_{n,m,j,k}(x,t):=(\chi^*_{I_{n,j}}\cdot\tilde{\psi}_{k}\cdot\check{\phi}_k\ast g)(x-\gamma(2^{-j}t)),$$
and define
\begin{align}\label{eq:6.39}
T_m(f,g)(x):=\sum_{j\in \mathbb{N}}\sum_{k\in \mathbb{Z}}\phi\left(\frac{\gamma'(2^{-j})}{2^{m+j-k}}\right)\int_{-\infty}^{\infty} \left| \left[\sum_{n\in \mathbb{Z}} F_{n,m,j}(x,t)\right]\cdot\left[\sum_{n\in \mathbb{Z}} G_{n,m,j,k}(x,t)\right]\cdot \rho(t)\right|\,\textrm{d}t.
\end{align}
Next, we introduce the definition of a tree.
\begin{definition}\label{definition 6.4}
Let $S\subset S_0:=\{(j,n); j\in \mathbb{N}, n\in \mathbb{Z}\}$. A subset $T\subset S$ is called a \emph{tree} of $S$ with top $(j_0,n_0)\in S$ if $I_{n,j}\subset I_{n_0,j_0}$ for all $(j,n)\in T$. $T$ is called a \emph{maximal tree} with top $(j_0,n_0)$ in $S$ if there is no tree in $T'\subset S$ with the same top but strictly containing $T$.
\end{definition}

We still need several notations. For any fixed set $S\subset S_0$, we abuse the notation $j\in S$ if and only if $(j,n)\in S$. For any $j\in S$, we denote  $S_j:=\{n\in \mathbb{Z}:\ (j,n)\in S\}$. An operator $\Lambda_S[f,g]$ based on the set $S$ is defined as
\begin{align}\label{eq:6.41}
\sum_{j\in S}\sum_{k\in \mathbb{Z}}\phi\left(\frac{\gamma'(2^{-j})}{2^{m+j-k}}\right)\int_{-\infty}^{\infty}\int_{-\infty}^{\infty} \left| \left[\sum_{n\in S_j} F_{n,m,j}(x,t)\right]\cdot\left[\sum_{n\in S_j} G_{n,m,j,k}(x,t)\right]\cdot \rho(t)\right|\,\textrm{d}t\,\textrm{d}x.
\end{align} We can use this philosophy to define other operators based on any set $U\subset S_0$.
Then, our aim is to show that
\begin{align}\label{eq:6.42}
\Lambda_{S_0}[f,g]\ls m|F_1|^{\frac{1}{p}}|F_2|^{\frac{1}{q}}|F_3|^{\frac{1}{r'}}.
\end{align}

Let $T$ be a tree; we rewrite $\Lambda_T[f,g]$ as
\begin{align*}
\sum_{j\in T}\sum_{k\in \mathbb{Z}}\phi\left(\frac{\gamma'(2^{-j})}{2^{m+j-k}}\right)\int_{-\infty}^{\infty}\int_{-\infty}^{\infty} \left| \left[\sum_{n\in T_j} f_{n,m,j}\left(u-\textbf{tr}_j(t)\right)\right]\cdot\left[\sum_{n\in T_j} g_{n,m,j,k}(u)\right]\cdot \rho(t)\right|\,\textrm{d}t\,\textrm{d}u.
\end{align*}
where $f_{n,m,j}:=\chi^*_{I_{n,j}}\cdot\tilde{\psi}_{m+j}\cdot\check{\phi}_{m+j}\ast f$ and $g_{n,m,j,k}:=\chi^*_{I_{n,j}}\cdot\tilde{\psi}_{k}\cdot\check{\phi}_k\ast g$. Noting that $|\textbf{tr}_j(t)|\ls |t|$, we have $$\int_{-\infty}^{\infty} \left| \left[\sum_{n\in T_j} f_{n,m,j}(u-\textbf{tr}_j(t))\right] \rho(t)\right|\,\textrm{d}t\ls  M\left(\sum_{n\in T_j} f_{n,m,j}\right)(u).$$  By the Cauchy-Schwarz inequality and $\mathrm{H}\ddot{\mathrm{o}}\mathrm{lder}$ inequality, it is bounded by
\begin{align*}
\left\| \left[ \sum_{j\in T}\sum_{k\in \mathbb{Z}}\phi\left(\frac{\gamma'(2^{-j})}{2^{m+j-k}}\right)\left|M\left[\sum_{n\in T_j} f_{n,m,j}\right]\right|^2\right]^{\frac{1}{2}} \right\|_{L^{q'}(\mathbb{R})} \left\| \left[ \sum_{j\in T}\sum_{k\in \mathbb{Z}}\phi\left(\frac{\gamma'(2^{-j})}{2^{m+j-k}}\right)\left|\sum_{n\in T_j} g_{n,m,k,j}\right|^2\right]^{\frac{1}{2}}  \right\|_{L^{q}(\mathbb{R})}.
\end{align*}
By the Fefferman-Stein inequality and \eqref{eq:3.6}, it is controlled by
\begin{align}\label{eq:6.44}
&\left\| \left[ \sum_{j\in T}\left|\sum_{n\in T_j} f_{n,m,j}\right|^2\right]^{\frac{1}{2}} \right\|_{L^{q'}(\mathbb{R})} \left\| \left[ \sum_{j\in T}\sum_{k\in \mathbb{Z}}\phi\left(\frac{\gamma'(2^{-j})}{2^{m+j-k}}\right)\left|\sum_{n\in T_j} g_{n,m,k,j}\right|^2\right]^{\frac{1}{2}}  \right\|_{L^{q}(\mathbb{R})}\\
=:&\| S_{1,T}(f)\|_{L^{q'}(\mathbb{R})} \| S_{2,T}(g)\|_{L^{q}(\mathbb{R})},\nonumber
\end{align}
where we have defined $S_{1,T}(f)$ and $S_{2,T}(g)$ by, respectively,
\begin{align}\label{eq:6.45}
\begin{cases}S_{1,T}(f)(x):=\left[ \sum_{j\in T}\left|\sum_{n\in T_j} f_{n,m,j}(x)\right|^2\right]^{\frac{1}{2}};\\
S_{2,T}(g)(x):=\left[ \sum_{j\in T}\sum_{k\in \mathbb{Z}}\phi\left(\frac{\gamma'(2^{-j})}{2^{m+j-k}}\right)\left|\sum_{n\in T_j} g_{n,m,k,j}(x)\right|^2\right]^{\frac{1}{2}}.
\end{cases}
\end{align}

\subsubsection{Sizes and BMO estimates}\label{subsubsection 6.2.1}

For any positive integer $K$ large enough, we have
$$
\chi^*_{I_{n,j}}(x)\ls \int_{I_{n,j}}\frac{2^{j+m}}{(1+2^{j+m}|x-y|)^K}\,\textrm{d}y=:\chi^{**}_{I_{n,j}}(x)\quad \textrm{and}\quad
\chi^*_{I_{n,j}}(x)\ls \frac{1}{\left(1+2^{j+m}\textrm{dist}(x,I_{n,j})\right)^K}.
$$
\begin{definition}\label{definition 6.6}
Let $T\subset S$ be a tree, $I_T$ be the \emph{time interval of the top of tree $T$}, and $Df$ be the derivative of $f$. We define the following $\textbf{1-size}(T)$ and $\textbf{2-size}(T)$ as
\begin{align}\label{eq:6.47}
\textbf{1-size}(T):=|I_T|^{-\frac{1}{p}}\Bigg\{ &\left\| \left[ \sum_{j\in T}\left|\sum_{n\in T_j} \left(\chi^{**}_{I_{n,j}}\cdot\chi_{\Omega_{m+j}^{\complement}}\ast \psi_{m+j}\cdot \check{\phi}_{m+j}\ast f\right)\right|^2\right]^{\frac{1}{2}} \right\|_{L^{p}(\mathbb{R})} \\
+&\left\| \left[ \sum_{j\in T}\left|\sum_{n\in T_j} \left(\chi^{**}_{I_{n,j}}\cdot\chi_{\Omega_{m+j}^{\complement}}\ast \psi_{m+j}\cdot (D\check{\phi})_{m+j}\ast f\right)\right|^2\right]^{\frac{1}{2}} \right\|_{L^{p}(\mathbb{R})}\nonumber\\
+&\left\| \left[ \sum_{j\in T}\left|\sum_{n\in T_j} \left(\chi^{**}_{I_{n,j}}\cdot\chi_{\Omega_{m+j}^{\complement}}\ast (D\psi)_{m+j}\cdot \check{\phi}_{m+j}\ast f\right)\right|^2\right]^{\frac{1}{2}} \right\|_{L^{p}(\mathbb{R})}
     \Bigg\},\nonumber
\end{align}
\begin{align}\label{eq:6.48}
\textbf{2-size}(T):=|I_T|^{-\frac{1}{q}}\left\| \left[ \sum_{j\in T}\sum_{k\in \mathbb{Z}}\phi\left(\frac{\gamma'(2^{-j})}{2^{m+j-k}}\right)\left|\sum_{n\in T_j} \left(\chi^{**}_{I_{n,j}}\cdot\tilde{\psi}_{k}\cdot\check{\phi}_k\ast g\right)\right|^2\right]^{\frac{1}{2}} \right\|_{L^{q}(\mathbb{R})}.
\end{align}
\end{definition}

\begin{definition}\label{definition 6.7}
For any  subset $U\subset S_0$ and $i\in\{1,2\}$, the $\textbf{size}_\textbf{i}(U)$ is defined as
\begin{align}\label{eq:6.49}
\textbf{size}_\textbf{i}(U):=\sup_{T\subset U}|\textbf{i-size}(T)|,
\end{align}
where $T\subset U$ is a tree.
\end{definition}

\begin{lemma}\label{lemma 6.8}
Let $\rho\in(1,\infty)$, $M_\rho f:=(M(|f|^\rho))^{\frac{1}{\rho}}$. Then, for any tree $T\subset S$, there exists a positive constant $C$ such that
\begin{align}\label{eq:6.52}
\textbf{1-size}(T)\leq C \inf_{x\in I_T}M_p(Mf)(x)\quad \textrm{and}\quad \textbf{2-size}(T)\leq C \inf_{x\in I_T}M_q(Mg)(x).
\end{align}
\end{lemma}

\begin{proof}
By the Littlewood-Paley theory,
\begin{align}\label{eq:6.56}
\left\| \left[ \sum_{j\in T}\left|\sum_{n\in T_j} \left(\chi^{**}_{I_{n,j}}\cdot\chi_{\Omega_{m+j}^{\complement}}\ast \psi_{m+j}\cdot \check{\phi}_{m+j}\ast f\right)\right|^2\right]^{\frac{1}{2}} \right\|_{L^{p}(\mathbb{R})}\ls& \left\| \left[ \sum_{j\in T}\left|\check{\phi}_{m+j}\ast f\right|^2\right]^{\frac{1}{2}} \right\|_{L^{p}(\mathbb{R})}
\ls \| f \|_{L^{p}(\mathbb{R})}.
\end{align}
For the first part of \eqref{eq:6.52}, we split $f$ into $f\chi_{2I_T}$ and $f\chi_{(2I_T)^{\complement}}$. It is easy to see that
\begin{align}\label{eq:6.57}
\left\| \left[ \sum_{j\in T}\left|\sum_{n\in T_j} \left(\chi^{**}_{I_{n,j}}\cdot\chi_{\Omega_{m+j}^{\complement}}\ast \psi_{m+j}\cdot \check{\phi}_{m+j}\ast (f\chi_{2I_T})\right)\right|^2\right]^{\frac{1}{2}} \right\|_{L^{p}(\mathbb{R})}\ls \| f \chi_{2I_T}\|_{L^{p}(\mathbb{R})}
\ls |I_T|^{\frac{1}{p}}\inf_{x\in I_T}M_pf(x),
\end{align}
and
\begin{align}\label{eq:6.58}
&\left\| \left[ \sum_{j\in T}\left|\sum_{n\in T_j}\left( \chi^{**}_{I_{n,j}}\cdot\chi_{\Omega_{m+j}^{\complement}}\ast \psi_{m+j}\cdot \check{\phi}_{m+j}\ast (f\chi_{(2I_T)^{\complement}})\right)\right|^2\right]^{\frac{1}{2}} \right\|_{L^{p}(\mathbb{R})}\\
\ls& \sum_{j\in T}\sum_{n\in T_j} \left\|  \chi^{**}_{I_{n,j}}\cdot\chi_{\Omega_{m+j}^{\complement}}\ast \psi_{m+j}\cdot \check{\phi}_{m+j}\ast (f\chi_{(2I_T)^{\complement}}) \right\|_{L^{p}(\mathbb{R})}.\nonumber
\end{align}
On the other hand, $(\chi^{**}_{I_{n,j}}\cdot\chi_{\Omega_{m+j}^{\complement}}\ast \psi_{m+j}\cdot \check{\phi}_{m+j}\ast (f\chi_{(2I_T)^{\complement}}))(x)$ can be bounded by
\begin{align*}\frac{1}{(1+2^{j+m}\textrm{dist}(x,I_{n,j}))^{\frac{K}{2}}}  \frac{1}{(1+2^{j+m}\textrm{dist}(I_{n,j}, (2I_T)^{\complement}))^{\frac{K}{2}}} Mf(x).
\end{align*}
\eqref{eq:6.58} is bounded by
\begin{align}\label{eq:6.60}
\sum_{j\in T}\sum_{n\in T_j} \left[\int_{-\infty}^{\infty}\frac{1}{\left(1+2^{j+m}\textrm{dist}(x,I_{n,j})\right)^{\frac{pK}{2}}} \frac{1}{\left(1+2^{j+m}\textrm{dist}(I_{n,j}, (2I_T)^{\complement})\right)^{\frac{pK}{2}}}(Mf)^p(x)\,\textrm{d}x \right]^{\frac{1}{p}}.
\end{align}
For the integral in the square bracket above, which is further bounded by the sum of
\begin{align*}
\int_{I_T}\frac{1}{\left(1+2^{j+m}\textrm{dist}(x,I_{n,j})\right)^{\frac{pK}{2}}} \frac{1}{\left(1+2^{j+m}\textrm{dist}(I_{n,j}, (2I_T)^{\complement})\right)^{\frac{pK}{2}}}(Mf)^p(x)\,\textrm{d}x=:\Theta
\end{align*}
and
\begin{align*}
\sum_{l\in\mathbb{N} }\int_{2^{l+1}I_T\setminus 2^{l}I_T}\frac{1}{\left(1+2^{j+m}\textrm{dist}(x,I_{n,j})\right)^{\frac{pK}{2}}} \frac{1}{\left(1+2^{j+m}\textrm{dist}(I_{n,j}, (2I_T)^{\complement})\right)^{\frac{pK}{2}}}(Mf)^p(x)\,\textrm{d}x=:\sum_{l\in\mathbb{N} }\Theta_l.
\end{align*}
Therefore, \eqref{eq:6.60} is bounded by
\begin{align}\label{eq:6.61}
\sum_{j\in T}\sum_{n\in T_j}\Theta^{\frac{1}{p}}+\sum_{j\in T}\sum_{n\in T_j}\left(\sum_{l\in\mathbb{N} }\Theta_l\right)^{\frac{1}{p}}.
\end{align}

For the first part in \eqref{eq:6.61}, we have
$$\frac{\frac{|I_T|}{|I_{n,j}|} }{\left(1+2^{j+m}\textrm{dist}(I_{n,j}, (2I_T)^{\complement})\right)^{\frac{pK}{4}}}\leq \frac{\frac{|I_T|}{|I_{n,j}|} }{\left(1+\frac{|I_T|}{|I_{n,j}|} \right)^{\frac{pK}{4}}}\ls 1~~\textrm{and}~~
\frac{1}{|I_T|}\int_{I_T} (Mf)^p(x)\,\textrm{d}x \ls \inf_{x\subseteq I_T} (M_p(Mf))^p(x).$$
Thus
\begin{align*}
\Theta\leq&\frac{|I_{n,j}|}{\left(1+2^{j+m}\textrm{dist}(I_{n,j}, (2I_T)^{\complement})\right)^{\frac{pK}{2}}} \frac{|I_T|}{|I_{n,j}|}  \frac{1}{|I_T|}\int_{I_T} (Mf)^p(x)\,\textrm{d}x\\
\ls&\frac{|I_{n,j}|}{\left(1+2^{j+m}\textrm{dist}(I_{n,j}, (2I_T)^{\complement})\right)^{\frac{pK}{4}}}\inf_{x\subseteq I_T} (M_p(Mf))^p(x). \nonumber
\end{align*}
Furthermore, we have
\begin{align*}
\sum_{j\in T}\sum_{n\in T_j}\Theta^{\frac{1}{p}}  \ls\sum_{j\in T}\sum_{n\in T_j} \frac{|I_{n,j}|^{\frac{1}{p}}}{\left(1+2^{j+m}\textrm{dist}(I_{n,j}, (2I_T)^{\complement})\right)^{\frac{K}{4}}}  \inf_{x\subseteq I_T} M_p(Mf)(x).
\end{align*}
It suffices to show that
\begin{align}\label{eq:6.64}
\sum_{j\in T}\sum_{n\in T_j} \frac{|I_{n,j}|^{\frac{1}{p}}}{\left(1+2^{j+m}\textrm{dist}(I_{n,j}, (2I_T)^{\complement})\right)^{\frac{K}{4}}}  \ls |I_T|^{\frac{1}{p}}.
\end{align}
Indeed, note that $\textrm{dist}(I_{n,j}, (2I_T)^{\complement})\gs |I_T|$; the left-hand side of \eqref{eq:6.64} is bounded by
\begin{align*}
\sum_{j\in T}\sum_{n\in T_j} \frac{|I_{n,j}|^{\frac{1}{p}}}{\left(1+\frac{|I_T|}{|I_{n,j}|}\right)^{\frac{K}{4}}} \ls \sum_{l\in \mathbb{Z}_-}\sum_{I_{n,j}:\ 2^l|I_T|\leq|I_{n,j}|\leq 2^{l+1}|I_T| } \frac{(2^{l+1}|I_T|)^{\frac{1}{p}}}{\left(1+\frac{|I_T|}{2^{l+1}|I_T|}\right)^{\frac{K}{4}}}
\ls |I_T|^{\frac{1}{p}} \sum_{l\in \mathbb{Z}_-}2^{-l} \frac{(2^{l+1})^{\frac{1}{p}}}{\left(1+\frac{1}{2^{l+1}}\right)^{\frac{K}{4}}}
\ls |I_T|^{\frac{1}{p}}.
\end{align*}
This is the desired estimate.

For the second part in \eqref{eq:6.61}, note that
$$
\begin{cases}\sum_{l\in\mathbb{N} }\left(1+\frac{2^{l}|I_T|}{|I_{n,j}|}\right)^{-\frac{pK}{2}} 2^{l+1}\frac{|I_T|}{|I_{n,j}|} \ls 1;\\
\frac{1}{2^{l+1}|I_T|}\int_{2^{l+1}I_T\backslash 2^{l}I_T} (Mf)^p(x)\,\textrm{d}x \ls \inf_{x\subseteq 2^{l+1}I_T} (M_p(Mf))^p(x)\ls \inf_{x\subseteq I_T} (M_p(Mf))^p(x).
\end{cases}
$$
Then, $\sum_{l\in\mathbb{N} }\Theta_l$ can be bounded by
\begin{align*}
&\frac{|I_{n,j}|}{\left(1+2^{j+m}\textrm{dist}(I_{n,j}, (2I_T)^{\complement})\right)^{\frac{pK}{2}}} \sum_{l\in\mathbb{N} }\left(1+\frac{2^{l}|I_T|}{|I_{n,j}|}\right)^{-\frac{pK}{2}} 2^{l+1}\frac{|I_T|}{|I_{n,j}|} \frac{1}{2^{l+1}|I_T|}\int_{2^{l+1}I_T\setminus 2^{l}I_T} (Mf)^p(x)\,\textrm{d}x\\
\ls &\frac{|I_{n,j}|}{\left(1+2^{j+m}\textrm{dist}(I_{n,j}, (2I_T)^{\complement})\right)^{\frac{pK}{2}}}\inf_{x\subseteq I_T} (M_p(Mf))^p(x).\nonumber
\end{align*}
As in the first part in \eqref{eq:6.61}, we can obtain \eqref{eq:6.52}.

We now turn to the second part of \eqref{eq:6.52}. From \eqref{eq:3.7} and the Littlewood-Paley theory
\begin{align}\label{eq:6.67}
\left\| \left[ \sum_{j\in T}\sum_{k\in \mathbb{Z}}\phi\left(\frac{\gamma'(2^{-j})}{2^{m+j-k}}\right)\left|\sum_{n\in T_j} \left(\chi^{**}_{I_{n,j}}\cdot\tilde{\psi}_{k}\cdot\check{\phi}_k\ast g\right)\right|^2\right]^{\frac{1}{2}} \right\|_{L^{q}(\mathbb{R})}\ls\left\| \left[ \sum_{k\in \mathbb{Z}}\left(\check{\phi}_k\ast g\right)^2\right]^{\frac{1}{2}} \right\|_{L^{q}(\mathbb{R})}
\ls\| g\|_{L^{q}(\mathbb{R})}.
\end{align}
We also split $g$ into $g\chi_{2I_T}$ and $g\chi_{(2I_T)^{\complement}}$. As in \eqref{eq:6.57},
\begin{align*}
\left\| \left[ \sum_{j\in T}\sum_{k\in \mathbb{Z}}\phi\left(\frac{\gamma'(2^{-j})}{2^{m+j-k}}\right)\left|\sum_{n\in T_j} \left(\chi^{**}_{I_{n,j}}\cdot\tilde{\psi}_{k}\cdot\check{\phi}_k\ast (g\chi_{2I_T})\right)\right|^2\right]^{\frac{1}{2}} \right\|_{L^{q}(\mathbb{R})}\ls|I_T|^{\frac{1}{q}}\inf_{x\in I_T}M_q(Mg)(x).
\end{align*}
For the $g\chi_{(2I_T)^{\complement}}$ part, from the fact that $\frac{\gamma'(2^{-j})}{2^{m+j-k}}\in \textrm{supp}~\phi$, $\gamma'$ is strictly increasing on $(0,\infty)$, $j> 0$. it implies that $\frac{1}{2}\leq \frac{\gamma'(2^{-j})}{2^{m+j-k}}\leq \frac{\gamma'(1)}{2^{m+j-k}}$, which further implies that $2^{m+j}\ls 2^k$. Therefore,
\begin{align*}
\left(\chi^{**}_{I_{n,j}}\cdot\tilde{\psi}_{k}\cdot\check{\phi}_k\ast (g\chi_{(2I_T)^{\complement}})\right)(x)\ls \frac{1}{\left(1+2^{j+m}\textrm{dist}(x,I_{n,j})\right)^{\frac{K}{2}}} \frac{1}{\left(1+2^{j+m}\textrm{dist}(I_{n,j}, (2I_T)^{\complement})\right)^{\frac{K}{2}}}Mg(x).
\end{align*}
As in \eqref{eq:3.6}, we also have $\sum_{k\in \mathbb{Z}} |\phi^{\frac{1}{2}}(\frac{\gamma'(2^{-j})}{2^{n+j-k}})|\ls 1$. As in \eqref{eq:6.60}, we have
\begin{align*}
\left\| \left[ \sum_{j\in T}\sum_{k\in \mathbb{Z}}\phi\left(\frac{\gamma'(2^{-j})}{2^{m+j-k}}\right)\left|\sum_{n\in T_j} \left(\chi^{**}_{I_{n,j}}\cdot\tilde{\psi}_{k}\cdot\check{\phi}_k\ast (g\chi_{(2I_T)^{\complement}})\right)\right|^2\right]^{\frac{1}{2}} \right\|_{L^{q}(\mathbb{R})}\ls|I_T|^{\frac{1}{q}}\inf_{x\in I_T}M_q(Mg)(x).
\end{align*}
Therefore, we obtain \eqref{eq:6.52}. Hence, we complete the proof of Lemma \ref{lemma 6.8}.
\end{proof}

\begin{lemma}\label{lemma 6.9}
For the general subset $U$ of $S_0$ and $p\in(1,\infty)$, there exists a positive constant $C$ such that
\begin{align}\label{eq:6.72}
 \| S_{1,U}(f) \|_{BMO}\leq C \min\left\{1, \frac{|F_1|}{|F_3|}\right\}^{\frac{1}{p}}.
\end{align}
\end{lemma}
\begin{proof}
Let $\textbf{J}$ be a dyadic interval of length $2^{-J}$. It suffices to bound the following formula:
\begin{align}\label{eq:6.73}
 \inf_{c\in \mathbb{R}} \int_{\textbf{J}}\left| \left[\sum_{j\in U}\left|\sum_{n\in U_j} \left(\chi^*_{I_{n,j}}\cdot\tilde{\psi}_{m+j}\cdot\check{\phi}_{m+j}\ast f\right)(x)\right|^2\right]^{\frac{1}{2}}-c  \right|\,\textrm{d}x,
\end{align}
which further is bounded by a sum of the following two parts:
\begin{align}\label{eq:6.74}
\begin{cases}J_1:= \int_{\textbf{J}} \left[ \sum_{j\in U:\ J\leq j+m }\left|\sum_{n\in U_j} \left(\chi^*_{I_{n,j}}\cdot\tilde{\psi}_{m+j}\cdot\check{\phi}_{m+j}\ast f\right)(x)\right|^2\right]^{\frac{1}{2}}\,\textrm{d}x;\\
J_2:=\inf_{c\in \mathbb{R}} \int_{\textbf{J}}\left| \left[ \sum_{j\in U:\ J> j+m }\left|\sum_{n\in U_j} \left(\chi^*_{I_{n,j}}\cdot\tilde{\psi}_{m+j}\cdot\check{\phi}_{m+j}\ast f\right)(x)\right|^2\right]^{\frac{1}{2}}-c\right|\,\textrm{d}x.
\end{cases}
\end{align}

For $J_1$ in \eqref{eq:6.74}, we bound it by $J_{1,1}+J_{1,2}$, where
\begin{align*}
\begin{cases}J_{1,1}:= \int_{\textbf{J}} \left[ \sum_{j\in U:\ J\leq j+m }\left|\sum_{n\in U_j} \left(\chi^*_{I_{n,j}}\cdot\tilde{\psi}_{m+j}\cdot\check{\phi}_{m+j}\ast (f\chi_{2\textbf{J}})\right)(x)\right|^2\right]^{\frac{1}{2}}\,\textrm{d}x;\\
J_{1,2}:= \int_{\textbf{J}} \left[ \sum_{j\in U:\ J\leq j+m }\left|\sum_{n\in U_j} \left(\chi^*_{I_{n,j}}\cdot\tilde{\psi}_{m+j}\cdot\check{\phi}_{m+j}\ast (f\chi_{(2\textbf{J})^{\complement}})\right)(x)\right|^2\right]^{\frac{1}{2}}\,\textrm{d}x.
\end{cases}
\end{align*}

For $J_{1,1}$, by the $\mathrm{H}\ddot{\mathrm{o}}\mathrm{lder}$ inequality, which is bounded by
\begin{align*}
|\textbf{J}|^{\frac{1}{p'}} \left\| \left[ \sum_{j\in U:\ J\leq j+m }\left|\sum_{n\in U_j} \left(\chi^*_{I_{n,j}}\cdot\tilde{\psi}_{m+j}\cdot\check{\phi}_{m+j}\ast (f\chi_{2\textbf{J}})\right)\right|^2\right]^{\frac{1}{2}} \right\|_{L^{p}(\textbf{J})}.
\end{align*}
As in \eqref{eq:6.56}, the above expression is bounded by
\begin{align}\label{eq:6.78}
|\textbf{J}|^{\frac{1}{p'}} \| f\chi_{2\textbf{J}} \|_{L^{p}(\mathbb{R})}\sup_{j\in U:\ J\leq j+m }\| \tilde{\psi}_{m+j} \|_{L^{\infty}(\textbf{J})}.
\end{align}
Let $s\in \mathbb{N}$ be the least integer such that $2^s\textbf{J}\bigcap \Omega^{\complement}\neq \emptyset$, where $2^s\textbf{J}$ denotes the interval of length $2^s|\textbf{J}|$ whose center is the same as that of $\textbf{J}$, then
\begin{align}\label{eq:6.79}
|\textbf{J}|^{-\frac{1}{p}} \| f\chi_{2\textbf{J}} \|_{L^{p}(\mathbb{R})}\ls 2^{\frac{s}{p}} \inf_{x\in2^s\textbf{J} }M_pf(x)\ls 2^{\frac{s}{p}}\min\left\{1, \frac{|F_1|}{|F_3|}\right\}^{\frac{1}{p}}.
\end{align}
On the other hand, $|\textbf{J}|=2^{-J}\geq2^{-j-m} $; it implies
\begin{align}\label{eq:6.80}
\sup_{j\in U:\ J\leq j+m }\| \tilde{\psi}_{m+j} \|_{L^{\infty}(\textbf{J})}\ls \sup_{j\in U:\ J\leq j+m }\frac{1}{\left(1+2^{j+m}\textrm{dist}(x,\Omega_{j+m}^{\complement})\right)^K}\ls\frac{1}{\left(1+2^{j+m}2^s|\textbf{J}|\right)^K}\ls 2^{-Ks}.
\end{align}
From \eqref{eq:6.78}, \eqref{eq:6.79} and \eqref{eq:6.80}, we have
\begin{align}\label{eq:6.81}
J_{1,1}\ls |\textbf{J}|\min\left\{1, \frac{|F_1|}{|F_3|}\right\}^{\frac{1}{p}}.
\end{align}

For $J_{1,2}$, for each $x\in \textbf{J}$, we choose $z\in \Omega^{\complement}$ such $\textrm{dist}(x,\Omega^{\complement})\approx |x-z|$. It implies that $\check{\phi}_{m+j}\ast (f\chi_{(2\textbf{J})^{\complement}})(x)$ can be bounded by
\begin{align*}
\int_{(2\textbf{J})^{\complement}} |f(y)| 2^{j+m} \delta_{j,K}(x,y)\,\textrm{d}y
\leq \int_{(2\textbf{J})^{\complement}} |f(y)|  \frac{\delta_{j,\frac{K}{2}}(z,y)}{\delta_{j,\frac{K}{2}}(z,x)\cdot\delta_{j,\frac{K}{2}}(x,y)}  2^{j+m} \delta_{j,K}(x,y)\,\textrm{d}y.
\end{align*}
On the other hand, for each $x\in \textbf{J}$ and $y\in (2\textbf{J})^{\complement}$, it implies that $|x-y|\geq 2^{-J-1}$. Furthermore,
\begin{align*}
\check{\phi}_{m+j}\ast \left(f\chi_{(2\textbf{J})^{\complement}}\right)(x)\leq Mf(z) \frac{\left(1+2^{j+m}\textrm{dist}(x,\Omega^{\complement})\right)^{\frac{K}{2}}}{(1+2^{j+m}|x-y|)^{\frac{K}{2}}}\leq Mf(z) \frac{\left(1+2^{j+m}\textrm{dist}(x,\Omega^{\complement})\right)^{\frac{K}{2}}}{(1+2^{j+m-J-1})^{\frac{K}{2}}}.
\end{align*}
Note that $Mf(z)\ls \min\{1, \frac{|F_1|}{|F_3|}\}^{\frac{1}{p}}$ and $\tilde{\psi}_{m+j}(x)(1+2^{j+m}\textrm{dist}(x,\Omega^{\complement}))^{\frac{K}{2}}\ls 1$; we have
\begin{align}\label{eq:6.82}
J_{1,2}\ls \min\left\{1, \frac{|F_1|}{|F_3|}\right\}^{\frac{1}{p}}\cdot\int_{\textbf{J}} \left[ \sum_{j\in U:\ J\leq j+m }\left|\frac{1}{(1+2^{j+m-J-1})^{\frac{K}{2}}}\right|^2\right]^{\frac{1}{2}}\,\textrm{d}x \ls|\textbf{J}|\min\left\{1, \frac{|F_1|}{|F_3|}\right\}^{\frac{1}{p}}.
\end{align}

By the H\"older inequality and $\mathrm{Poincar}\acute{\mathrm{e}}$ inequality, $J_2$ in \eqref{eq:6.74} is bounded by
\begin{align}\label{eq:6.83}
&|\textbf{J}|^{\frac{1}{2}}\inf_{c\in \mathbb{R}}\left\{ \int_{\textbf{J}}\left| \left[ \sum_{j\in U:\ J> j+m }\left|\sum_{n\in U_j} \left(\chi^*_{I_{n,j}}\cdot\tilde{\psi}_{m+j}\cdot\check{\phi}_{m+j}\ast f\right)(x)\right|^2\right]^{\frac{1}{2}}-c\right|^2\,\textrm{d}x\right\}^{\frac{1}{2}}\\
\ls&|\textbf{J}|\left\{ \int_{\textbf{J}} \left|D\left[\sum_{j\in U:\ J> j+m }\left[\sum_{n\in U_j} \left(\chi^*_{I_{n,j}}\cdot\tilde{\psi}_{m+j}\cdot\check{\phi}_{m+j}\ast f\right)\right]^2\right](x)\right|\,\textrm{d}x\right\}^{\frac{1}{2}}\nonumber.
\end{align}
Note that $|D[\sum_{n\in U_j} (\chi^*_{I_{n,j}}\cdot\tilde{\psi}_{m+j}\cdot\check{\phi}_{m+j}\ast f)]^2(x)|$ can be written as
$2| \sum_{n\in U_j} (\chi^*_{I_{n,j}}\cdot\tilde{\psi}_{m+j}\cdot\check{\phi}_{m+j}\ast f)(x)|\cdot|\sum_{n\in U_j} D(\chi^*_{I_{n,j}}\cdot\tilde{\psi}_{m+j}\cdot\check{\phi}_{m+j}\ast f)(x)|$. Regarding $J_{1,2}$, we have $\check{\phi}_{m+j}\ast f(x)\ls Mf(z) \left(1+2^{j+m}\textrm{dist}(x,\Omega^{\complement})\right)^{\frac{K}{2}}$, where $z\in \Omega^{\complement}$, which further implies that $|\sum_{n\in U_j} (\chi^*_{I_{n,j}}\cdot\tilde{\psi}_{m+j}\cdot\check{\phi}_{m+j}\ast f)(x)|\ls Mf(z)\ls \min\{1, \frac{|F_1|}{|F_3|}\}$. At the same time, we also have that $|\sum_{n\in U_j} D(\chi^*_{I_{n,j}}\cdot\tilde{\psi}_{m+j}\cdot\check{\phi}_{m+j}\ast f)(x)|\ls 2^{j+m} \min\{1, \frac{|F_1|}{|F_3|}\}$. Therefore, from \eqref{eq:6.83},
\begin{align}\label{eq:6.84}
J_2\ls|\textbf{J}|   \left[ \int_{\textbf{J}}  \sum_{j\in U:\ J> j+m } 2^{j+m} \left(\min\left\{1, \frac{|F_1|}{|F_3|}\right\}\right)^2 \,\textrm{d}x\right]^{\frac{1}{2}}\ls|\textbf{J}|\min\left\{1, \frac{|F_1|}{|F_3|}\right\}^{\frac{1}{p}}.
\end{align}
From \eqref{eq:6.81}, \eqref{eq:6.82} and \eqref{eq:6.84}, we obtain \eqref{eq:6.72}.
\end{proof}

\begin{lemma}\label{lemma 6.10}
There exists a positive constant $C$ such that
\begin{align}\label{eq:6.85}
 \left\| \chi^{**}_{I_{n,j}}\cdot\tilde{\psi}_{m+j}\cdot\check{\phi}_{m+j}\ast f \right\|_{BMO}\leq C 2^{m}\textbf{size}_\textbf{1}(I_{n,j}).
\end{align}
\end{lemma}
\begin{proof}
Let $\textbf{J}$ be a dyadic interval. It suffices to bound the following formula:
\begin{align}\label{eq:6.86}
 \inf_{c\in \mathbb{R}} \int_{\textbf{J}}\left| \left(\chi^{**}_{I_{n,j}}\cdot\tilde{\psi}_{m+j}\cdot\check{\phi}_{m+j}\ast f\right)(x)-c  \right|\,\textrm{d}x.
\end{align}
If $|I_{n,j}|\leq |\textbf{J}|$, by the $\mathrm{H}\ddot{\mathrm{o}}\mathrm{lder}$ inequality, we have
\begin{align*}
 \inf_{c\in \mathbb{R}} \int_{\textbf{J}}\left| \left(\chi^{**}_{I_{n,j}}\cdot\tilde{\psi}_{m+j}\cdot\check{\phi}_{m+j}\ast f\right)(x)-c  \right|\,\textrm{d}x\leq \left\| \chi^{**}_{I_{n,j}}\cdot\tilde{\psi}_{m+j}\cdot\check{\phi}_{m+j}\ast f\right \|_{L^{p}(\mathbb{R})}|\textbf{J}|^{\frac{1}{p'}}
 \leq|\textbf{J}|\textbf{size}_\textbf{1}(I_{n,j}).
\end{align*}
If $|I_{n,j}|> |\textbf{J}|$, by the $\mathrm{Poincar}\acute{\mathrm{e}}$ inequality, we have
\begin{align}\label{eq:6.88}
 \inf_{c\in \mathbb{R}} \int_{\textbf{J}}\left|\left( \chi^{**}_{I_{n,j}}\cdot\tilde{\psi}_{m+j}\cdot\check{\phi}_{m+j}\ast f\right)(x)-c  \right|\,\textrm{d}x\leq |\textbf{J}|\int_{\textbf{J}}\left| D\left(\chi^{**}_{I_{n,j}}\cdot\tilde{\psi}_{m+j}\cdot\check{\phi}_{m+j}\ast f\right)(x)\right|\,\textrm{d}x.
\end{align}
By the $\mathrm{H}\ddot{\mathrm{o}}\mathrm{lder}$ inequality, without loss of generality, the right-hand side of \eqref{eq:6.88} can be bounded by
\begin{align*}
 |\textbf{J}|2^{j+m}|I_{n,j}|^{-\frac{1}{p}}\left\| \chi^{**}_{I_{n,j}}\cdot\tilde{\psi}_{m+j}\cdot(D\check{\phi})_{m+j}\ast f\right\|_{L^{p}(\mathbb{R})} |\textbf{J}|^{\frac{1}{p'}}|I_{n,j}|^{\frac{1}{p}}\leq |\textbf{J}|2^{m}\textbf{size}_\textbf{1}(I_{n,j}).
\end{align*}
Therefore, we obtain \eqref{eq:6.85}.
\end{proof}

\begin{lemma}\label{lemma 6.11}
Let $T\subset S_0$ be a tree; then there exists a positive constant $C$ such that
\begin{align}\label{eq:6.89}
 \| S_{1,T}(f) \|_{BMO}\leq C 2^m\textbf{size}_\textbf{1}(T).
\end{align}
\end{lemma}
\begin{proof}
Let $\textbf{J}$ be a dyadic interval and $T_{\textbf{J}}:=\left\{(j,n)\in T:\ I_{n,j}\subset 3\textbf{J}\right\}$. It suffices to bound the following formula:
\begin{align}\label{eq:6.90}
 \inf_{c\in \mathbb{R}} \int_{\textbf{J}}\left| \left[ \sum_{j\in T}\left|\sum_{n\in T_j} \left(\chi^*_{I_{n,j}}\cdot\tilde{\psi}_{m+j}\cdot\check{\phi}_{m+j}\ast f\right)(x)\right|^2\right]^{\frac{1}{2}}-c  \right|\,\textrm{d}x.
\end{align}
Furthermore, let
$$
T_{\textbf{J}}^1:=\left\{(j,n)\in T\backslash{T_{\textbf{J}}}:\ |I_{n,j}|\leq |\textbf{J}|\right\}\quad \textrm{and}\quad
T_{\textbf{J}}^2:=\left\{(j,n)\in T\backslash{T_{\textbf{J}}}:\ |I_{n,j}|> |\textbf{J}|\right\}.
$$
Then, \eqref{eq:6.90} can be bounded by a sum of the following three parts:
\begin{align*}
\begin{cases}J^1:= \int_{\textbf{J}} \left[ \sum_{j\in T_{\textbf{J}}} \left|\sum_{n\in T_{\textbf{J},j}}\left(\chi^*_{I_{n,j}}\cdot\tilde{\psi}_{m+j}\cdot\check{\phi}_{m+j}\ast f\right)(x)\right|^2\right]^{\frac{1}{2}}\,\textrm{d}x;\\
J^2:= \int_{\textbf{J}} \left[ \sum_{j\in T_{\textbf{J}}^1} \left|\sum_{n\in T_{\textbf{J},j}^1}\left(\chi^*_{I_{n,j}}\cdot\tilde{\psi}_{m+j}\cdot\check{\phi}_{m+j}\ast f\right)(x)\right|^2\right]^{\frac{1}{2}}\,\textrm{d}x;\\
J^3:= \inf_{c\in \mathbb{R}} \int_{\textbf{J}}\left| \left[\sum_{j\in T_{\textbf{J}}^2}\left(\sum_{n\in T_{\textbf{J},j}^2} \left(\chi^*_{I_{n,j}}\cdot\tilde{\psi}_{m+j}\cdot\check{\phi}_{m+j}\ast f\right)(x)\right)^2\right]^{\frac{1}{2}}-c  \right|\,\textrm{d}x.
\end{cases}
\end{align*}

By the H\"older inequality, $J^1$ can be bounded by
\begin{align*}
\left\| \left[ \sum_{j\in T_{\textbf{J}}} \left|\sum_{n\in T_{\textbf{J},j}}\left(\chi^*_{I_{n,j}}\cdot\tilde{\psi}_{m+j}\cdot\check{\phi}_{m+j}\ast f\right)\right|^2\right]^{\frac{1}{2}} \right\|_{L^{p}(\textbf{J})} |\textbf{J}|^{\frac{1}{p'}} \leq\textbf{size}_\textbf{1}(T)|\textbf{J}|^{\frac{1}{p}}|\textbf{J}|^{\frac{1}{p'}}=\textbf{size}_\textbf{1}(T)|\textbf{J}|.
\end{align*}
By the H\"older inequality, $J^2$ can be bounded by
\begin{align*}
&\left\| \left[ \sum_{j\in T_{\textbf{J}}^1} \left|\sum_{n\in T_{\textbf{J},j}^1}\left(\chi^*_{I_{n,j}}\cdot\tilde{\psi}_{m+j}\cdot\check{\phi}_{m+j}\ast f\right)\right|^2\right]^{\frac{1}{2}} \right\|_{L^{p}(\textbf{J})} |\textbf{J}|^{\frac{1}{p'}}\\
\leq&\textbf{size}_\textbf{1}(T)\sum_{j\in T_{\textbf{J}}^1} \sum_{n\in T_{\textbf{J},j}^1}\frac{ |I_{n,j}|^{\frac{1}{p}} }{\left(1+2^{j+m}\textrm{dist}(\textbf{J},I_{n,j})\right)^{K}} |\textbf{J}|^{\frac{1}{p'}}
\ls \textbf{size}_\textbf{1}(T)|\textbf{J}|.\nonumber
\end{align*}
By the H\"older inequality and $\mathrm{Poincar}\acute{\mathrm{e}}$ inequality, as in \eqref{eq:6.83}, we bound $J^3$ by
\begin{align}\label{eq:6.96}
|\textbf{J}|\left\{ \int_{\textbf{J}} \left|D\left[\sum_{j\in T_{\textbf{J}}^2 }\left(\sum_{n\in T_{\textbf{J},j}^2} \chi^*_{I_{n,j}}\cdot\tilde{\psi}_{m+j}\cdot\check{\phi}_{m+j}\ast f\right)^2\right](x)\right|\,\textrm{d}x\right\}^{\frac{1}{2}}.
\end{align}
Without loss of generality, by the Cauchy-Schwarz inequality, $D[\sum_{n\in T_{\textbf{J},j}^2} (\chi^*_{I_{n,j}}\cdot\tilde{\psi}_{m+j}\cdot\check{\phi}_{m+j}\ast f)]^2$  can be bounded by the sum of
\begin{equation}\label{new1}
2^{\frac{m}{2}}|\textbf{J}|\left\{\int_{\textbf{J}}\sum_{j\in T_{\textbf{J}}^2 } \left[ \sum_{n\in T_{\textbf{J},j}^2} \frac{1}{|I_{n,j}|^{\frac{1}{2}}} |\chi^{*}_{I_{n,j}}\cdot\tilde{\psi}_{m+j}\cdot\check{\phi}_{m+j}\ast f(x)|\right]^2 \,\textrm{d}x\right\}^{\frac{1}{2}}
\end{equation}
and
\begin{equation}\label{new2}
2^{\frac{m}{2}}|\textbf{J}|\left\{\int_{\textbf{J}}\sum_{j\in T_{\textbf{J}}^2 } \left[\sum_{n\in T_{\textbf{J},j}^2} \frac{1}{|I_{n,j}|^{\frac{1}{2}}} |\chi^{*}_{I_{n,j}}\cdot\tilde{\psi}_{m+j}\cdot(D\check{\phi})_{m+j}\ast f(x)|\right]^2 \,\textrm{d}x\right\}^{\frac{1}{2}}.
\end{equation}
It is suffices to bound \eqref{new1}. Note that $$\left\|\chi^{**}_{I_{n,j}}\cdot\tilde{\psi}_{m+j}\cdot\check{\phi}_{m+j}\ast f\right\|_{L^{p}(\mathbb{R})} \leq |I_{n,j}|^{\frac{1}{p}}\textbf{size}_\textbf{1}(I_{n,j})$$ and \eqref{eq:6.85}, by interpolating, we obtain $$\left\|\chi^{**}_{I_{n,j}}\cdot\tilde{\psi}_{m+j}\cdot\check{\phi}_{m+j}\ast f\right\|_{L^{2p}(\mathbb{R})} \leq 2^{\frac{m}{2}}|I_{n,j}|^{\frac{1}{2p}}\textbf{size}_\textbf{1}(I_{n,j}).$$ By the H\"older inequality, \eqref{new1} can be bounded by
\begin{align*}
&2^{\frac{m}{2}}|\textbf{J}| \sum_{j\in T_{\textbf{J}}^2 } \sum_{n\in T_{\textbf{J},j}^2}\frac{1}{|I_{n,j}|^{\frac{1}{2}}}\frac{\left\|\chi^{**}_{I_{n,j}}\cdot\tilde{\psi}_{m+j}\cdot\check{\phi}_{m+j}\ast f \right\|_{L^{2p}(\textbf{J})} |J|^{\frac{1}{2}-\frac{1}{2p}}}{ \left(1+2^{j+m}\textrm{dist}(\textbf{J},I_{n,j})\right)^{K} }\\
\leq&2^{\frac{m}{2}}|\textbf{J}| \sum_{j\in T_{\textbf{J}}^2 } \sum_{n\in T_{\textbf{J},j}^2}\frac{1}{|I_{n,j}|^{\frac{1}{2}}}\frac{2^{\frac{m}{2}}|I_{n,j}|^{\frac{1}{2p}}\textbf{size}_\textbf{1}(I_{n,j}) |J|^{\frac{1}{2}-\frac{1}{2p}} }{ \left(1+2^{j+m}\textrm{dist}(\textbf{J},I_{n,j})\right)^{K} }
\ls 2^{m}|\textbf{J}|\textbf{size}_\textbf{1}(T).
\end{align*}
The last inequality follows from
\begin{align}\label{eq:6.97}
 \sum_{j\in T_{\textbf{J}}^2 } \sum_{n\in T_{\textbf{J},j}^2}\frac{ |I_{n,j}|^{\frac{1}{2p}-\frac{1}{2}}|J|^{\frac{1}{2}-\frac{1}{2p}}}{ \left(1+2^{j+m}\textrm{dist}(\textbf{J},I_{n,j})\right)^{K} }\ls 1,
\end{align}
which can be found in Li \cite{L3}. Hence, we complete the proof of Lemma \ref{lemma 6.11}.
\end{proof}

From \eqref{eq:6.45} to \eqref{eq:6.49}, we have
\begin{align}\label{eq:6.50}
\|S_{1,T}(f)\|_{L^{p}(\mathbb{R})}\leq \textbf{size}_\textbf{1}(T)\cdot|I_T|^{\frac{1}{p}}\quad \textrm{and}\quad  \|S_{2,T}(f)\|_{L^{q}(\mathbb{R})}\leq \textbf{size}_\textbf{2}(T)\cdot|I_T|^{\frac{1}{q}}.
\end{align}
From \eqref{eq:6.72} in Lemma \ref{lemma 6.9} and \eqref{eq:6.89} Lemma \ref{lemma 6.11}, by interpolation with \eqref{eq:6.50}, respectively, we obtain
\begin{align*}
\begin{cases}\| S_{1,T}(f)\|_{L^{q'}(\mathbb{R})}\ls |I_T|^{\frac{1}{q'}}\textbf{size}_\textbf{1}(T)^{\frac{p}{q'}}\min\left\{1, \frac{|F_1|}{|F_3|}\right\}^{\frac{1}{p}-\frac{1}{q'}} ;\\
\| S_{1,T}(f)\|_{L^{q'}(\mathbb{R})}\ls |I_T|^{\frac{1}{q'}}\textbf{size}_\textbf{1}(T)2^{m(1-\frac{p}{q'})}.
\end{cases}
\end{align*}
Furthermore, we have
\begin{align}\label{eq:6.102}
\| S_{1,T}(f)\|_{L^{q'}(\mathbb{R})}\ls |I_T|^{\frac{1}{q'}}\textbf{size}^*_\textbf{1}(T),
\end{align}
where $\textbf{size}^*_\textbf{1}(T):=\min\{\textbf{size}_\textbf{1}(T)^{\frac{p}{q'}}\min\{1, \frac{|F_1|}{|F_3|}\}^{\frac{1}{p}-\frac{1}{q'}}, \textbf{size}_\textbf{1}(T)2^{m(1-\frac{p}{q'})} \}$.

\subsubsection{The estimates for $|H^4_m|(f,g)$}\label{subsubsection 6.2.2}

For $S\subset S_0$, we rewrite $\Lambda_S[f,g]$ as
\begin{align}\label{eq:6.103}
\sum_{j\in S}\sum_{k\in \mathbb{Z}}\phi\left(\frac{\gamma'(2^{-j})}{2^{m+j-k}}\right)\int_{-\infty}^{\infty}\int_{-\infty}^{\infty} \left| \left[\sum_{n\in S_j} f_{n,m,j}(x-\textbf{tr}_j(t))\right]\cdot\left[\sum_{n\in S_j} g_{n,m,j,k}(x)\right]\cdot \rho(t)\right|\,\textrm{d}x\,\textrm{d}t.
\end{align}

\begin{lemma}\label{lemma 6.12}
Let $T\subset S_0$ be a tree and $P\subset S_0$ be a subset; if $T\cap P=\emptyset$ and $T$ is a maximal tree in $P\cup T$, then there exists a positive constant $C$ such that
\begin{align}\label{eq:6.104}
|\Lambda_{P\bigcup T}[f,g]-\Lambda_{P}[f,g]-\Lambda_{ T}[f,g]|\leq C \textbf{size}^*_\textbf{1}(P\cup T)\textbf{size}_\textbf{2}(P\cup T)|I_T|,
\end{align}
where $\textbf{size}^*_\textbf{1}(P\cup T)$ is defined naturally as $\textbf{size}^*_\textbf{1}(T)$.
\end{lemma}
\begin{proof} We begin our proof by a definition and a lemma. Let
\begin{align*}
d_j(S_1,S_2):=\left(1+2^{j+m}\textrm{dist}(S_1,S_2)\right)^{-K}\quad \textrm{for} ~\textrm{any} ~S_1,S_2\subset S_0,
\end{align*}
where $S_1$ and $S_2$ are defined as the union of all intervals $I_{n,j}$ with $(j,n)\in S_1$ and $(j,n)\in S_2$, respectively. As in \cite[Lemma 10.1]{LX}, if $T\cap P=\emptyset$, we have
\begin{align}\label{eq:6.1040}
\begin{cases} \sum_{j\in \mathbb{N}}\sum_{I:\ |I|=2^{-j}} |I|d_j(5I,T_j)d_j(I,P_j)\ls  |I_T|;\\
\sum_{j\in \mathbb{N}}\sum_{I:\ |I|=2^{-j}} |I|d_j(5I,P_j)d_j(I,T_j)\ls  |I_T|.
\end{cases}
\end{align}
Since $T$ is a maximal tree in $P\cup T$, then the left hand side of \eqref{eq:6.104} can be bounded by
\begin{align*}
\mathbb{A}+\mathbb{B}:=&\sum_{j\in \mathbb{N}}\sum_{k\in \mathbb{Z}}\phi\left(\frac{\gamma'(2^{-j})}{2^{m+j-k}}\right)\int_{-\infty}^{\infty}\int_{-\infty}^{\infty} \left| \left[\sum_{n\in P_j} f_{n,m,j}\left(x-\textbf{tr}_j(t)\right)\right]\cdot\left[\sum_{n\in T_j} g_{n,m,j,k}(x)\right]\cdot \rho(t)\right|\,\textrm{d}x\,\textrm{d}t\\
&+\sum_{j\in \mathbb{N}}\sum_{k\in \mathbb{Z}}\phi\left(\frac{\gamma'(2^{-j})}{2^{m+j-k}}\right)\int_{-\infty}^{\infty}\int_{-\infty}^{\infty} \left| \left[\sum_{n\in T_j} f_{n,m,j}\left(x-\textbf{tr}_j(t)\right)\right]\cdot\left[\sum_{n\in P_j} g_{n,m,j,k}(x)\right]\cdot \rho(t)\right|\,\textrm{d}x\,\textrm{d}t.\nonumber
\end{align*}

We here give the estimate of $\mathbb{A}$; $\mathbb{B}$ can be handled similarly. From the H\"older inequality, $\mathbb{A}$ is bounded by
\begin{align*}
\sum_{j\in \mathbb{N}}\sum_{k\in \mathbb{Z}}\phi^{\frac{1}{2}}\left(\frac{\gamma'(2^{-j})}{2^{m+j-k}}\right)\sum_{I:\ |I|=2^{-j}}&\left\|\int_{-\infty}^{\infty}\left|\sum_{n\in P_j} f_{n,m,j}(\cdot-\textbf{tr}_j(t))\cdot \rho(t)\right|\,\textrm{d}t \right\|_{L^{q'}(I)}
\left\| \phi^{\frac{1}{2}}\left(\frac{\gamma'(2^{-j})}{2^{m+j-k}}\right)\cdot \sum_{n\in T_j} g_{n,m,j,k} \right\|_{L^{q}(I)}.
\end{align*}
For $x\in I$, note that $\textbf{tr}_j(t)\ls 2^{-j}$; without loss of generality, we may write $\chi^{*}_{I_{n,j}}(x-\textbf{tr}_j(t))\ls d_j(5I,I_{n,j})\cdot\chi^{**}_{I_{n,j}}(x-\textbf{tr}_j(t))$. From \eqref{eq:6.102},  $\|\int_{-\infty}^{\infty}|\sum_{n\in P_j} f_{n,m,j}(\cdot-\textbf{tr}_j(t)) \cdot\rho(t)|\,\textrm{d}t \|_{L^{q'}(I)}$ is controlled by
\begin{align}\label{eq:6.106}
& \sum_{n\in P_j} d_j(5I,I_{n,j}) \left\|\int_{-\infty}^{\infty}\left|\left(\chi^{**}_{I_{n,j}}\cdot\tilde{\psi}_{m+j}\cdot\check{\phi}_{m+j}\ast f\right)(\cdot-\textbf{tr}_j(t))\cdot \rho(t)\right|\,\textrm{d}t \right\|_{L^{q'}(I)}\\
\leq& \sum_{n\in P_j} d_j(5I,I_{n,j}) \left\|M\left(\chi^{**}_{I_{n,j}}\cdot\tilde{\psi}_{m+j}\cdot\check{\phi}_{m+j}\ast f\right) \right\|_{L^{q'}(\mathbb{R})}
\leq d_j(5I,P_j) |I|^{\frac{1}{q'}}\textbf{size}^*_\textbf{1}(P\cup T).\nonumber
\end{align}
It is easy to see that
\begin{align}\label{eq:6.107}
\left\| \phi^{\frac{1}{2}}\left(\frac{\gamma'(2^{-j})}{2^{m+j-k}}\right) \cdot\sum_{n\in T_j} g_{n,m,j,k} \right\|_{L^{q}(I)}\ls d_j(I,T_j)|I|^{\frac{1}{q}}\textbf{size}_\textbf{2}(P\cup T).
\end{align}
As in \eqref{eq:3.6}, we have $\sum_{k\in \mathbb{Z}}\phi^{\frac{1}{2}}\left(\frac{\gamma'(2^{-j})}{2^{m+j-k}}\right)\ls 1$. This, combined with \eqref{eq:6.106}, \eqref{eq:6.107} and \eqref{eq:6.1040}, implies that Part $\mathbb{A}$ can be bounded by
\begin{align*}
\sum_{j\in \mathbb{N}}\sum_{I:\ |I|=2^{-j}}|I|d_j(5I,P_j) d_j(I,T_j)\textbf{size}^*_\textbf{1}(P\cup T)\textbf{size}_\textbf{2}(P\cup T)\ls \textbf{size}^*_\textbf{1}(P\cup T)\textbf{size}_\textbf{2}(P\cup T)|I_T|.
\end{align*}
Hence, we complete the proof of Lemma \ref{lemma 6.12}.
\end{proof}

For any $S\subset S_0$, $k\in\{1,2\}$, as in \cite[Lemma 6.12]{LX}, by \eqref{eq:6.52} in Lemma \ref{lemma 6.8}, we can always split $S$ into $S_1$ and $S_2$:
\begin{enumerate}
  \item[\rm(i)] $S_1:=\bigcup_{T\in\mathcal{F}}T$ with $\bigcup_{T\in\mathcal{F}}|I_T|\ls \frac{|F_1|}{\textbf{size}_\textbf{1}(S)^p}$ and $\bigcup_{T\in\mathcal{F}}|I_T|\ls \frac{|F_2|}{\textbf{size}_\textbf{2}(S)^q}$, where $T$ is maximal tree;
  \item[\rm(ii)] $S_2:=S\backslash S_1$ with $\textbf{size}_\textbf{1}(S_2)\leq (\frac{1}{2})^{\frac{1}{p}} \textbf{size}_\textbf{1}(S)$ and $\textbf{size}_\textbf{2}(S_2)\leq (\frac{1}{2})^{\frac{1}{q}} \textbf{size}_\textbf{2}(S)$,
\end{enumerate}
which further implies that we can write $S_0$ as
\begin{align}\label{eq:6.108}
S_0=\bigcup_{\sigma\leq 1}S_\sigma,
\end{align}
where $\sigma$ ranges over positive dyadic numbers, and $S_\sigma$ is a union of maximal trees such that for each $T\in S_\sigma$, we have
\begin{align}\label{eq:6.109}
\textbf{size}_\textbf{1}(T)\leq \sigma^{\frac{1}{p}}\left(\frac{|F_1|}{|F_3|}\right)^{\frac{1}{p}}\quad \textrm{and}\quad \textbf{size}_\textbf{2}(T)\leq \sigma^{\frac{1}{q}}\left(\frac{|F_2|}{|F_3|}\right)^{\frac{1}{q}},
\end{align}
and
\begin{align}\label{eq:6.110}
\textbf{size}_\textbf{1}(T)\geq \left(\frac{\sigma}{2}\right)^{\frac{1}{p}}\left(\frac{|F_1|}{|F_3|}\right)^{\frac{1}{p}}\quad \textrm{or}\quad \textbf{size}_\textbf{2}(T)\geq \left(\frac{\sigma}{2}\right)^{\frac{1}{q}}\left(\frac{|F_2|}{|F_3|}\right)^{\frac{1}{q}}.
\end{align}

We now turn to the proof of \eqref{eq:6.42}. Indeed, it is easy to see that
\begin{align}\label{eq:6.111}
\sum_{T\in S_\sigma}|I_T|\ls \frac{|F_1|}{\left(\left(\frac{\sigma}{2}\right)^{\frac{1}{p}}\left(\frac{|F_1|}{|F_3|}\right)^{\frac{1}{p}}\right)^p}
+\frac{|F_2|}{\left(\left(\frac{\sigma}{2}\right)^{\frac{1}{q}}\left(\frac{|F_2|}{|F_3|}\right)^{\frac{1}{q}}\right)^q}\ls \frac{|F_3|}{\sigma}.
\end{align}
Furthermore, the fact that $\Lambda_T[f,g]$ is bounded by \eqref{eq:6.44}, combined with the second part in \eqref{eq:6.50} and \eqref{eq:6.102}, gives
\begin{align}\label{eq:6.112}
\Lambda_T[f,g]\ls |I_T|\textbf{size}^*_\textbf{1}(T)\textbf{size}_\textbf{2}(T).
\end{align}
From Lemma \ref{lemma 6.12}, \eqref{eq:6.108} and \eqref{eq:6.112}, we conclude
\begin{align}\label{eq:6.113}
\Lambda_{S_0}[f,g]\ls \sum_{\sigma\leq1}\sum_{T\in S_\sigma}|I_T|\textbf{size}^*_\textbf{1}(T)\textbf{size}_\textbf{2}(T).
\end{align}
From the definition of $\textbf{size}^*_\textbf{1}(T)$, \eqref{eq:6.109} and \eqref{eq:6.111}, the above expression can be bounded by
\begin{align*}
|F_1|^{\frac{1}{p}}|F_2|^{\frac{1}{q}}|F_3|^{\frac{1}{r'}}\sum_{\sigma\leq1}\sigma^{\frac{1}{q}-1} \min\left\{2^{m(1-\frac{p}{q'})}\sigma^{\frac{1}{p}},\sigma^{\frac{1}{q'}}    \right\}  \ls m|F_1|^{\frac{1}{p}}|F_2|^{\frac{1}{q}}|F_3|^{\frac{1}{r'}}.
\end{align*}
This is the desired estimate.

\section{The boundedness of $M_{\gamma}(f,g)$}\label{section 7}

For $M^1_{\gamma}(f,g)$, as $H^1_{\gamma}(f,g)$, we write $\tilde{m}_j^1(\xi,\eta)$ as
\begin{align*}
\sum_{m,n\in \mathbb{Z_-}}\sum_{k\in \mathbb{Z}}\sum_{u,v\in \mathbb{N}}  \frac{(-i)^{u+v}2^{mu}2^{nv}}{u\mathrm{!}v\mathrm{!}} \bar{\phi}_u\left(\frac{\xi}{2^{m+j}}\right) \bar{\phi}_v\left(\frac{\eta}{2^{k}}\right) \bar{\phi}_v\left(\frac{\gamma'(2^{-j})}{2^{n+j-k}}\right)\int_{-\infty}^{\infty}t^u\left(\frac{2^j\gamma(2^{-j}t)}{\gamma'(2^{-j})}\right)^v|\rho(t)|\,\textrm{d}t.
\end{align*}
Furthermore, $M^1_{\gamma}(f,g)$ can be written as
\begin{align*}
M^1_{\gamma}(f,g)(x)=\sup_{j\in \mathbb{Z}} \Bigg|\sum_{m,n\in \mathbb{Z_-}}\sum_{k\in \mathbb{Z}}\sum_{u,v\in \mathbb{N}}&\frac{(-i)^{u+v}2^{mu}2^{nv}}{u\mathrm{!}v\mathrm{!}} \bar{\phi}_v\left(\frac{\gamma'(2^{-j})}{2^{n+j-k}}\right) \\ \times&\left(\int_{-\infty}^{\infty}t^u\left(\frac{2^j\gamma(2^{-j}t)}{\gamma'(2^{-j})}\right)^v|\rho(t)|\,\textrm{d}t\right)   \check{\bar{\phi}}_{u,m+j} \ast f(x) \cdot \check{\bar{\phi}}_{v,k} \ast g(x)\Bigg|.
\end{align*}
From \cite[P. 24 Proposition]{S}, there exists a positive constant $C$ such that $\sup_{j\in \mathbb{Z}} |\check{\bar{\phi}}_{u,m+j} \ast f |\ls C^u Mf$ and $|\check{\bar{\phi}}_{v,k} \ast g|\ls C^v Mg$. As in \eqref{eq:3.5} and \eqref{eq:3.6}, combining with the fact that $\sum_{m,n\in \mathbb{Z_-}}\sum_{u,v\in \mathbb{N}}\frac{C^u2^{mu}C^v2^{nv}}{u\mathrm{!}v\mathrm{!}}\ls 1$, we conclude that
\begin{align*}
M^1_{\gamma}(f,g)(x)\ls Mf(x)Mg(x).
\end{align*}
For $r>\frac{1}{2}$, by the $\mathrm{H}\ddot{\mathrm{o}}\mathrm{lder}$ inequality, it leads to
\begin{align}\label{eq:7.4}
\left\|M^1_{\gamma}(f,g)\right\|_{L^{r}(\mathbb{R})}\ls\|f\|_{L^{p}(\mathbb{R})}\|g\|_{L^{q}(\mathbb{R})}.
\end{align}

For $M^2_{\gamma}(f,g)$, as in $M^1_{\gamma}(f,g)$, we also have $M^2_{\gamma}(f,g)(x)\ls Mf(x)Mg(x)$. Therefore, for all $r>\frac{1}{2}$,
\begin{align}\label{eq:7.6}
\left\|M^2_{\gamma}(f,g)\right\|_{L^{r}(\mathbb{R})}\ls\|f\|_{L^{p}(\mathbb{R})}\|g\|_{L^{q}(\mathbb{R})}.
\end{align}

For $M^3_{\gamma}(f,g)$, without loss of generality, we may assume that $m=n$. Therefore, we rewrite $M^3_{\gamma}(f,g)$ as
\begin{align}\label{eq:7.7}
M^3_{\gamma}(f,g)(x)=\sup_{j\in \mathbb{Z}} \left|\int_{-\infty}^{\infty}\int_{-\infty}^{\infty}\hat{f}(\xi)\hat{g}(\eta)e^{i\xi x}e^{i\eta x} \tilde{m}_j^3(\xi,\eta)\,\textrm{d}\xi \,\textrm{d}\eta\right|,
\end{align}
where $\tilde{m}_j^3(\xi,\eta):=\sum_{m\in \mathbb{N},k\in \mathbb{Z} }\tilde{m}_{j,m,m,k}(\xi,\eta)$. Then \eqref{eq:7.7} can be bounded by
\begin{align}\label{eq:7.8}
\sum_{m\in \mathbb{N} }\sum_{j,k\in \mathbb{Z}}\left|\int_{-\infty}^{\infty}\int_{-\infty}^{\infty}\hat{f}(\xi)\hat{g}(\eta)e^{i\xi x}e^{i\eta x} \tilde{m}_{j,m,m,k}(\xi,\eta)\,\textrm{d}\xi \,\textrm{d}\eta\right|.
\end{align}
As in \eqref{eq:3.37} and \eqref{eq:3.38}, we have that there exists a positive constant $\varepsilon_0$ such that
\begin{align}\label{eq:7.9}
\left\|\sum_{j,k\in \mathbb{Z}}\left|\int_{-\infty}^{\infty}\int_{-\infty}^{\infty}\hat{f}(\xi)\hat{g}(\eta)e^{i\xi x}e^{i\eta x} \tilde{m}_{j,m,m,k}(\xi,\eta)\,\textrm{d}\xi \,\textrm{d}\eta\right|\right\|_{L^{1}(\mathbb{R})}\ls 2^{-\varepsilon_0 m}\|f\|_{L^{2}(\mathbb{R})}\|g\|_{L^{2}(\mathbb{R})}
\end{align}
holds uniformly for $j,k\in \mathbb{Z}$. As in \eqref{eq:6.3}, we have
\begin{align}\label{eq:7.10}
\left\|\sum_{j,k\in \mathbb{Z}}\left|\int_{-\infty}^{\infty}\int_{-\infty}^{\infty}\hat{f}(\xi)\hat{g}(\eta)e^{i\xi x}e^{i\eta x} \tilde{m}_{j,m,m,k}(\xi,\eta)\,\textrm{d}\xi \,\textrm{d}\eta\right|\right\|_{L^{r,\infty}(\mathbb{R})}\ls m\|f\|_{L^{p}(\mathbb{R})}\|g\|_{L^{q}(\mathbb{R})}
\end{align}
holds uniformly for $j,k\in \mathbb{Z}$. By interpolating between \eqref{eq:7.9} and \eqref{eq:7.10}, we conclude that
\begin{align}\label{eq:7.11}
\left\|M^3_{\gamma}(f,g)\right\|_{L^{r}(\mathbb{R})}\ls\|f\|_{L^{p}(\mathbb{R})}\|g\|_{L^{q}(\mathbb{R})}
\end{align}
for $r>\frac{1}{2}$. From \eqref{eq:7.4}, \eqref{eq:7.6} and \eqref{eq:7.11}, we obtain the $L^p(\mathbb{R})\times L^q(\mathbb{R})\rightarrow L^r(\mathbb{R})$ boundedness for the (sub)bilinear maximal function $M_{\gamma}(f,g)$ for $r>\frac{1}{2}$.

\bigskip

\noindent Junfeng Li

\smallskip

\noindent
School of Mathematical Sciences, Dalian University of Technology, Dalian, 116024, People's Republic of China

\smallskip

\noindent{\it E-mail}: \texttt{junfengli@dlut.edu.cn}

\bigskip

\noindent Haixia Yu (Corresponding author)

\smallskip

\noindent
Department of Mathematics, Sun Yat-sen University, Guangzhou, 510275,  People's Republic of China

\smallskip

\noindent{\it E-mail}: \texttt{yuhx26@mail.sysu.edu.cn}

\end{document}